\documentclass[oneside,english,a4wide]{amsart}
\usepackage{amssymb, graphicx, color}
\usepackage{amsfonts}
\usepackage{amsthm}
\usepackage{enumerate}
\usepackage[mathscr]{eucal}
\usepackage{graphicx}
\usepackage{dsfont}

\usepackage{accents}
\makeatletter
\def\widebar{\accentset{{\cc@style\underline{\mskip10mu}}}}
\def\Widebar{\accentset{{\cc@style\underline{\mskip8mu}}}}
\makeatother

\usepackage[pdftex, bookmarksnumbered, bookmarksopen, colorlinks, citecolor=blue, linkcolor=blue]{hyperref}
\usepackage{comment}

\allowdisplaybreaks

\newcommand{\norm}[1]{\left\Vert#1\right\Vert}

\newtheorem{Btheorem}{Theorem}

\newtheorem{thm}{Theorem}[section]
\newtheorem{defi}[thm]{Definition}
\newtheorem{cor}[thm]{Corollary}
\newtheorem{lem}[thm]{Lemma}
\newtheorem{lemma}[thm]{Lemma}

\newtheorem{prop}[thm]{Proposition}

\newtheorem{remark}[thm]{Remark}

\theoremstyle{definition}
\theoremstyle{remark}
\numberwithin{equation}{section}

\newcommand{\R}{{\mathbb R}}
\newcommand{\N}{{\mathcal N}}
\newcommand{\Z}{{\mathbb Z}}

\newcommand{\M}{{\mathcal M}}

\newcommand{\bs}{\begin{split}}
\newcommand{\es}{\end{split}}

\newcommand{\be}{\begin{eqnarray*}}
\newcommand{\ee}{\end{eqnarray*}}
\newcommand{\beq}{\begin{align}}
\newcommand{\eeq}{\end{align}}

\renewcommand{\a}{\alpha}
\renewcommand{\b}{\beta}
\newcommand{\8}{\infty}
\newcommand{\G}{G}
\newcommand{\m}{{\mu}}
\newcommand{\E}{{\mathbb{E}}}

\theoremstyle{plain}

\def\Q{\mathcal{Q}}
\def\1{\mathbf{1}}

\begin{document}
\setcounter{page}{1}
\title[Quantitative mean ergodic inequalities]
{Lamperti Operators, Dilation Theory, and Applications in Noncommutative Ergodic Theory}
\author[G. Hong]{Guixiang Hong}
\address{%
	Institute for Advanced Study in Mathematics\\
Harbin Institute of Technology\\
Harbin 150001\\
	China}

\email{gxhong@hit.edu.cn}
\author[W. Liu]{Wei Liu}
\address{School of Mathematics, Southwestern University of Finance and Economics\\
Chengdu 611130\\
China}
\email{lwei@swufe.edu.cn}
\author[S. Ray]{Samya Kumar Ray}
\address{%
	The Institute of Mathematical Sciences, Chennai, India and Homi Bhabha National
Institute, Mumbai, India \\
	}
\email{samya@imsc.res.in}
\author[B. Xu]{Bang Xu}
\address{%
School of Mathematical Sciences\\
Xiamen University\\
	Xiamen 361005\\
	China}
\email{bangxu@xmu.edu.cn}

\date{\today. 
}
\begin{abstract}
In this paper, we develop a novel framework for quantitative mean ergodic theorems in the noncommutative setting, with a focus on actions of amenable groups and semigroups. We prove square function inequalities for ergodic averages arising from actions of groups of polynomial volume growth on a fixed noncommutative $L_p$-space for $1<p<\8$. To achieve this, we establish two endpoint estimates for a noncommutative square function on non-homogeneous space. Our approach relies on semi-commutative non-homogeneous harmonic analysis, including the non-doubling Calder\'on-Zygmund arguments for non-smooth kernels and $\mathrm{BMO}$ space theory, operator-valued inequalities related to balls and cubes in groups equipped with non-doubling measures, and a noncommutative generalization of the classical transference method for amenable group actions.  As an application, we establish a quantitative ergodic theorem for the ergodic averages associated with the positive power of modulus representation arising from a Lamperti representation on noncommutative $L_p$-spaces, extending some results in \cite{Templeman2015}. To obtain quantitative ergodic theorem for semigroups of operators, in this paper, we address the open question of extending dilation theorem of Fackler-Gl\"uck from single operators to commuting tuples on Banach spaces including noncommutative $L_p$-spaces. Indeed our approach provides genuine joint $N$-dilations for commuting families, unifying and extending the classical dilation theorems of 
Sz.-Nagy--Foia\c{s} and Ak\c{c}oglu--Sucheston for a natural class of commuting tuple of contractions extending the abstract dilation theorem of of Fackler-Gl\"uck for commuting tuple of contractions. This enables us to obtain a quantitative ergodic theorem for a large class of semigroups 
of operators on $\mathbb{R}^d_{+}$.


\end{abstract}

\maketitle

\section{Introduction}



The study of ergodic theory for group actions has followed a remarkable trajectory since its inception in the 1930s. The foundational works of Birkhoff~\cite{ Birk31} and von Neumann~\cite{von-Neu32} established ergodic theorems for 1-parameter flows, laying the groundwork for Wiener's extension to multiple commuting flows and ball averages on $\mathbb{R}^d$ and $\mathbb{Z}^d$~\cite{Wiener39}. The investigation of ergodic theorems for non-abelian group actions was initiated by Calderón \cite{Cal53}, who established results for averages over increasing families of compact symmetric neighborhoods of the identity satisfying a doubling condition. Motivated by Calder\'{o}n's result, ergodic theorems for various group actions have been extensively investigated. A notable example is the work of Breuillard \cite{Breuillard14} (see also Tessera \cite{Tessera07}), which proved that sequences of balls defined by any fixed word metric of polynomial growth satisfy the doubling condition. This finding yielded the corresponding ergodic theorems for ball averages and affirmatively resolved a long-standing open question posed by Calder\'{o}n~\cite{Cal53}. In the context of amenable group actions, Lindenstrauss~\cite{Lindenstrauss01} established the pointwise ergodic theorem for averages along tempered F{\o}lner sequences. We refer the reader to the survey works~\cite{AAB+, Nevo06} for more information and other generalizations.

The deep connection between ergodic theory and von Neumann algebras, dating back to the theory of rings of operators, has motivated key results in noncommutative ergodic theory. Early work on the mean ergodic theorem for $\mathbb{Z}$-actions by  Kov\'{a}cs and Sz\"{u}cs \cite{Ko-Sz} was later expanded by Lance~\cite{Lan76} and significantly generalized by Conze and Dang Ngoc for amenable groups~\cite{CoDa78} and Yeadon for positive Dunford-Schwartz operators \cite{Ye, Yean80}. For a comprehensive overview of the mean ergodic theorem, we refer to~\cite{Ja1,Ja2}. Meanwhile, pointwise ergodic theorems in noncommutative $L_p$-spaces were first established for $p=\8$~\cite{Lan76, Kum78, CoDa78,Wal97} and $p=1$~\cite{Ye}. A major breakthrough came from Junge and Xu, who extended these results to all $1<p<\8$
~\cite{JX}, thereby inspiring a wealth of subsequent work (see, e.g.~\cite{An,Be,HS,Hu,Li}). Recent developments have moved beyond this framework. Hong, Liao and Wang \cite{HLW} proved ergodic theorems for groups with polynomial growth, while Cadilhac and Wang \cite{CW22}  established pointwise convergence for amenable groups acting along regular filtered F{\o}lner sequences using quasi-tiling and noncommutative Calder\'{o}n-Zygmund techniques.

Quite recently, Hong, Liu and Xu \cite{HLX} established the first quantitative mean ergodic theorem under the noncommutative framework, specifically addressing the cases of power-bounded invertible operators and Lamperti contractions. To properly present the main results of \cite{HLX}, we first recall some necessary notions. Let $\mathcal{M}$ be a semifinite von Neumann algebra equipped with a normal semifinite faithful trace $\tau$. For $1 \le p < \infty$, we denote by $L_p(\mathcal{M})$ the associated noncommutative $L_p$-space, and by $L_p(\mathcal{M}; \ell^{rc}_{2})$ the corresponding Hilbert-valued space (see Section~\ref{pre} for precise definitions). Consider a bounded invertible operator $T$ on $L_p(\mathcal{M})$. The associated ergodic average $M_{n}(T)$ is defined as
$$M_{n}(T)=\frac{1}{2n+1}\sum_{k=-n}^{n}T^{k},\ \ \forall\ n\in\mathbb N.$$ 
The first main result of \cite{HLX} established a quantitative mean ergodic theorem for power-bounded invertible operators, stated as follows.
\vskip0.2cm
\noindent
\textbf{Theorem A}\ {\rm (\cite{HLX}).}
Let $1<p<\infty$. Suppose that 
\begin{equation*}
	\sup_{k\in\Z}\|T^{k}:L_p(\M)\rightarrow L_p(\M)\|<\infty.
\end{equation*}
Then there exists a positive constant $C_{p}$ {depending only on $p$} such that
\begin{equation}\label{result-1}
	\sup
	\big\|\big(M_{n_{i}}(T)x-M_{n_{i+1}}(T)x\big)_{i\in \mathbb N}\big\|_{L_p(\M;\ell^{rc}_{2})}\leq C_{p}\|x\|_{L_{p}(\M)},\forall x\in L_p(\M),
\end{equation}
where the supremum is taken over all the increasing subsequence $(n_{i})_{i\in\mathbb N}$ of positive integers.
\vskip0.2cm

The second main result of \cite{HLX} established a quantitative mean ergodic theorem for Lamperti operators. Recall that a bounded operator $T$ on $L_p(\M)$  is called a Lamperti operator if for any two $\tau$-finite projections $e$ and $f$ in $\mathcal{M}$ satisfying $ef=0$, the following orthogonality condition holds
$$(Te)^*Tf=Te(Tf)^*=0.$$
The following result is given as \cite[Theorem 1.4]{HLX}.

\vskip0.2cm
\noindent
\textbf{Theorem B}\ {\rm (\cite{HLX}).}
Let $1< p<\infty$. Suppose that $T$ belong to the class
\begin{equation}\label{con-lam}
	\mathfrak{S}=\overline{\mbox{conv}}^{\mbox{sot}}\{S:L_p(\mathcal M)\rightarrow L_p(\mathcal M)~\mbox{Lamperti contractions}\},
\end{equation}	
i.e., the strong operator topology closed convex hull of all Lamperti contractions on $L_p(\mathcal M)$.	{Then} there exists a positive constant $C_{p}$ depending only on $p$ such that for every increasing sequence of positive integers $(n_{i})_{i\in\mathbb N}$,
$$\big\|\big(M_{n_{i}}(T)x-M_{n_{i+1}}(T)x\big)_{i\in \mathbb N}\big\|_{L_p(\M;\ell^{rc}_{2})}\leq C_{p}\|x\|_{L_{p}(\M)}\;\forall x\in L_p(\M).$$
\vskip0.2cm

The main purpose of this paper is to extend the study of quantitative mean ergodic theorems to the noncommutative setting, with a particular focus on actions of amenable groups and semigroups. {Throughout this work}, let $G$ denote a locally compact group equipped with a left-invariant Haar measure $m$. Recall that for a von Neumann algebra $\M$, a map $\alpha:G\rightarrow \text{Aut}(\mathcal{M})$ is called an action if for every $u\in G$, $\alpha_u:\mathcal{M}
\to \mathcal{M}$ is a $*$-preserving automorphism; and for all $s,u\in G$, $\alpha_s\circ\alpha_u=\alpha_{su}$. In addition, if $\tau\circ \alpha_u=\tau$ for all $u\in G$, we say that $\alpha$ is a $\tau$-preserving action or an action by $\tau$-preserving automorphisms, denoted by $\alpha \curvearrowright (\mathcal{M},\tau)$. As is well-known, the $\tau$-preserving action $\alpha$ extends isometrically on the noncommutative $L_p(\mathcal{M})$-spaces for all $1 \leq p \leq \infty$ (see, e.g., \cite[Lemma 1.1]{JX}). The following is one of our main results.

\begin{thm}\label{main-thm2}
	Let $G$ be a group of polynomial growth with a symmetric compact generating set $V$. Define the ergodic average
	\begin{equation}\label{er-aver}
		A_n(\alpha)x=\frac{1}{m(V^n)}\int_{V^n}\alpha_gxdm(g),~x\in L_p(\M), n\in\mathbb N.
	\end{equation}
	The following statements hold:
	\begin{enumerate}[\noindent]
		\item \emph{(i)}~Fix $1<p<\8$. Let $\a$ be a strongly continuous and uniformly bounded action of $G$ on $L_p(\M)$. Then there exists a constant $C_p>0$ such that
		\begin{equation*}
			\sup_{(n_i)_i}\|(A_{n_{i+1}}x-A_{n_i}x)_i\|_{L_p(\M;\ell^{rc}_2)}\le C_p\|x\|_{p},\ \forall x\in L_{p}(\M),
		\end{equation*}
		where the supremum is taken over all the increasing subsequence $(n_i)_i\subseteq\mathbb N$.
		\item \emph{(ii)}~Let $\a$ be a strongly continuous action of $G$ on $\M$ by $\tau$-preserving automorphisms of $\M$. Then there exists a constant $c_p>0$ such that
		$$\sup_{(n_i)_i}\|(A_{n_{i+1}}x-A_{n_i}x)_i\|_{L_{p,\infty}(\M;
			\ell_{2}^{rc})}\leq c_{p} \|x\|_{p}, \; \forall x\in L_{p}(\M),$$
		with the supremum taken as above.
	\end{enumerate}
\end{thm}
\begin{remark}\rm
When $G=\Z$, $V=\{-1,1\}$ and $A_n=\frac1{2n+1}\sum^n_{k=-n} T^k$, part (i) of Theorem \ref{main-thm2} recovers  {\textbf{Theorem A}}. Thus, our result can be viewed as a quantitative mean ergodic theorem for bounded representations of polynomial growth groups on noncommutative $L_p(\M)$-spaces. This result is motivated by the foundational contributions of Hong, Liao and Wang \cite{HLW} on noncommutative maximal ergodic theorems for group actions. Moreover, we emphasize that Theorem \ref{main-thm2} applies not only to groups of polynomial growth but also extends to more general geometric groups, including homogeneous groups such as $\mathbb{R}^n$ and the Heisenberg group $\mathbb{H}^n$ when equipped with homogeneous metrics.
\end{remark}

Our second main result establishes a quantitative mean ergodic theorem for Lamperti representations. Recall that a strongly continuous uniformly bounded action $\alpha$ of $G$ on $L_p(\mathcal{M})$ is called a Lamperti representation if $\alpha_g$ is a Lamperti operator for every $g \in G$. Building upon the canonical decomposition theorem for such operators (Proposition \ref{moduluslamp}), we construct an associated modulus representation $|\alpha|$ and prove the following result, which generalizes Templeman's classical work \cite{Templeman2015} to the noncommutative setting and provides a substantial strengthening of \textbf{Theorem B} in this structured setting which can be derived for $p=\gamma.$
\begin{thm}\label{lamprepmainres1}
	Let $\mathcal{M}$ be a finite von Neumann algebra and $G$  be a group of polynomial growth with a symmetric compact generating set $V.$  Suppose that \(\alpha\) is a strongly continuous and uniformly bounded Lamperti representation of \(G\) on \(L_p(\mathcal M)\), where \(1 < p < \infty\). Let $A_n(\alpha)$ be defined in (\ref{er-aver}). Then for any $1\leq\gamma<\infty$ and $\mu=\frac{p}{\gamma}$, there exists \(C_{p,\gamma} > 0\) such that
	\[
	\sup_{(n_i)_i} 
	\Big\| \big( A_{n_{i+1}}(|\alpha|^{(\mu)})x - A_{n_i}(|\alpha|^{(\mu)})x \big)_i \Big\|_{L_\gamma(\mathcal M; \ell_2^{rc})} 
	\le C_{p,\gamma} \, \|x\|_\gamma, 
	\quad \forall x \in L_\gamma(\mathcal M).
	\]
	where the supremum is taken over all increasing sequences \((n_i)_i \subset \mathbb{N}\).
\end{thm}

We now turn our attention to semigroups of Lamperti operators. Let $1\leq p<\infty.$ Denote $\mathbb{R}^d_{+}:=\{(t_1,\dots,t_d):t_i\geq 0,\ 1\leq i\leq d\}.$ Let  $\alpha$ be a strongly continuous uniformly bounded action of $\mathbb{R}^d_{+}$ on $L_p(\mathcal{M})$. Define for any $t>0,$ \[
A_{t}(\alpha) x: = \frac{1}{t^d} \int_{0}^{t}\dots\int_{0}^{t} \alpha_{(s_1,\dots,s_d)}(x) \, ds_1\dots ds_d, \quad x \in L_p(\mathcal M).\]
\begin{thm}\label{semiofR6d1} Let $1<p\neq 2<\infty$ and $\alpha:\mathbb{R}^d_{+}\to \mathcal{L}(L_p(\mathcal M))$ be a strongly continuous semigroup of contractions such that for all $k\in\mathbb N,$ the commuting tuple of contractions $\mathbf{T}_{\frac{1}{\mathbf{k}}}:=(\alpha_{(\frac{1}{k},\dots,0)},\dots,\alpha_{(0,\dots,\frac{1}{k}}))$ admits a joint dilation to some other noncommutative $L_p$-space, then 
	\begin{equation}\label{lampertisemigroup}
		\sup_{(t_i)_i} 
		\Big\| \big( A_{t_{i+1}}(\alpha)x - A_{t_{i}}(\alpha)x \big)_i \Big\|_{L_p(\mathcal M; \ell_2^{rc})} 
		\le C \, \|x\|_p, 
		\quad \forall x \in L_p(\mathcal M).
	\end{equation}
	where the supremum being taken over all increasing sequences \((t_i)_{i \in \mathbb{N}} \subset (0, \infty)\)
	
\end{thm}

Recall that in \cite[Theorem~1.4]{HLX} (\textbf{Theorem B}), a quantitative ergodic theorem was established 
for the \emph{absolutely convex hull} of Lamperti contractions. 
A fundamental technique in proving such quantitative ergodic inequalities 
was the \emph{dilation} of this class of contractions to isometries, 
as achieved earlier in \cite{HongRayWang2023}. For general semigroups of contractions, however, such dilation theorems are not available, 
even in the setting of Hilbert spaces. The best known result in this direction concerns semigroups of contractions over 
\emph{ordered abelian semigroups}, established in \cite{ML65}. 
By the classical dilation theorems of Sz.-Nagy--Foia\c{s} and And\"o~\cite{Nagyfoias, Ando63}, one can dilate semigroups of contractions indexed by $\mathbb{N}$ or $\mathbb{N}^2$ to semigroups of isometries. However, this fails already in three variables, as shown in \cite{VA74, VA76} (see also \cite{GuR18} for numerous counterexamples). Establishing dilation theorems for various classes of commuting tuples of contractions 
remains an active area of research in operator theory. 
We refer to the survey of Shalit~\cite{Shalit21} for a comprehensive overview 
and several open problems related to this theme. In the setting of contractions on classical $L_p$-spaces for $1 < p < \infty$, 
the most important dilation theorem is due to Ak\c{c}oglu and Sucheston~\cite{AkcogluSucheston1977}, 
which, in turn, yields a maximal ergodic theorem for this class of contractions. To the best of our knowledge, no general dilation theorem exists for tuples of contractions 
on $L_p$-spaces with $1 < p < \infty$, except for certain \emph{loose dilation} results such as those in \cite{Monpalray, MoRa}. We refer to \cite{kan78erglamperti, fendler97dilation,fendler98dilation,peller76vn,peller76poly,peller81vn,
	peller85polyschatten} and references therein for more on dilation on classical $L_p$-spaces. The situation for noncommutative $L_p$-spaces is even more delicate: 
as shown by Junge and Le~Merdy~\cite{JUL07}, there is no reasonable analogue 
of the Ak\c{c}oglu--Sucheston dilation theorem in this setting. The very first dilation theorem for contractions on noncommuttaive $L_p$-spaces for a single $1<p<\infty$ was constructed in \cite{HongRayWang2023}. In a remarkable paper, Fackler and Gl\"uck~\cite{FacklerGluck2019} proved several 
dilation theorems for general Banach spaces. Their results unify both the Sz.-Nagy--Foia\c{s} and Ak\c{c}oglu--Sucheston dilation theorems under a single abstract framework. However, their results are restricted to the single-variable case and do not yield any dilation theorems for commuting tuples. 
Indeed, as the authors themselves note \cite[Section 8]{FacklerGluck2019}, ``\emph{...we do not know whether commutative dilation theorems can be derived from our simultaneous dilation results.}'' 
In the present paper, we extend their main theorem to the setting of commuting tuples. In particular, we establish the following result.
\begin{Btheorem}\label{thm:convex-dilation}
	Fix $p \in (1,\infty)$ and let $\mathcal{X}$ be a class of Banach spaces stable under 
	finite $\ell_p$-sums. 
	Let $X \in \mathcal{X}$, and let $(\mathcal{T}_1,\ldots,\mathcal{T}_n)$ be a tuple of 
	pairwise commuting families of bounded linear operators on $X$ which admit a 
	joint simultaneous $N$-dilation in $\mathcal{X}$. 
	Then every tuple of operators $\mathbf{T} = (T_1,\dots,T_n)$, 
	where each $T_i$ belongs to the convex hull of $\mathcal{T}_i$ for $1 \leq i \leq n$, 
	also admits a joint $N$-dilation in $\mathcal{X}$.
\end{Btheorem}
We refer to Section~\ref{dilationfamily} for unexplained notation. 
As an immediate corollary, we obtain the following: Let $1 < p < \infty$, and suppose that for each $i = 1, \ldots, n$, 
$\mathcal{U}_i$ is a family of pairwise commuting \emph{isometries} acting on 
a noncommutative $L_p$-space (or Hilbert space) $L_p(\mathcal{M})$. 
If $\mathbf{T} = (T_1, \ldots, T_n)$ is an $n$-tuple of operators such that 
each $T_i$ belongs to the \emph{convex hull} of $\mathcal{U}_i$, 
then $\mathbf{T}$ admits a \emph{joint $N$-dilation} on a larger noncommutative 
$L_p$-space (or Hilbert space) for every $N \in \mathbb{N}$. This theorem appears to be new even in the context of Hilbert spaces 
and should be of independent interest to operator theorists. As an application we prove a quantitative ergodic theorem for semigroup of operators on $\mathbb{R}^d_{+}$ in Theorem \ref{semiofR6d1}.

Our approach to the proof of Theorem \ref{main-thm2} is the semi-commutative harmonic analysis {and} noncommutative Calder\'on transference techniques developed in \cite{HLW,HLX}. To that end, we set some notation. Let $d$ be an invariant metric on $G$. For $r>0$ and $h\in G$,  define the ball $B(h,r)=\{g\in G:d(g,h)\le r\}$, and {we will write $B_r$ for simply whenever $h=e$, where $e$ is the identity of $G$.}  Let $\mathcal{S}_{\mathcal M}$ be the collection of operators with $\tau$-finite support in $\mathcal M$. {For} $r>0$ and $f:G\rightarrow \mathcal{S}_{\mathcal M}$, {the Hardy-Littlewood average operator is defined by}
\begin{equation}\label{averaging-op}
	M_rf(h)=\frac{1}{m(B(h,r))}\int_{B(h,r)}f(g)dm(g).
\end{equation}
{We also recall} that $G$ is called amenable if it admits a F${\o}$lner sequence $(F_n)_{n\in \mathbb N}$, that is, for every $g\in G$,
\begin{align}\label{asymptotically invariant}
	\lim_{n\rightarrow\infty} \frac{m((F_n g) \bigtriangleup F_n)}{m(F_n)}= 0,
\end{align}
where $\bigtriangleup$ denotes the usual symmetric difference of two sets. Our second result is a transference principle stated as follows. In the suquel, we denote $\mathcal{N}=L_\infty(\G)\overline{\otimes}\M$ equipped with the tensor trace $\varphi=\int_{G}\otimes \tau$. 
\begin{thm}\label{thm:trans}
	{Let $G$ be an amenable group equipped with an invariant metric $d$}. For $1\leq p<\8$ and $x\in L_p(\M)$, define the ergodic average $A_r$ as
	\begin{equation*}
		A_rx=\frac{1}{m(B_r)}\int_{B_r}\a_gxdm(g),~r>0.
	\end{equation*}
	\begin{enumerate}[\noindent]
		\item \emph{(i)}~Let $1<p<\8$ {and let} $\a$ be a strongly continuous and uniformly bounded action of $G$ on $L_p(\M)$. Assume that there exists a positive constant $C_p$ such that
		\begin{equation*}
			\|(M_{r_i}f-M_{r_{i+1}}f)_{i}\|_{L_p(\mathcal N;\ell_2^{rc})}\le C_p\|f\|_p,~\forall f\in L_p(\mathcal N).
		\end{equation*}
		Then there exists {an absolute} constant $C>0$ such that
		\begin{equation*}
			\|(A_{r_i}x-A_{r_{i+1}}x)_{i}\|_{L_p(\M;\ell_2^{rc})}\le CC_p\|x\|_p,~\forall x\in L_p(\M).
		\end{equation*}
		\item \emph{(ii)}~Let $1\le p<\8$ {and let} $\a$ a strongly continuous action of $G$ on $\M$ induced by $\tau$-preserving automorphisms of $\M$.
		Assume that  there exists a positive constant $C'_p$ such that
		\begin{equation*}
			\|(M_{r_i}f-M_{r_{i+1}}f)_{i}\|_{L_{p,\infty}(\mathcal N;\ell_2^{rc})}\le C'_p\|f\|_p,~\forall f\in L_p(\mathcal N),
		\end{equation*}
		Then there exists {an absolute} constant $C'>0$ such that
		$$\|(A_{r_{i+1}}x-A_{r_i}x)_i\|_{L_{p,\infty}(\M;
			\ell_{2}^{rc})}\leq C'C'_{p} \|x\|_{p}, \; \forall x\in L_{p}(\M).$$
	\end{enumerate}
\end{thm}
\begin{remark}\rm
If $G=\Z$, then Theorem \ref{thm:trans}(i) recovers~\cite[Proposition 1.6]{HLX}. Moreover, when $\M$ is a commutative algebra, Theorem \ref{thm:trans}(ii) seems new even in the commutative case. This transfer principle differs subtly from those in \cite{HLW} and \cite{HLX}, as it originates from Nevo's foundational work on ergodic theorems for group actions~\cite{Nevo06}.

On the other hand, it is well-known that if $G$ is a group with polynomial volume growth, then it is amenable (cf. \cite{Gui}).  In particular, the family of balls $\{B_r\}_{r>0}$ generated by any word metric on $G$ forms a F${\o}$lner sequence (cf. \cite{Breuillard14,Tessera07}).
\end{remark}

\medskip

{With Theorem \ref{thm:trans} in hand, we are reduced to showing the strong type $(p,p)$ and weak type $(1,1)$ estimates of the corresponding operator-valued square function.} Interestingly, {we may study the boundedness theory of such square function} in a much more general setting. To see this, we call $(G,d)$ satisfies the geometric doubling property if there exists a integer $D>0$ such that for every ball $B(h,r)$ can be covered by at most $D$ balls $B(h_i,r/2)$, i.e.,
\begin{equation}\label{geo-doubling}
	B(h,r)\subseteq \bigcup_{1\le i\le{D}}B(h_i,r/2).
\end{equation}
Let $r_0>0$ and $\epsilon\in(0,1]$. {The pair $(G,d,m)$ is said to satisfy} the $(\epsilon,r_0)$-annular decay property if there exists a constant $K>0$ such that for all $h\in G$, $r\in (r_0,\infty)$ and $s\in (0,r]$,
\begin{equation}\label{decay property}
	m(B(h,r+s))-m(B(h,r))\le K\bigg(\frac{s}{r}\bigg)^{\epsilon}m(B(h,r)).
\end{equation}

We now present the third conclusion and the reader is referred to Section~\ref{BMO} for the definition of $\mathrm{BMO}$ spaces.
\begin{thm}\label{main-thm1}
	{Assume that $(\G,d,m)$ satisfies}~\eqref{geo-doubling} and \eqref{decay property}. Let $M_r$ be an average operator defined by~\eqref{averaging-op}. Let $(r_{i})_{i\in\mathbb N}\subseteq(r_0,\8)$ an increasing sequence and set $T_i=M_{r_{i+1}}-M_{r_i}$ for each $i$. Then for $1\le p\le \8$ the following assertions hold with a positive constant $C_p$ independent of $T_i$ and $f$:
	\begin{enumerate}[\noindent]
		\item \emph{(i)}~for $p=1$,
		$$\| (T_{i}f)_{i\in\mathbb N}\|_{L_{1,\8}(\mathcal{N};
			\ell_{2}^{rc})}\leq C_{p} \|f\|_{1},~\forall f\in L_{1}(\mathcal N);$$
		\item\emph{(ii)}~for $p=\8$,
		$$\Big\|\sum_{i:i\in\mathbb N} T_i \hskip-1pt f \otimes e_{1i}
		\Big\|_{\mathrm{BMO}(\mathcal{A})} +\Big\|\sum_{i:i\in\mathbb N} T_i \hskip-1pt f \otimes e_{i1}\Big\|_{\mathrm{BMO}(\mathcal{A})} \leq C_p
		\|f\|_\infty,\; \forall f\in L_{\infty}(\mathcal N);$$
		\item \emph{(iii)}~for $1<p<\8$,
		$$\| (T_{i}f)_{i\in\N}\|_{L_p(\mathcal{N};
			\ell_{2}^{rc})}\leq C_{p} \|f\|_{p}, \; \forall f\in L_{p}(\mathcal N),$$
	\end{enumerate}
	where  $\mathcal{A}=\mathcal{N}\overline{\otimes}\mathcal{B}(\ell_2)$ with the tensor trace $\varphi\otimes tr$ and $tr$ is the canonical trace on $\mathcal{B}(\ell_2)$.
\end{thm}
Here, $\mathrm{BMO}(\mathcal{A})$ denotes the dyadic $\mathrm{BMO}$ space, which is defined in Subsection~\ref{BMO-est}. The group structure of $G$ is not essential here; in fact, Theorem~\ref{main-thm1} extends to general metric measure spaces. Moreover, Theorem~\ref{main-thm1} remains valid for any regular Borel measure when $1\le p\le 2$ (see Section~\ref{weak-type}). By contrast, the conditions \eqref{decay property} are intrinsically tied to the geometric structure of $G$. Canonical examples include groups of polynomial growth equipped with the word metric and length spaces. See, e.g., \cite{Hong-Liu21} for more details.
 
\bigskip

\begin{remark}\label{app}
		We remark that the infinite summations appeared in all the aforementioned theorems should be understood as a result of the corresponding uniform boundedness for all finite summations by the standard approximation arguments (see e.g. \cite[Section 6.A]{JMX}). So, as in \cite{HX,HLX}, we will not explain the convergence of infinite sums appearing in the paper if there is no ambiguity.
\end{remark}

An outline of this paper is as follows. In Section \ref{pre}, we review the definition of noncommutative $L_p$-spaces and Hilbert-valued $L_p$-spaces. We then introduce the `dyadic cubes' constructed by Hyt\"{o}nen and Kairema~\cite{Hyt-Anna12} in metric space and present some technical lemmas. This section also gives the content of noncommutative Calder{\'o}n-Zygmund decomposition for non-doubling measure recently obtained in \cite{Cad-Alo-Par21}. We end this section by discussing the annular decay property and providing some examples which satisfy conditions \eqref{geo-doubling} and \eqref{decay property}. In particular, we verify that the polynomial growth group equipped with a word metric satisfies the doubling measure property, ($\epsilon$,1)-annular decay property and \eqref{asymptotically invariant}. As a consequence, Theorem \ref{main-thm2} holds by virtue of Theorem \ref{thm:trans} and \ref{main-thm1}.  Section \ref{weak-type} is devoted to the proof of weak type $(1,1)$ estimates announced in Theorem \ref{long-ineq} and Theorem  \ref{short-ineq}, while the proof of BMO estimates of  Theorem \ref{long-ineq} and \ref{short-ineq} is given in Section \ref{BMO}.  In section \ref{thisissec5}, we  prove Theorem \ref{thm:trans}. In Section~\ref{laprepsec}, we prove quantitative ergodic theorems for positive powers of the modulus of Lamperti representations. 
Section~\ref{dilationfamily} is devoted to the proof of our dilation theorem. 
Finally, in Section~\ref{quanergforopsemi}, we illustrate how quantitative ergodic theorems for semigroups over $\mathbb{R}^d_{+}$ can be deduced via discretization.

 Throughout this paper, let $C$ denote a positive constant that may vary from line to line, while $c_p$ and $C_p$ denote a positive constant possibly depending on the subscripts. Also, let the notation $X\lesssim Y$ mean that $X\le CY$ for some inessential constant $C>0$ and $X\thickapprox Y$ mean that both $X\lesssim Y$ and $Y \lesssim X$ hold.

After completing a preliminary version of this paper, we learned that some intermediate results leading to part (iii) of Theorem \ref{main-thm1} has been independently obtained by \cite{BikramSaha2025}.


\section{Preliminaries and some technical lemmas}\label{pre}
\subsection{Noncommutative $L_p$-spaces}\quad

\medskip

Throughout this paper, let $\mathcal M$ be a von Neumann algebra equipped with a normal semifinite faithful (\textit{n.s.f.}) trace $\tau$ and $\mathcal M_+$ be the positive part of $\mathcal M$. Let $\1_{\mathcal M}$ denote the unit element in $\M$. Given $x\in \mathcal M_+$, the support of $x$, denoted by $s(x)$, is defined to be the smallest projection $a$ in $\mathcal M$ such that $xa=x=ax$. We denote by $\mathcal{S}_{\mathcal M_+}$ the set of all $x\in\mathcal M_+$ such that $\tau(s(x))<\8$, and by $\mathcal{S}_{\mathcal M}$ the linear span of $\mathcal{S}_{\mathcal M_+}$. Then $\mathcal{S}_{\M}$ is a $w^{*}$-dense $\ast$-subalgebra of $\M$. Let $1\le p<\8$, the noncommutative $L_p$-space $L_p(\M,\tau)$ is defined by the completion of $\mathcal{S}_{\M}$ with respect to the norm
\begin{equation*}
  \|x\|_p=\tau(|x|^p)^{\frac{1}{p}},~\qquad~\mbox{where}~|x|=(x^*x)^{\frac{1}{2}}.
\end{equation*}
For $p=\8$, we set $L_{\8}(\mathcal M,\tau)=\mathcal M$ equipped with the operator norm.

Let $L_{0}(\M)$ be the set of the $\ast$-algebra of \emph{$\tau$-measurable} operators. The noncommutative weak $L_p$-space $L_{p,\8}(\mathcal M)$ is defined as the set of all $x\in L_0(\mathcal M)$ equipped the following finite quasi-norm
\begin{equation*}
 \|x\|_{p,\8}=\sup_{\lambda>0}\lambda\tau(\chi_{(\lambda,\8)}(|x|))^{\frac{1}{p}}.
\end{equation*}
In particular, for weak $L_1$-space, one has the basic fact (cf.~\cite{Su12}): for any $x_1,x_2\in L_{1,\8}(\mathcal M)$ and $\lambda>0$, 
\begin{equation}\label{distrib}
  \tau(\chi_{(\lambda,\8)}(|x_1+x_2|))\leq\tau(\chi_{(\lambda/2,\8)}(|x_1|))+\tau(\chi_{(\lambda/2,\8)}(|x_2|)).
\end{equation}

For more explanation of noncommutative $L_p$-spaces and weak $L_p$-spaces we refer to \cite{Pis-Xu03,Fack-Kos86}.

\subsection{Noncommutative Hilbert-valued $L_p$-spaces}\quad

\medskip

Let $1\leq p<\infty$ and $(x_{n})$ be a finite sequence in $L_{p}(\mathcal M)$. Define
$$\|(x_{n})\|_{L_{p}(\mathcal M; \ell_{2}^{r})}=\Big\|\Big(\sum_{n}|x^{\ast}_{n}|^{2}\Big)^{\frac{1}{2}}\Big\|_{p},\quad\|(x_{n})\|_{L_{p}(\mathcal M; \ell_{2}^{c})}=\Big\|\Big(\sum_{n}|x_{n}|^{2}\Big)^{\frac{1}{2}}\Big\|_{p}.$$
Let $L_{p}(\mathcal M; \ell_{2}^{r})$ (resp. ${L_{p}(\mathcal M; \ell_{2}^{c})}$) be the completion of all finite sequences in $L_p(\mathcal M)$ with respect to $\|\cdot\|_{{L_{p}(\mathcal M; \ell_{2}^{r})}}$ (resp. $\|\cdot\|_{{L_{p}(\mathcal M; \ell_{2}^{c})}}$). The mixed space $L_p(\mathcal{M};\ell_{2}^{rc})$ is defined as follows:
\begin{itemize}
	\item If $1\leq p<2$, let
	$$L_p(\mathcal{M};
	\ell_{2}^{rc})=L_{p}(\mathcal M; \ell_{2}^{r})+L_{p}(\mathcal M; \ell_{2}^{c})$$
	equipped with the sum norm
	$$\|(x_{n})\|_{L_p(\mathcal{M};
		\ell_{2}^{rc})}=\inf\{\|(y_{n})\|_{L_{p}(\mathcal M; \ell_{2}^{r})}+\|(z_{n})\|_{L_{p}(\mathcal M; \ell_{2}^{c})}\},$$
	where the infimun runs over all decompositions $x_{n}=y_{n}+z_{n}$ in $L_{p}(\mathcal{M})$.
\item If $2\leq p\leq\infty$, let
	$$L_p(\mathcal{M};
	\ell_{2}^{rc})=L_{p}(\mathcal M; \ell_{2}^{r})\cap L_{p}(\mathcal M; \ell_{2}^{c})$$
	equipped with the intersection norm
	$$\|(x_{n})\|_{L_p(\mathcal{M};
		\ell_{2}^{rc})}=\max\{\|(x_{n})\|_{L_{p}(\mathcal M; \ell_{2}^{r})},\|(x_{n})\|_{L_{p}(\mathcal M; \ell_{2}^{c})}\}.$$
\end{itemize}
Clearly, $L_2(\mathcal{M}; \ell_{2}^{r}) = L_2(\mathcal{M};\ell_{2}^{c})=L_2(\mathcal{M};\ell_{2}^{rc})$. Replacing the $L_p$-norm with the weak $L_p$-norm, one can also use the above procedure to define the noncommutative Hilbert-valued weak $L_p$-spaces.
 For instance, the $L_{1,\8}(\M;\ell_2^{rc})$ is defined by
\begin{equation*}
  \|(x_n)\|_{L_{1,\8}(\M;\ell_2^{rc})}=\inf_{x_n=y_n+z_n}\{\|(y_n)\|_{L_{1,\8}(\M;\ell_2^{r})}+\|(z_n)\|_{L_{1,\8}(\M;\ell_2^{c})}\}.
\end{equation*}
We refer the reader to Pisier-Xu's book \cite{Pis-Xu03} for the properties of such spaces. In particular, the following noncommutative Khintchine inequalities will be frequently used, see~\cite{LP86,LP-Pi91,Pis98, C1} for the proof.

\begin{prop}\label{nonkin}
Let $(\varepsilon_n)$ be a Rademarcher sequence on probability space $(\Omega, P)$. Let $1\le p<\8$ and $(x_n)$ be a sequence in $L_p(\M;\ell^{rc}_2)$. Then there exist two positive constants $c_p$ and $c^\prime_p$ such that
\begin{equation*}
  c_p\|(x_n)\|_{L_p(\mathcal M;\ell^{rc}_2)}\le \bigg\|\sum_{n}\varepsilon_nx_n\bigg\|_{L_p(L_{\infty}(\Omega)\overline{\otimes}
\mathcal M))}\le  c^\prime_p\|(x_n)\|_{L_p(\mathcal M;\ell^{rc}_2)}.
\end{equation*}
The above estimate also holds if one replaces the $L_p$-spaces by  $L_{1,\infty}$-space.
\end{prop}

We also need another class of non-commutative Hilbert-valued $L_p$-spaces. Let $(\Sigma, \mu)$ be a $\sigma$-finite measurable space and $0 < p \leq \infty$. The column space $L_p(\M; L^{c}_2(\Sigma))$ consists of operator-valued functions $f$ for which the norm (or $p$-norm, when $0 < p < 1$) is finite. The norm is defined as follows
\begin{equation*}
    \|f\|_{L_p(\M;L^{c}_2(\Sigma))} = \left\|\int_\Sigma f^*(w)f(w) d\mu(w)\right\|_p.
\end{equation*}
The key property for our purpose is the following H\"{o}lder's inequality (see, e.g., \cite[Proposition 1.1]{M}).
\begin{prop}\label{Holder-ineq}
Let $0 \leq p, q, r \leq \infty$ satisfy $\frac{1}{r} = \frac{1}{p} + \frac{1}{q}$. Then for any $f \in L_p(\mathcal{M}; L^{c}_2(\Sigma))$ and $g \in L_q(\mathcal{M}; L^{c}_2(\Sigma))$, the following inequality holds
\begin{equation*}
\left\| \int_\Sigma f^*(w)g(w)\,d\mu(w) \right\|_r \leq \left\| \left( \int_\Sigma |f(w)|^2\,d\mu(w) \right)^{\frac{1}{2}} \right\|_p \left\| \left( \int_\Sigma |g(w)|^2\,d\mu(w) \right)^{\frac{1}{2}} \right\|_q.
\end{equation*}
\end{prop}

\subsection{`Dyadic cubes' and boundary properties}\quad

\medskip

In this subsection, we revisit the `dyadic cubes' constructed on the metric space $(\G,d,m)$ that satisfy the geometric doubling condition~\eqref{geo-doubling} and the $(\epsilon,r_0)$-annular decay property~\eqref{decay property}. Note that $(\G,d,m)$ may not necessarily be a measure doubling space. As a result, the 'dyadic cubes' lose the boundary property (cf. \cite[Theorem 11 (3.6)]{Christ90}) of Christ-type dyadic cubes constructed in doubling metric measure spaces. By leveraging the refined lemmas from~\cite{Hong-Liu21}, which address boundary conditions in dyadic cubes, balls, and their specific configurations, we derive key technical lemmas to establish Theorems~\ref{long-ineq} and~\ref{short-ineq}. These lemmas could be regarded as results in vector-valued non-homogeneous harmonic analysis with the underlying Banach spaces being noncommutative $L_p$ spaces. 
Since the estimates are delicate, we will fix throughout the paper several constants---including $k_1,n_0,n_1, c_0, C_0, C_1, L_0,L_1,\delta$---that depend on conditions~\eqref{geo-doubling},~\eqref{decay property}, and the construction.

As is well-known that the geometrically doubling property extends the doubling measure property (cf. ~\cite{Coifman-Weiss71}). Consequently, the metric space $(G,d,m)$ satisfying conditions \eqref{geo-doubling} and \eqref{decay property} may constitute a non-doubling measure space. However, a straightforward calculation reveals that \eqref{decay property} implies the following property: For any $h\in \G$ and $r_0<r\le R<\8$,
\begin{align}\label{int}
\frac{m(B(h,R))}{m(B(h,r))}\lesssim\Big(\frac{R}{r}\Big)^{\epsilon}.
\end{align}
This estimation will be applied subsequently. Furthermore, the geometrically doubling condition immediately implies the following property.
\begin{prop}\label{geometry-doubling}
Let $(G,d)$ satisfy the geometrically doubling property, namely condition~\eqref{geo-doubling}. For any ball $B(h,R)$ with radius $R > 0$ and arbitrary $0 < r \leq R$, there exists a finite covering of $B(h,R)$ by at most $D^{[ \log_2(R/r)] + 1}$ balls of radius $r$.
\end{prop}

The following system of `dyadic cubes' in geometrically doubling space was constructed in \cite[Theorem 2.2]{Hyt-Anna12}.
\begin{prop}\label{dyadic cube}
Let $(\G,d)$ satisfy the geometrically doubling property, namely condition~\eqref{geo-doubling}. Fix constants $0<c_0<C_0<\8$ and $\delta>1$ such that
$$18C_0\delta^{-1}\le c_0.$$
Let $I_k$ be an index set for every $k\in\mathbb{Z}$ and $\{z_\alpha^k\in \G:\alpha\in I_k,k\in\Z\}$ an associated collection of points with the properties that
\begin{equation}\label{distance}
  d(z_\alpha^k,z_\beta^k)\ge c_0\delta^k~(\alpha\neq\beta),~\min_{\alpha}d(w,z_\alpha^k)<C_0\delta^k,~\forall~w\in \G,~k\in\mathbb{Z}.
\end{equation}
Set $a_0:=c_0/3$ and $C_1:=2C_0$. Then there exists a sets $\big\{Q_\alpha^k\big\}_{\alpha\in I_k}$, associating with the index set $\{z_\alpha^k\}_{\alpha\in I_k}$, satisfies the following  properties:
\begin{enumerate}[\noindent]
\item\emph{(i)}~$\forall~k\in\mathbb{Z}$, $\cup_{\alpha\in I_k} Q_\alpha^k=\G$;
\item \emph{(ii)}~if $k\ge l$ then either $Q_\alpha^l\subset Q_\beta^k$ or $Q_\alpha^l\cap Q_\beta^k=\emptyset$;
\item \emph{(iii)}~for each $(k,\alpha)$ and each $n>k$, there exists a unique $Q_\beta^n$ such that $Q_\alpha^k\subset Q_\beta^n$, and for $n=k+1$,  the element $Q_\beta^{k+1}$ 
 is designated as the parent of $Q_\alpha^{k}$, denoted by the symbol  $\widehat{Q}_\alpha^{k}$.
\item \emph{(iv)}~$B(z_\alpha^k, a_0\delta^k)\subseteq Q_\alpha^k\subseteq B(z_\alpha^k, C_1\delta^k)$. 
\end{enumerate}
\end{prop}
We emphasize that the geometric doubling property ensures the validity of the second inequality in~\eqref{distance}, while the small boundary property (see \cite[Theorem 11 (3.6)]{Christ90}) of dyadic cubes constructed via Proposition~\ref{dyadic cube} does not hold universally. Furthermore, Hong and Liu~\cite{Hong-Liu21} established distinct boundary properties of these cubes. To achieve our objectives, we synthesize the following four propositions which are established in \cite[Lemmas 2.6--2.7 and 2.9--2.10]{Hong-Liu21}.  

\begin{prop}\label{boundary}
	Fix $0<c_0<C_0<\8$ and $\delta>1$ in Proposition~\ref{dyadic cube}. Define
	\begin{align*}
		&L_0=[\log_\delta(12/c_0)]+1, L_1=[\log_\delta(36r_0/c_0)]+1.
	\end{align*}
	Let $k,L\in\mathbb{Z}$ satisfy $L_0<L<k+L_0-L_1$ and $\alpha\in I_k$. There exists a constant $\eta>0$ independent of $k$, $L$ and `dyadic cubes' $Q_\alpha^k$ such that
	\begin{equation*}
		\begin{split}
			&m\big(\{g\in Q_\alpha^k:d(h,\G\setminus Q_\alpha^k)\le \delta^{k-L}\}\big)\lesssim\delta^{-L\eta}m(Q_\alpha^k),\\
			&m\big(\{g\in \G\setminus Q_\alpha^k:d(h,Q_\alpha^k)\le \delta^{k-L}\}\big)\lesssim\delta^{-L\eta}m(Q_\alpha^k).
		\end{split}
	\end{equation*}
\end{prop}

\begin{prop}\label{measure estimate of cube}
	Given a ball $B_{r}$ and a `dyadic cube' $Q^{k}_\alpha$, define
	$$\mathcal{H}(B_{r},Q^{k}_\alpha)=\{h\in Q_\alpha^{k}:B(h,r)\cap (Q^{k}_\alpha)^{c}\neq \emptyset\}.$$
	Set $n_0=\max\{L_1-L_0,0\}$. Then for every $n>n_0$, $k>L_0$ and $Q_\alpha^{n+k}$, we have
	\begin{equation*}
		m(\mathcal{H}(B_{\delta^{n}},Q^{n+k}_\alpha))\lesssim\delta^{-k\eta} m(Q^{n+k}_\alpha).
	\end{equation*}
\end{prop}
\begin{prop}\label{lem:decay}
	Let $r\ge2r_0$ and $s\in(0,r]$, then
	\begin{equation*}
		\begin{split}
			m(B(h, r+s)\setminus B(h,r-s))\lesssim\bigg(\frac{s}{r}\bigg)^{\epsilon}m(B(h,r)).
		\end{split}
	\end{equation*}
\end{prop}
\begin{prop}\label{measure estimate of annulus}
	Let $A$ be a measurable set on $\G$, define
	\begin{align*}
		\mathcal{I}_1(A,k)=\bigcup_{{\begin{subarray}{c}
					\a\in I_k\\ \partial A\cap Q_\a^k\neq \emptyset
		\end{subarray}}} Q_\a^k.
	\end{align*}
	and
	\begin{align*}
		\mathcal{I}(A,k)=\bigcup_{{\begin{subarray}{c}
					\a\in I_k\\ \partial A\cap Q_\a^k\neq \emptyset
		\end{subarray}}} Q_\a^k\cap A,
	\end{align*}
	where $\partial A$ stands for the boundary of $A$. Set
	$$n_1=\min\{n\in\mathbb{N}:\delta^n\ge2r_0\}, k_1=\max\{k\in\mathbb{Z}:2C_1\delta^k\le 1\}.$$
	Assume that $n>n_1$ and $k<k_1$, then for any $h\in \G$, we have
	\begin{equation*}
		\begin{split}
			&\sup_{r\in[\delta^{n},\delta^{n+1}]}\frac{m(\mathcal{I}_1(B(h,r),n+k))}{m(B(h,r))}\lesssim\delta^{\epsilon k}\\
			&\sup_{r\in[\delta^{n},\delta^{n+1}]}\frac{m(\mathcal{I}(B(h,r),n+k))}{m(B(h,r))}\lesssim\delta^{\epsilon k}.
		\end{split}
	\end{equation*}
\end{prop}

\bigskip

Based on the four propositions above, we obtain the following technical lemmas.
\begin{lemma}\label{aver}
	Let $(\G,d)$ satisfy the geometric doubling property~\eqref{geo-doubling}.
Then for $f\in L_p(\N)$ with $p\in[1,\8]$,
	\begin{equation*}
		\sup_{r>0}\|M_rf\|_{p}\le D^{1/p}\|f\|_p,
	\end{equation*}
where $M_r$ was defined in~\eqref{averaging-op}.
\end{lemma}
\begin{proof}	
For $p=1$, we have
	\begin{equation*}
		\|M_rf\|_{L_1(\N)}= \|M_rf\|_{L_1(G;L_1(\mathcal M))}\le \int_G\frac{1}{m(B(h,r))}\int_{B(h,r)}\|f(g)\|_{L_1(\mathcal M)}dm(g)dm(h).
	\end{equation*}
Observe that by \cite[Theorem 3.5]{AL19}
	\begin{equation*}
		\int_G\frac{1}{m(B(h,r))}\int_{B(h,r)}\|f(g)\|_{L_1(\mathcal M)}dm(g)dm(h)\le D\|f\|_{L_1(\N)}.
	\end{equation*}
	Thus $\|M_r\|_{L_1(\N)\rightarrow L_1(\N)}\le D$ for all $r>0$. On the other hand, $ \|M_r\|_{L_{\8}(\N)\rightarrow L_{\8}(\N)}\le 1$ for all $r>0$. It then follows from interpolation argument (cf. \cite{Pis-Xu03}) that  $\|M_r\|_{L_p(\N)\rightarrow L_p(\N)}\le D^{1/p}$ for all $p\in[1,\8]$.
\end{proof}

\begin{lem}
	Define
	\begin{align*}
		M_{r,n+k}f(h)=\frac{1}{m(B(h,r))} \int_{\mathcal{I}(B(h,r),n+k)}f(g)dm(g)
	\end{align*}
	and
	\begin{align*}
		M^1_{r,n+k}f(h)=\frac{1}{m(B(h,r))} \int_{\mathcal{I}_1(B(h,r),n+k)}f(g)dm(g).
	\end{align*}
	Let $p\in[1,2]$. Assume that $n>n_{r_0}$ and $k\in\mathbb Z$, then
	\begin{align}\label{L-ineq}
		\sup_{r\in[\delta^n,\delta^{n+1}]}\Big(\|M_{r,n+k}f\|_p+\|M^1_{r,n+k}f\|_p\Big)\lesssim D^{1/p}(1+\delta^k)^{\epsilon/p}\|f\|_p.
	\end{align}
	Furthermore, if $n>n_1$, $k<k_1$, then
	\begin{align}\label{L-ineq1}
		\sup_{r\in[\delta^n,\delta^{n+1}]}\Big(\|M_{r,n+k}f\|_p+\|M^1_{r,n+k}f\|_p\Big)\lesssim D^{1/p}\delta^{\epsilon k/p}\|f\|_p.
	\end{align}
\end{lem}
\begin{proof}
We begin with the first inequality. For $p=1$, $n>n_{r_0}$ and $r\in[\delta^n, \delta^{n+1}]$, since $\mathcal{I}_1(B(h,r),n+k)\subseteq B(h,r+2C_1\delta^{n+k})$, we have from \eqref{int} that
\begin{equation}\label{meas}
  \frac{m(\mathcal{I}_1(B(h,r),k+n))}{m(B(h,r))}\le \frac{m(B(h,r+2C_1\delta^{n+k}))}{m(B(h,r))}\lesssim (1+\delta^k)^\epsilon.
\end{equation}
We then  apply \cite[Theorem 3.5]{AL19} again to deduce that
\begin{equation}\label{L1-esti}
	\begin{split}
		\|M^1_{r,n+k}f\|_1&\lesssim\int_G\frac{1}{m(B(h,r))}\int_{B(h,r+2C_1\delta^{n+k})}\|f(g)\|_{L_1(\M)}dm(g)dm(h)\\
		&\lesssim\int_G\frac{(1+\delta^k)^\epsilon}{m(B(h,r+2C_1\delta^{n+k}))}\int_{B(h,r+2C_1\delta^{n+k})}\|f(g)\|_{L_1(\M)}dm(g)dm(h)\\
		&\le D(1+\delta^k)^\epsilon\|f\|_1.
	\end{split}
\end{equation}
For $p=2$, by the Minkowski and Cauchy-Schwarz inequalities, we have
\begin{equation}\label{short1}
	\begin{split}
		&\Big\|\int_{\mathcal{I}_1(B(h,r),n+k)}f(g)dm(g)\Big\|^2_{L_2(\M)}\le \Big(\int_{\mathcal{I}_1(B(h,r),n+k)}\|f(g)\|_{L_2(\M)}dm(g)\Big)^2\\
		&\le m(\mathcal{I}_1(B(h,r),n+k))\int_{\mathcal{I}_1(B(h,r),n+k)}\|f(g)\|^2_{L_2(\M)}dm(g).
	\end{split}
\end{equation}
According to \eqref{meas} and the proof of Lemma \ref{aver}, we finally obtain
\begin{equation}\label{short}
\begin{split}
 \|M^1_{r,n+k}f\|^2_2&\le \int_G\frac{m(\mathcal{I}_1(B(h,r),n+k))}{m(B(h,r))^2}\int_{\mathcal{I}_1(B(h,r),n+k)}\|f(g)\|^2_{L_2(\M)}dm(g)dm(h)\\
 &\lesssim\int_G\frac{(1+\delta^k)^\epsilon}{m(B(h,r+2C_1\delta^{n+k}))}\int_{B(h,r+2C_1\delta^{n+k})}\|f(g)\|^2_{L_2(\M)}dm(g)dm(h)\\
 &\lesssim D(1+\delta^k)^\epsilon\|f\|^2_{L_2(\N)}.
\end{split}
\end{equation}
The above two estimates are also valid for $M_{r,n+k}f$. So by interpolation argument (cf. \cite{Pis-Xu03}), the first inequality is proved.

For $n>n_1$ and $k<k_1$, a more refined analysis is required compared to the case $n>n_{r_0}$. Observe that for $r\in[\delta^n, \delta^{n+1}]$, the inclusion $\mathcal{I}_1(B(h,r),n+k)\subseteq B(h,r+2C_1\delta^{n+k}) \setminus B(h,r-2C_1\delta^{n+k})$ holds. Define $A(h)$ as the annulus $A(h)={B(h,r+2C_1\delta^{n+k+1})\setminus B(h,r-2C_1\delta^{n+k+1})}$. 

When $p=1$, applying the argument of ~\eqref{L1-esti} and Fubini theorem, we obtain 
\begin{equation}\label{Lemm-1}
	\begin{split}
		\|M^1_{r,n+k}f\|_1&\lesssim\int_G\frac{1}{m(B(h,r))}\int_{A(h)}\|f(g)\|_{L_1(\M)}dm(g)dm(h)\\
		&=\int_G\int_G\frac{\chi_{A(g)}(h)}{m(B(h,r))}\|f(g)\|_{L_1(\M)}dm(h)dm(g).
	\end{split}
\end{equation}
From estimate (4.6) in \cite{Hong-Liu21}, it follows that
\begin{equation}\label{Lemm-2}
  \int_{G}\frac{\chi_{A(g)}(h)}{m(B(h,r))}dm(h)\lesssim D r^{\epsilon k}.
\end{equation}
Combining \eqref{Lemm-1} with \eqref{Lemm-2}, we conclude
\begin{equation}\label{ineq1}
	\begin{split}
		\sup_{r\in[\delta^n,\delta^{n+1}]}\|M^1_{r,n+k}f\|_1\lesssim D\delta^{\epsilon k}\|f\|_1.
	\end{split}
\end{equation}
We consider the case when $p=2$. From~\eqref{short}, we obtain
\begin{equation}\label{Lemm-3}
 \|M^1_{r,n+k}f\|^2_2\le \int_G\frac{m(\mathcal{I}_1(B(h,r),n+k))}{m(B(h,r))^2}\int_{\mathcal{I}_1(B(h,r),n+k)}\|f(g)\|^2_{L_2(\M)}dm(g)dm(h)
 \end{equation}
Proposition~\ref{measure estimate of annulus} shows that
\begin{equation}\label{meas1}
	\frac{m(\mathcal{I}_1(B(h,r),k+n))}{m(B(h,r))}\lesssim \delta^{\epsilon k}.
\end{equation}
Combining \eqref{Lemm-3} with \eqref{meas1} and applying the reasoning from \eqref{short}, we obtain
\begin{equation*}
 \|M^1_{r,n+k}f\|^2_2\lesssim D\delta^{\epsilon k}\|f\|^2_2.
\end{equation*}
The same argument applies to $M_{r,n+k}f$ since $\mathcal{I}(B(h,r),k+n)\subseteq \mathcal{I}_1(B(h,r),k+n)$. Therefore, the interpolation argument gives the second inequality.
\end{proof}

\begin{lem}\label{short-in}
Let $n>n_{r_0}$. Let $(r_i)_i$ be an increasing sequence positive numbers belonging to $[\delta^n,\delta^{n+1}]$. Then 
\begin{align*}
&\|\sum_i\varepsilon_i(M_{r_i}f-M_{r_{i+1}}f)\|_{L_1(L_\8(\Omega)\overline{\otimes}\N)}\lesssim D\delta^{\epsilon}\|f\|_{L_1(\N)},
\end{align*}
where $(\varepsilon_i)$ is the Rademacher sequence on a probablity space $(\Omega,P)$.
\end{lem}
\begin{proof}
Note that
\begin{equation}\label{aver-in}
	\begin{split}
		M_{r_i}f(h)-M_{r_{i+1}}f(h)
		&=\Big(\frac{1}{m(B(h,r_i))}-\frac{1}{\m(B(h,r_{i+1}))}\Big)\int_{B(h,r_i)}f(g)dm(y)\\
		&\ \ +\frac{1}{m(B(h,r_{i+1}))}\int_{B(h,r_{i+1})\setminus B(h,r_i)}f(g)dm(y).
	\end{split}
\end{equation}
Therefore, by the Minkowski inequality,
\begin{equation*}
	\begin{split}
		&\|M_{r_i}f-M_{r_{i+1}}f\|_{L_1(\mathcal N)}\\
		&\le \int_G\Big(\frac{1}{m(B(h,r_i))}-\frac{1}{m(B(h,r_{i+1}))}\Big)\int_{B(h,r_i)}\|f(g)\|_{L_1(\M)}dm(g)dm(h)\\
		&\ \ +\int_G\frac{1}{m(B(h,r_{i+1}))}\int_{B(h,r_{i+1})\setminus B(h,r_i)}\|f(y)\|_{L_1(\M)}dm(g)dm(h).
	\end{split}
\end{equation*}
Then summing over all $i$, one gets
\begin{align*}
	&\sum_i\|M_{r_i}f-M_{r_{i+1}}f\|_{L_1(\mathcal N)}\le\int_G\frac{2}{m(B(h,\delta^{n}))}\int_{B(h,\delta^{n+1})}\|f(g)\|_{L_1(\M)}dm(g)dm(h).
\end{align*}
Combining \eqref{int} with Lemma \ref{aver}, we have
\begin{align*}
	&\int_G\frac{1}{m(B(h,\delta^{n}))}\int_{B(h,\delta^{n+1})}\|f(g)\|_{L_1(\M)}dm(g)dm(h)\\
	&\lesssim \delta^{\epsilon} \int_G\frac{1}{m(B(h,\delta^{n+1}))}\int_{B(h,\delta^{n+1})}\|f(g)\|_{L_1(\M)}dm(g)dm(h)\\
	& \lesssim D\delta^{\epsilon}\|f\|_1,
\end{align*}
and the lemma is proved.
\end{proof}
\subsection{Noncommutative  Calder{\'o}n-Zygmund decomposition}\label{nonmar}\quad

\medskip

In this subsection, we introduce the noncommutative Calder{\'o}n-Zygmund decomposition with the non-doubling measure recently developed in~\cite{Cad-Alo-Par21}. In order to introduce this decomposition, we first state the resulting martingale in noncommutative setting. For $k\in\mathbb Z$, we denote by $\mathcal{F}_k$ the $\sigma$-algebra generated by the `dyadic cubes' $\{Q_\alpha^{k}:\alpha\in I_k\}$. Let $\mathcal{F}=\cup_{k\in\Z}\mathcal{F}_k$. Set $\N_k=L_{\8}(\G,\mathcal{F}_k,m)\overline{\otimes}\M$. It is clear that $(\N_k)_{k\in\Z}$ is a sequence of decreasing von Neumann subalgebra of $\N=L_{\8}(\G,m)\overline{\otimes}\M$ and $\cup_{k\in\Z}\N_k$ is weak* dense in $\N$. Let $(\mathsf{E}_k)_{k\in\Z}$ be a sequence of conditional expectations of $\N$ with respect to the filtration $(\N_k)_{k\in\Z}$.  Moreover, let $p\in[1,\8)$ and $f\in L_p(\mathcal N)$, we have
\begin{equation}\label{con-exp}
  \mathsf{E}_kf=\sum_{\alpha\in I_k}f_{Q_\alpha^k} \chi_{Q_\alpha^k},
\end{equation}
where $\chi_{Q_\alpha^k}$ is the characteristic function of $Q_\alpha^k$ and
\begin{equation*}
  f_{Q_\alpha^k}=\frac{1}{m(Q_\alpha^k)}\int_{Q_\alpha^k}f(w)dm(w).
\end{equation*}
It is easily seen that $\mathsf{E}_k\circ \mathsf{E}_j=\mathsf{E}_{\max(k,j)}$. Let $f\in L_p(\N)$, one can easily check that $(\mathsf{E}_kf)_k$ is an $L_p$-reverse martingale, namely
\begin{equation}\label{mar}
 \sup_{k}\|\mathsf E_{k}f\|_{p}\le\|f\|_p.
\end{equation}
For abbreviation, we write $f_k$ instead of $\mathsf{E}_k f$.

According to the argument given in~\cite{Alo-San16}, we assume that $m(G)=\8$. Define a dense subset of $L_1(\mathcal N)_+$
\begin{equation*}
  \mathcal{N}_{c,+}=L_1(\mathcal{N})\cap\{f|f\in\mathcal{N}_+,~\overrightarrow{\mbox{supp}}f~\mbox{is compact}\},
\end{equation*}
where $\overrightarrow{\mbox{supp}}f=\mbox{supp}\|f\|_{L_1(\mathcal M)}$. Fix $f\in\mathcal{N}_{c,+}$ and $\lambda>0$. By this assumption and the argument stated  in \cite[Lemma 3.1]{Pracet08}, there exists $m_{\lambda}(f)\in\mathbb{N}$ such that $f_k\le\lambda\1_{\N}$ for all $k> m_\lambda(f)$. The following Cuculescu's theorem \cite{Cuc} can be found in \cite[Lemma 3.1]{Pracet08}.
\begin{lem}\label{Cu}
Given $f\in\mathcal{N}_{c,+}$ and let $(f_k)_k$ be the associated martingale with respect to the filtration $(\N_k)_k$. Fix $\lambda>0$, there exists a $m_{\lambda}(f)\in\mathbb{N}$ and an increasing sequence projections $(q_k)_{k\in\mathbb{Z}}$ given by
\begin{equation*}
   q_k=\left\{
         \begin{array}{ll}
           \1_{\mathcal N}, & \hbox{$k>m_{\lambda}(f)$;} \\
           \chi_{(0,\lambda]}(f_k), & \hbox{$k=m_{\lambda}(f)$;}\\
           \chi_{(0,\lambda]}(q_{k+1}f_kq_{k+1}), & \hbox{$k<m_{\lambda}(f)$.}
         \end{array}
       \right.
\end{equation*}
 such that
\begin{enumerate}[\noindent]
  \item\emph{(i)}~$q_k$ is a projection in $\mathcal{N}_k$ and commutes with $q_{k+1}f_kq_{k+1}$;
  \item\emph{(ii)}~$q_kf_kq_k\le \lambda q_k$ for each $k\in\mathbb Z$;
  \item\emph{(iii)}~set $q=\bigwedge_{k}q_k$, then
  \begin{equation}\label{weak-est-marti}
      \|qf_kq\|_{\mathcal N}\le \lambda~\mbox{and}~\varphi(\1_{\mathcal N}-q)\le\frac{\|f\|_1}{\lambda}.
  \end{equation}
\end{enumerate}
\end{lem}
More specifically, we have the following expression of $q_k$ for every $k\in\mathbb{Z}$,
\begin{equation*}
  q_k=\sum_{\alpha\in I_k}q_{Q_\alpha^k}\chi_{Q_\alpha^k},
\end{equation*}
where $q_{Q_\alpha^k}$ is a projection in $\mathcal M$ defined by
\begin{equation*}
  q_{Q_\alpha^k}=\left\{
                   \begin{array}{ll}
                     \1_{\mathcal{M}}, & \hbox{if $k> m_{\lambda}(f)$;} \\
                     \chi_{(0,\lambda]}(f_{Q_\alpha^k}),&\hbox{if $k=m_{\lambda}(f)$;}\\
                     \chi_{(0,\lambda]}(q_{\widehat{Q}_\alpha^k}f_{Q_\alpha^k}q_{\widehat{Q}_\alpha^k}), & \hbox{if $k<m_{\lambda}(f)$.}
                   \end{array}
                 \right.
\end{equation*}
According to the Cuculescu construction, one can see that these projections satisfy
\begin{equation}\label{Cucul}
 q_{Q_\alpha^k}~\mbox{commutes with}~q_{\widehat{Q}_\alpha^k}f_{Q_\alpha^k}q_{\widehat{Q}_\alpha^k},\ \ \ q_{Q_\alpha^k}\le q_{\widehat{Q}_\alpha^k},\ \ \ q_{Q_\alpha^k}f_{Q_\alpha^k}q_{Q_\alpha^k}\le \lambda q_{Q_\alpha^k}.
\end{equation}
Set $p_k=q_{k+1}-q_k$, then
\begin{equation}\label{decom}
 p_k=\sum_{\a\in I_k}(q_{\widehat{Q}_{\a}^k}-q_{Q_{\a}^k})\chi_{Q_{\a}^k}:=\sum_{\a\in I_k}p_{{Q}_{\a}^k}\chi_{Q_{\a}^k},\quad\mbox{and}\quad\sum_{k} p_k=\1_{\mathcal N}-q.
\end{equation}

We now present the noncommutative Calder\'on-Zygmund decomposition given in~\cite{Cad-Alo-Par21}.
\begin{prop}\label{CZ}
With the notations given in the Cuculescu construction and let $f\in\mathcal{N}_{c,+}$. Set
 \begin{align*}
&g:=qfq+\sum_{k}\mathsf{E}_{k+1}(p_kfp_k),\\
&b_d:=\sum_{k}b_{d,k}=\sum_{k}\big(p_kfp_k-\mathsf{E}_{k+1}(p_kfp_k)\big),\\
&b_{\textit{off}}:=\sum_{k}b_{\textit{off},k}=\sum_{k}\big(q_kfp_k+p_kfq_k\big).
\end{align*}
Then
\begin{enumerate}[\noindent]
   \item\emph{(i)}~$f=g+b_{d}+b_{\textit{off}}$;
    \item\emph{(ii)}~$\|g\|_{1}\le \|f\|_{1}$ and $\|g\|^2_{2}\le 6\lambda\|f\|_{1}$;
    \item \emph{(iii)}~for every $k\in\Z$,
    $$b_{d,k}=\sum_{\a\in I_{k}}b_{d,Q_{\a}^{k}}:=\sum_{\a\in I_{k}}\Big(p_{{Q}_{\a}^{k}}fp_{{Q}_{\a}^{k}}\chi_{Q_{\a}^{k}}-\frac{m({Q}_{\a}^{k})}
{m(\widehat{Q}_{\a}^{k})}p_{{Q}_{\a}^{k}}f_{Q_{\a}^{k}}p_{{Q}_{\a}^{k}}\Big)\chi_{\widehat{Q}_{\a}^{k}}$$
    and
    $$\int_{\G}b_{d,Q_{\a}^{k}}=0,~\qquad~ \|b_d\|_1\le\sum_{k}\sum_{\a\in I_{k}}\|b_{d,Q_{\a}^{k}}\|_1\le 2\|f\|_1$$
     \item \emph{(iv)}~for every $k\in\Z$,
     $$b_{\textit{off},k}=\sum_{\a\in I_{k}}b_{\textit{off},Q_{\a}^{k}}:=\sum_{\a\in I_k}(p_{Q_{\a}^k}fq_{Q_{\a}^k}+q_{Q_{\a}^k}fp_{Q_{\a}^k})\chi_{{Q}_{\a}^{k}}$$
     and
     $$\int_{\G}b_{\textit{off},Q_{\a}^{k}}=0.$$
  \end{enumerate}
\end{prop}

The following lemma will be useful in the weak type $(1,1)$ estimates.
\begin{lem}\label{supp lemma}
Set
 \begin{equation*}
 \widetilde{Q}_\alpha^k=\{w\in \G:d(w,z_\alpha^k)\le 4C_1\delta^{k+1}\}.
\end{equation*}
Define $k_2=\min\{k:a_0\delta^{k}>r_0\}$ and
\begin{equation*}
  \zeta=\1_{\mathcal{N}}-\bigg(\bigvee_{k> k_2}\bigvee_{\alpha\in I_k}p_{{Q}_\alpha^k}\chi_{\widetilde{Q}_\alpha^k}\bigg)\bigvee\bigg(\bigvee_{k\leq k_2}\bigvee_{\alpha\in I_k}p_{Q_\alpha^k}\chi_{{Q}_\alpha^k}\bigg).
\end{equation*}
Then
\begin{enumerate}[\noindent]
  \item~\emph{(i)} $\varphi(\1_{\mathcal{N}}-\zeta)\lesssim \frac{2(K+1)(4C_1\delta/a_0)^{\epsilon}}{\lambda}\|f\|_{1}$;
  \item~\emph{(ii)} for every $k>k_2$
  \begin{align*}
  	&\zeta(h)b_{d,k}(g)\chi_{B(h,2C_1\delta^{k+1})}(g)\zeta(h)=0,~\qquad~\zeta(h)b_{\textit{off},k}
  	(g)\chi_{B(h,2C_1\delta^{k+1})}(g)\zeta(h)=0;
  \end{align*}
  and for every $k\le k_2$
  \begin{align*}
  	&\zeta(h)b_{d,k}(g)\chi_{Q_{\a}^k(h)}(g)\zeta(h)=0,~\qquad~\zeta(h)b_{\textit{off},k}
  	(g)\chi_{Q_{\a}^k(h)}(g)\zeta(h)=0,
  \end{align*}
  where $Q_{\a}^k(h)$ denotes the `dyadic cube' $Q_{\a}^k$ in $\mathcal{F}_k$ containing $h$.
\end{enumerate}
\end{lem}
\begin{proof}
Note that
\begin{align*}
  \varphi(\1_{\N}-\zeta)&\le\sum_{k> k_2}\sum_{\alpha\in I_k}\varphi(p_{Q_\alpha^k}\chi_{\widetilde{Q}_\alpha^k})+\sum_{k< k_2}\sum_{\alpha\in I_k}\varphi(p_{Q_\alpha^k}\chi_{Q_\alpha^k})\\
&\le \sum_{k> k_2}\sum_{\alpha\in I_k}\tau(p_{Q_\alpha^k})\frac{m(\widetilde{Q}_\alpha^k)}{m(Q_\alpha^k)}m(Q_\alpha^k)+\sum_{k\leq k_2}\sum_{\alpha\in I_k}\tau(p_{Q_\alpha^k})m({Q_\alpha^k})\\
&\lesssim(4C_1\delta/a_0)^{\epsilon}\sum_{k> k_2}\sum_{\alpha\in I_k}\tau(p_{Q_\alpha^k})m({Q_\alpha^k})+\sum_{k\leq k_2}\sum_{\alpha\in I_k}\tau(p_{Q_\alpha^k})m({{Q}_\alpha^k})\\
&\lesssim\frac{(4C_1\delta/a_0)^{\epsilon}}{\lambda}\|f\|_{1},
\end{align*}
where we used Proposition~\ref{dyadic cube}(iv) and~\eqref{int} in the thrid inequality,~\eqref{decom} and~\eqref{weak-est-marti} in the last inequality.  The estimate (i) is proved.

For the assertion (ii), we first focus on the case $k>k_2$. By Proposition \ref{CZ}(iii), the term $b_{d,k}$  admits the decomposition $b_{d,k}=\sum_{\a\in I_{k}}b_{d,Q_{\a}^{k}}$,  where each $b_{d,Q_{\a}^k}$ has support contained in $\widehat{Q}_{\a}^{k}$. Without loss of generality, assume there exists an $\a\in I_k$ such that $B(h,2C_1\delta^{k+1})\cap \widehat{Q}_{\a}^{k}\neq\emptyset$. Taking $g\in B(h,2C_1\delta^{k+1})\cap \widehat{Q}_{\a}^{k}$. According to Proposition~\ref{dyadic cube}(iv), we have
\begin{equation*}
	d(h,{z}_{\a}^k)\le d(h,g)+d(g,\widehat{z}_{\a}^k)+d(\widehat{z}_{\a}^k,{z}_{\a}^k)\le 4C_1\delta^{k+1}.
\end{equation*}
This yields $h\in\widetilde{Q}_{\a}^k$. Since $p_{{Q}_\alpha^k}\leq\1_{\M}-\zeta(h)$, we immediately obtain  $\zeta(h)\le\1_{\M}-p_{{Q}_\alpha^k}$. Consequently,
\begin{align*}
	&\zeta(h)b_{d,Q_{\a}^k}(g)\chi_{B(h,2C_1\delta^{k+1})}(g)\zeta(h)\\
	&=\zeta(h)(\1_{\M}-p_{{Q}_\alpha^k})b_{d,Q_{\a}^k}(g)\chi_{B(h,2C_1\delta^{k+1})}(g)(\1_{\M}-p_{{Q}_\alpha^k})\zeta(h)=0
\end{align*}
The vanishing property $\zeta(h)b_{d,k}(g)\chi_{B(h,2C_1\delta^{k+1})}(g)\zeta(h)=0$ is established. Note that each $b_{\textit{off},Q_{\a}^k}$ is supported on ${Q}_{\a}^{k}$. Following the same methodology employed for $b_{d,Q_{\a}^{k}}$, we establish the vanishing property $\zeta(h)b_{\textit{off},Q_{\a}^k}(g)\chi_{B(h,2C_1\delta^{k+1})}(g)\zeta(h)=0$. It follows that the vanishing property of off-diagonal terms $\zeta(h)b_{\textit{off},k}
  	(g)\chi_{B(h,2C_1\delta^{k+1})}(g)\zeta(h)=0$ is established.

It remains to prove the assertion (ii) for the case $k\le k_2$. Fix $h$. Let $Q_{\a}^k$ denote the unique dyadic cube in $\mathcal{F}_k$ containing $h$. By the fact that $\zeta\le \1_{\N}-p_{{Q}_\alpha^k}\chi_{{Q}_\alpha^k}$, one can easily check that $\zeta(h)b_{d,Q_{\a}^k}(g)\chi_{Q_{\a}^k(h)}(g)\zeta(h)=\zeta(h)b_{\textit{off},Q_{\a}^k(h)}(g)\zeta(h)=0$. Hence $\zeta(h)b_{d,k}(g)\chi_{Q_{\a}^k(h)}(g)\zeta(h)=\zeta(h)b_{\textit{off},k}
  	(g)\chi_{Q_{\a}^k(h)}(g)\zeta(h)=0$. This completes the proof.
\end{proof}

\section{Proof Strategy of Theorem \ref{main-thm1}}
In the following, we focus on the proof of Theorem~\ref{main-thm1}. Motivated by the study of the variational inequalities~ (see e.g. \cite{Bour89}), the square function being considered can be split into the `long one' and `short one'. More precisely, for an interval $I_i=[r_i,r_{i+1})$, one has the following two cases:
\begin{enumerate}[(1)]
	\item [$\bullet$] Case 1: $I_i$ contains no `dyadic point' $\delta^k$, that is, for any $k\in\mathbb N$, $\delta^k\notin I_i$;
	\item [$\bullet$] Case 2: $I_i$ contains at least one `dyadic point' $\delta^k$ for $k\in\mathbb N$.
\end{enumerate}
Here $\delta>1$ is a constant depending on $\G$ that will be determined in Proposition~\ref{dyadic cube}. By above classification, we then divide the interval $I_i=[r_i,r_{i+1})$ into disjoint parts
\begin{equation}\label{divide}
	[r_i,r_{i+1})=[r_i,\tilde{r}_i)\cup[\tilde{r}_i,\tilde{\tilde{r}}_{i})\cup[\tilde{\tilde{r}}_{i},r_{i+1}),
\end{equation}
where $\tilde{r}_i$ and $\tilde{\tilde{r}}_{i}$ is determined by the law: if $I_i$ belongs to Case 1, set $\tilde{r}_i=\tilde{\tilde{r}}_{i}=r_{i+1}$; if $I_i$ belongs to Case 2, set $\tilde{r}_i=\delta^{k_i}:=\min\{\delta^{k}:\delta^k\in I_i\}$ and $\tilde{\tilde{r}}_{i}=\delta^{l_i}:=\max\{\delta^{k}:\delta^k\in \bar{I}_i\}$ in which $\bar{I_i}$ is the closure of $I_i$. According to~\eqref{divide}, we now introduce two collections of intervals with respect to $\{[r_i,r_{i+1})\}_i$:
\begin{enumerate}[(1)]
	\item [$\bullet$] ${\mathrm S}$ consists of all intervals $I_i$ belonging to Case 1, or $[{r}_{i},{\tilde{r}}_{i})$, $[{\tilde{\tilde{r}}}_{i},r_{i+1})$ in~\eqref{divide}. 
	\item [$\bullet$] ${\mathrm L}$ consists of all intervals $[\tilde{r}_i, \tilde{\tilde{r}}_{i})$ in~\eqref{divide}.
\end{enumerate}
Then by above decomposition of intervals and the quasi-triangle inequality for weak $L_1$-norm, we deduce that
\begin{equation*}
	\begin{split}
		&\|(M_{{r_{i}}}f-M_{{r_{i+1}}}f)_{i\in\mathbb N}\|_{L_{1,\8}(\mathcal{N};\ell_{2}^{rc})}\le 3\|(M_{{r_{i}}}f-M_{{\tilde{r}_i}}f)_{i\in\mathbb N}\|_{L_{1,\8}(\mathcal{N};\ell_{2}^{rc})}\\
		&+3\|(M_{{\tilde{r}_{i}}}f-M_{{\tilde{\tilde{r}}_{i}}}f)_{i\in\mathbb N}\|_{L_{1,\8}(\mathcal{N};\ell_{2}^{rc})}+3\|(M_{{\tilde{\tilde{r}}_{i}}}f-M_{{r_{i+1}}}f)_{i\in\mathbb N}\|_{L_{1,\8}(\mathcal{N};\ell_{2}^{rc})}\\
		&\leq3\|(M_{{\tilde{r}_i}}f-M_{{\tilde{\tilde{r}}_{i:}}}f)_{i:[\tilde{r}_i,\tilde{\tilde{r}}_{i})\in L}\|_{L_{1,\8}(\mathcal{N};\ell_{2}^{rc})}+6\|(M_{{s_{i}}}f-M_{{\tilde{s}_{i}}}f)_{i:[s_i,\tilde{s}_i)\in\mathrm S}\|_{L_{1,\8}(\mathcal{N};\ell_{2}^{rc})}.
	\end{split}
\end{equation*}

Consider the term $\|(M_{{\tilde{r}_i}}f-M_{{\tilde{\tilde{r}}_{i:}}}f)_{i:[\tilde{r}_i,\tilde{\tilde{r}}_{i})\in L}\|_{L_{1,\8}(\mathcal{N};\ell_{2}^{rc})}$ firstly. By the law, we wirte $[\tilde{r}_i,\tilde{\tilde{r}}_i)=[\delta^{k_i},\delta^{l_i})$, and decompose
\begin{equation*}
	M_{\delta^{k_i}}f-M_{\delta^{l_i}}f=M_{\delta^{k_i}}f-\mathsf{E}_{k_i}f+\mathsf{E}_{k_i}f-\mathsf{E}_{l_i}f+M_{\delta^{l_i}}f-\mathsf{E}_{l_i}f,
\end{equation*}
where $(\mathsf{E}_kf)_k$ is the sequence of conditional expectations defined in Section \ref{pre}. Consequently, there exists a sequence of positive integers $k_1<l_1\leq k_2<l_2< \dotsm\leq k_i<l_i\leq\dotsm$ such that
\begin{equation*}
	\begin{split}
		&\|(M_{{\tilde{r}_i}}f-M_{{\tilde{\tilde{r}}_{i:}}}f)_{i:[\tilde{r}_i,\tilde{\tilde{r}}_{i})\in L}\|_{L_{1,\8}(\mathcal{N};\ell_{2}^{rc})}\le
		3\|(M_{\delta^{k_i}}f-\mathsf{E}_{k_i}f)_i\|_{L_{1,\8}(\mathcal{N};\ell_{2}^{rc})}\\
		&+3\|(\mathsf{E}_{k_i}f-\mathsf{E}_{l_i}f)_i\|_{L_{1,\8}(\mathcal{N};\ell_{2}^{rc})}
		+3\|(M_{\delta^{l_i}}f-\mathsf{E}_{l_i}f)_i\|_{L_{1,\8}(\mathcal{N};\ell_{2}^{rc})}\\
		&\le 6\|(M_{\delta^n}f-\mathsf{E}_nf)_{n>n_{r_0}}\|_{L_{1,\8}(\mathcal N;\ell_2^{rc})}+3\|(\mathsf{E}_{k_i}f-\mathsf{E}_{l_i}f)_i\|_{L_{1,\8}(\mathcal{N};\ell_{2}^{rc})},
	\end{split}
\end{equation*}
where $n_{r_0}$ is the unique integer such that $\delta^{n_{r_0}}<r_0\leq \delta^{n_{r_0}+1}$ and the last inequality follows from the fact that $n\mapsto(\sum_{i=0}^n|x_i|^2)^{1/2}$ is increasing. It then follows that
\begin{equation}\label{decomp}
	\begin{split}
		&\|(M_{{r_{i}}}f-M_{{r_{i+1}}}f)_{i\in\mathbb N}\|_{L_{1,\8}(\mathcal{N};\ell_{2}^{rc})}\le 6\|(M_{{s_{i}}}f-M_{{\tilde{s}_{i}}}f)_{i:[s_i,\tilde{s}_i)\in\mathrm S}\|_{L_{1,\8}(\mathcal{N};\ell_{2}^{rc})}\\
		&+9\|(\mathsf{E}_{k_i}f-\mathsf{E}_{l_i}f)_i\|_{L_{1,\8}(\mathcal{N};\ell_{2}^{rc})}+18\|(M_{\delta^n}f-\mathsf{E}_nf)_{n>n_{r_0}}\|_{L_{1,\8}(\mathcal N;\ell_2^{rc})}.
	\end{split}
\end{equation}

By using the above argument, it is not difficult to see that~\eqref{decomp} also holds for $\|\cdot\|_{L_p(\mathcal{N}; \ell_{2}^{rc})}$ and $\|\cdot\|_{\mathrm{BMO}_{d}(\mathcal{A})}$ via the triangle inequalities with possibly different constants. On the other hand,  $(\mathsf{E}_{k_i}f-\mathsf{E}_{l_{i}}f)_{{i}}$ forms a new sequence of martingale differences, and the analogue of Theorem~\ref{main-thm1} for martingale differences have already been established in \cite{Ran1}, see also \cite{PX,Ran2}.

Therefore, it is sufficient to estimate the `long one'  $(M_{\delta^n}f-\mathsf{E}_nf)_{n>n_{r_0}}$ and `short one' $(M_{{s_{i}}}f-M_{{\tilde{s}_{i}}}f)_{i:[s_i,\tilde{s}_i)\in\mathrm S}$. {We begin by addressing the `long one' and proceed to derive a more general result.}
\begin{thm}\label{long-ineq}
	Let $1\le p\le \8$. Define the operator 
	$$Lf=\sum_{n\in\mathbb{N}} v_n(M_{\delta^{n+n_{r_0}}}f-\mathsf{E}_{n+n_{r_0}}f),$$
	 where $(v_n)_{n\in\mathbb{N}}$ is the sequence of bounded numbers in $\ell_{\infty}$. Then the following assertions hold with a positive constant $C_p$ independent of $L_n$ and $f$:
	\begin{enumerate}[\noindent]
		\item~\emph{(i)} for $p=1$, $$\|Lf\|_{L_{1,\infty}(\mathcal{N})}\le  C_{p}\|f\|_{1},\; \forall f\in L_{1}(\mathcal N);$$
		\item~\emph{(ii)} for $p=\infty$,
		\begin{align*}
			\|Lf\|_{\mathrm{BMO}_{d}(\mathcal{N})}\le  C_{p}\,\|f\|_\infty,\; \forall f\in L_{\infty}(\mathcal N);
		\end{align*}
		\item~\emph{(iii)} for $1<p<\infty$, $$\|Lf\|_{L_p(\mathcal{N})}\le  C_{p} \|f\|_{p}, \; \forall f\in L_{p}(\mathcal N).$$
	\end{enumerate}
\end{thm}
{It remains to consider the `short one'.} By an abuse of notation, we rewrite the sequence $\{s_0,\tilde{s}_0,s_1,\tilde{s}_1,\cdots,s_i,\tilde{s}_i,\cdots\}$ in order as $\{s_1,s_2,\cdots,s_i,s_{i+1},\cdots\}$; we then denote the collection of such intervals $[s_i,s_{i+1})$ as $\mathrm S$. We abbreviate each interval $[s_i,s_{i+1})$ as $i$ and denote the collection of all such $i$ by $\mathcal S$. Let $T_i=M_{s_i}-M_{s_{i+1}}$, and let $(T_if)_{i\in\mathcal S}$ denote $(M_{s_i}f-M_{s_{i+1}}f)_{i:[s_i,s_{i+1})\in\mathrm S}$. Since the mapping $n\mapsto(\sum_{k=1}^n|x_k|^2)^{1/2}$ is increasing, it suffices to show the following result.
\begin{thm}\label{short-ineq}
	Let $1\le p\le \8$. Let $T_i$ and $\mathcal S$ defined as above. Then the following assertions hold with a positive constant $C_p$ independent of $T_i$ and $f$:
	\begin{enumerate}[\noindent]
		\item~\emph{(i)} for $p=1$, $$\|(T_if)_{i\in \mathcal S}\|_{L_{1,\infty}(\mathcal{N};
			\ell_{2}^{rc})}\le  C_{p}\|f\|_{1},\; \forall f\in L_{1}(\mathcal N);$$
		\item~\emph{(ii)} for $p=\infty$,
		\begin{align*}
			\Big\|\sum_{i\in\mathcal S} T_if\otimes e_{1i}\Big\|_{\mathrm{BMO}_{d}(\mathcal{A})}&+\Big\|\sum_{i\in\mathcal S} T_if \otimes e_{i1}\Big\|_{\mathrm{BMO}_{d}(\mathcal{A})}\le  C_{p}\,\|f\|_\infty,\; \forall f\in L_{\infty}(\mathcal N);
		\end{align*}
		\item~\emph{(iii)} for $1<p<\infty$, $$\|(T_if)_{i\in \mathcal S}\|_{L_p(\mathcal{N};
			\ell_{2}^{rc})}\le  C_{p} \|f\|_{p}, \; \forall f\in L_{p}(\mathcal N).$$
	\end{enumerate}
\end{thm}
\medskip

\begin{remark}{\rm
In Theorem~\ref{long-ineq}, when $(v_n)$ is a Rademacher sequence $(\varepsilon_n)$, then by the noncommutative Khintchine inequalities (see Proposition~\ref{nonkin}), Theorem~\ref{long-ineq}(i) and (iii) immediately yield the following two inequalities:
$$\|(M_{\delta^n}f-\mathsf{E}_nf)_{n>n_{r_0}}\|_{L_{1,\8}(\mathcal N;\ell_2^{rc})}\lesssim \|f\|_{1} \quad \mbox{and} \quad \|(M_{\delta^n}f-\mathsf{E}_nf)_{n>n_{r_0}}\|_{L_{p}(\mathcal N;\ell_2^{rc})}\lesssim \|f\|_{p}.$$
We also note that while Theorem \ref{long-ineq}(ii) does not provide the $\mathrm{BMO}$ estimate for the `long one' $(M_{\delta^n}f-\mathsf{E}_nf)_{n>n_{r_0}}$. Although this estimate is similar to the one in Theorem~\ref{short-ineq}(ii), a detailed analysis will be presented in Subsection~\ref{BMO-est}.

Theorem \ref{long-ineq} was first established by Jones and Rosenblatt \cite[Theorem 2]{JR02} in the commutative setting. Later, the fourth author extended this result to the operator-valued setting for the translation action of $\mathbb{R}^d$ \cite{Xu23}. Theorem \ref{short-ineq} originated from the work of the first, second, and fourth authors on quantitative mean ergodic inequalities \cite{HLX}, building on Jones et al.'s study of square function inequalities for ergodic averaging \cite{JOR96}. These two theorems significantly refine and generalize these earlier results.
}
\end{remark}

The proof strategy for Theorem \ref{long-ineq} and \ref{short-ineq}  draws on ideas from \cite{HLX,Xu23}. However, conditions \eqref{decay property} and \eqref{geo-doubling} do not guarantee that the space $(G,d,m)$ is a doubling metric measure space, which introduces new challenges. For instance, the dual argument used to establish $L_p$ estimation for high-index $2<p<\infty$ in \cite{Xu23} is ineffective.
The `dyadic cubes' constructed in Proposition \ref{dyadic cube} may lack the small boundary property (cf.~\cite[Theorem 11]{Christ90}. Moreover,  the noncommutative Calder\'{o}n-Zygmund decomposition for homogeneous spaces becomes insufficient. Consequently, our analysis necessitates a careful geometric treatment of cubes and balls, in conjunction with the noncommutative non-doubling Calder\'{o}n-Zygmund theory.

\section{Proof of Theorem \ref{long-ineq} and \ref{short-ineq}: Weak type $(1,1)$ estimates}\label{weak-type}
In this section, we prove the weak type $(1,1)$ estimates for Theorem~\ref{long-ineq} and \ref{short-ineq}. We start by proving the estimate for the `long one', namely Theorem~\ref{long-ineq}, as the corresponding result for Theorem~\ref{short-ineq} (the `short one') can be derived analogously to the `long one' via the following lemma.
		\begin{lem}\label{kin1-weak}
		 Let $(\varepsilon_i)_{i\in\mathcal{S}}$ be a Rademacher sequence on a fixed probability space $(\Omega,P)$ and $T_i$ and $\mathcal{S}$ be defined in Theorem~\ref{short-ineq}. Define
		$$\mathcal{T}f=\sum_{i\in \mathcal{S}}\varepsilon_iT_{i}f.$$
		Then
		$$\|(T_if)_{i\in \mathcal{S}}\|_{L_{1,\infty}(\mathcal{N};\ell_{2}^{cr})}\thickapprox
		\|\mathcal{T}f\|_{L_{1,\infty}(L_{\infty}(\Omega)\overline{\otimes}\mathcal N)},$$
		where $L_{\infty}(\Omega)\overline{\otimes}\mathcal N$ is a tensor von Neumann algebra equipped with trace $\widetilde{\varphi}=\varphi\otimes P$.
	\end{lem}
	The aforementioned lemma is an immediate consequence of the noncommutative Khintchine inequalities (Proposition \ref{nonkin}). 
	
The procedure to show weak type $(1,1)$ estimates is now standard. Let $f\in L_1(\N)$. Since $f$ can be decomposed into with linear combination four positive elements, namely $f=f_1-f_2+i(f_3-f_4)$ with $f_j\ge 0$ and $\|f_j\|_1\le\|f\|_1$ for $j=1,2,3,4$,  it is enough to prove weak type $(1,1)$ estimates by assuming that $f$ is positive. Moreover, $\N_{c,+}$ is dense in $L_1(\N)_+$, so we may just consider $f\in \N_{c,+}$. Now fix $f\in \N_{c,+}$ and $\lambda>0$. In what follows, the symbol $T$ will denote either the operator $L$ or $\mathcal{T}$, while ${\varphi}_1$ will denote either the trace $\varphi$ or $\widetilde{\varphi}$. By combining the distribution inequality~\eqref{distrib} with the noncommutative Calder\'on-Zygmund decomposition (Proposition~\ref{CZ}), we conclude that
	\begin{equation*}
		{\varphi}_1(\chi_{(\lambda,\8)}(|Tf|))\le {\varphi}_1(\chi_{(\lambda/3,\8)}(|Tg|))+{\varphi}_1(\chi_{(\lambda/3,\8)}(|Tb_d|))
		+{\varphi}_1(\chi_{(\lambda/3,\8)}(|Tb_{\it{off}}|)).
	\end{equation*}
	Therefore, by Lemma \ref{kin1-weak}, the weak type $(1,1)$ estimate of $(T_if)_i$ reduces to prove
	\begin{equation}\label{weak}
		\begin{split}
			&{\varphi}_1(\chi_{(\lambda,\8)}(|Tb_d|))\lesssim \frac{\|f\|_1}{\lambda}\\
			&{\varphi}_1(\chi_{(\lambda,\8)}(|Tb_{\it{off}}|))\lesssim \frac{\|f\|_1}{\lambda}\\
			&{\varphi}_1(\chi_{(\lambda,\8)}(|Tg|))\lesssim \frac{\|f\|_1}{\lambda}.
		\end{split}
	\end{equation}

\subsection{Proof of Theorem \ref{long-ineq}(i)}\quad

\medskip
We begin with the following estimate: Let $L_n=M_{\delta^{n+n_{r_0}}}-\mathsf{E}_{n+n_{r_0}}$ and $n_2=\max\{n_{r_0}, n_1, k_2\}$, where $k_2$ is the constant defined in Lemma \ref{supp lemma}. Then we have
\begin{equation}\label{weak-ineq1}
\begin{split}
	\Big\|\sum_{0<n\le n_2-n_{r_0}}v_n L_nf\Big\|_
	{L_{1,\8}(L_{\infty}(\Omega)\overline{\otimes}\mathcal N)}\lesssim\|f\|_1.
	\end{split}
\end{equation}
It is clear that
\begin{equation*}
	\Big\|\sum_{0<n\le n_2-n_{r_0}}v_n L_nf\Big\|_
	{L_{1,\8}(L_{\infty}(\Omega)\overline{\otimes}\mathcal N)}\le  \Big\|\sum_{0<n\le n_2-n_{r_0}}v_n L_nf\Big\|_
	{L_{1}(L_{\infty}(\Omega)\overline{\otimes}\mathcal N)},
\end{equation*}
and using the triangle inequality for $L_1$ norm $\|\cdot\|_{L_{1}(L_{\infty}(\Omega)\overline{\otimes}\mathcal N)}$, we get
\begin{align*}
	\Big\|\sum_{0<n\le n_2-n_{r_0}}v_n L_nf\Big\|_
	{L_{1}(L_{\infty}(\Omega)\overline{\otimes}\mathcal N)}&\le \|(v_n)_n\|_{\ell^\infty}\sum_{0<n\le n_2-n_{r_0}}\|L_nf\|_1\\
	&\le (n_2-n_{r_0})(D+1)\|(v_n)_n\|_{\ell^\infty}\|f\|_1,
\end{align*}
where the last inequality follows from Lemma \ref{aver} and \eqref{mar}. Hence, \eqref{weak-ineq1} is proved.

By \eqref{weak-ineq1}, we can restrict our analysis to the operator $Lf=\sum_{n>n_2}v_n(M_{\delta^{n}}f-\mathsf{E}_{n}f)$. Furthermore, according to \eqref{weak}, it remains to demonstrate that
\begin{align}
	&\label{weak-eq1}{\varphi}(\chi_{(\lambda,\8)}(|Lb_d|))\lesssim \frac{\|f\|_1}{\lambda}\\
	&\label{weak-eq2}{\varphi}(\chi_{(\lambda,\8)}(|Lb_{\it{off}}|))\lesssim \frac{\|f\|_1}{\lambda}\\
	&\label{weak-eq3}{\varphi}(\chi_{(\lambda,\8)}(|Lg|))\lesssim \frac{\|f\|_1}{\lambda}.
\end{align}

\subsubsection{\bf{Weak type estimate for diagonal terms $b_d$}: \eqref{weak-eq1}}\quad

\medskip

Using the projection $ \zeta$ introduced in Lemma~\ref{supp lemma}, we decompose $Lb_{d}$ in the following way
$$Lb_{d} = (\1_\mathcal{N} - \zeta) Lb_{d} (\1_\mathcal{N} - \zeta) + \zeta Lb_{d} (\1_\mathcal{N}-\zeta)
+ (\1_\mathcal{N}-\zeta)Lb_{d}\zeta + \zeta Lb_{d}\zeta.$$
Lemma~\ref{supp lemma}(i) shows 
\begin{equation}\label{bd-in}
	\begin{split}
		{\varphi} \big(\chi_{(\lambda,\infty)}(|Lb_{d}|)\big)\lesssim&\ \varphi (\1_\mathcal{N}-\zeta) +{\varphi} \big(\chi_{(\lambda/4,\infty)}(|\zeta Lb_{d}\zeta|)\big)\\
		\lesssim&\ \frac{\|f\|_{1}}{\lambda}+{\varphi} \big(\chi_{(\lambda/4,\infty)}(|\zeta Lb_{d}\zeta|)
		\big).
	\end{split}
\end{equation}
Consequently, it suffices to establish the following bound
\begin{align}\label{bd}
	{\varphi} \big(\chi_{(\lambda,\infty)}(|\zeta Lb_{d}\zeta|)\big)\lesssim\frac{ \|f\|_{1}}{\lambda}.
\end{align}

\begin{proof}[Proof of estimate~\eqref{bd}]
Observe that
	\begin{align*}
		{\varphi} \big(\chi_{(\lambda,\infty)}(|\zeta Lb_{d}\zeta|)\big)&\le\frac{\|\zeta Lb_{d}\zeta\|_{L_{1}(L_{\infty}(\Omega)\overline{\otimes}\mathcal N)}}{\lambda}\\
		&\le \frac{\|(v_n)_n\|_{\ell^{\infty}}\sum_{n>n_{2}}\|\zeta (M_{\delta^n}b_d-\mathsf{E}_nb_d)\zeta\|_1}{\lambda}.
	\end{align*}
To complete the proof, it is sufficient to demonstrate the equality \begin{equation}\label{b-d-es}
		\sum_{n>n_{2}}\|\zeta(M_{\delta^n}b_d-\mathsf{E}_nb_d)\zeta\|_1\lesssim \|f\|_1.
	\end{equation}
Before establishing \eqref{b-d-es}, we first claim that
	\begin{equation}\label{canc-1}
		\forall n>n_2,~\zeta\mathsf{E}_nb_{\it{d}}\zeta=0.
	\end{equation}
	Indeed, applying Proposition \ref{CZ}, we write $b_{\it{d}}=\sum_kb_{\it{d},k}=\sum_{k}\big(p_kfp_k-\mathsf{E}_{k+1}(p_kfp_k)\big)$. Hence if we proved that $\forall k,\zeta\mathsf{E}_nb_{\it{d,k}}\zeta=0$, then the claim is proved. To prove this fact, we split $k$ into two cases: $k<n$ and $k\geq n$. For $k<n$, applying the conditional expectation property (see Subsection \ref{nonmar}) directly gives $\mathsf{E}_nb_{\it{d,k}}=\mathsf{E}_n(p_kfp_k)-\mathsf{E}_n(\mathsf{E}_{k+1}(p_kfp_k))=\mathsf{E}_n(p_kfp_k)-\mathsf{E}_{n}(p_kfp_k)=0$. For $k\ge n$, it follows from the definition of $n_2$ that $k>k_2$. By the definition of conditional exception $\mathsf{E}_n$ (see \eqref{con-exp}), we observe that	
	\begin{equation}\label{express-1}
	\begin{split}
		&\zeta(h)\mathsf{E}_n(b_{\textit{d,k}})(h)\zeta(h)\\
		&=\sum_{\alpha\in I_n}\frac{1}{m(Q_\alpha^n)}\int_{Q_\alpha^n}\zeta(h)b_{\textit{d,k}}(g)\zeta(h)dm(g)\chi_{Q_\alpha^n}(h).
		\end{split}
	\end{equation}
Observe that for any  $h\in Q_\a^n$, Proposition \ref{dyadic cube}(iv) combined with a straightforward calculation calculation shows that for all $g\in Q_\a^n$, $d(g,h)\leq d(g,z_\a^n)+d(z_\a^n,h)\leq 2C_1\delta^n$. This directly yields the containment relation $Q_\a^n\subseteq B(h,2C_1\delta^n)$. By invoking Lemma \ref{supp lemma}(ii), we immediately establish the vanishing property $\zeta\mathsf{E}_nb_{\it{d,k}}\zeta=0$, which completes the proof of claim \eqref{canc-1}. 
	
	Applying the claim \eqref{canc-1}, \eqref{b-d-es} reduces to establishing the inequality
\begin{equation}\label{b-d-es-1}
\sum_{n>n_{2}}\|\zeta M_{\delta^n}b_d\zeta\|_1 \lesssim \|f\|_1.
\end{equation}	
Following the decomposition in Proposition \ref{CZ}, we write $b_d=\sum_{k\in\Z}\sum_{\a\in I_k}b_{d,Q_{\a}^k}$. By applying the trigonometric inequality of $L_1$-norm, we demonstrate that the proof of \eqref{b-d-es-1} can be reduced to verifying
\begin{equation}\label{b-d-es-2}	
\sum_{n>n_2}\sum_{k\in\Z}\sum_{\a\in I_{k+n}}\|\zeta M_{\delta^n}b_{d, Q_\a^{k+n}}\zeta\|_1\lesssim\|f\|_1.
\end{equation}
Define
	$$k_3=\min\{k\in\mathbb{N}:~C_1\delta^{k+1}>2~\&~2C_1\delta^{-k+1}<1\},~k_4=\max\{k_2,k_3\}.$$
We will prove that
		\begin{equation}\label{est-bd}
		\|\zeta M_{\delta^n}b_{d, Q_\a^{k+n}}\zeta\|_1\lesssim a(k)\|b_{d,Q_{\a}^{k+n}}\|_1,
	\end{equation}
	where
	\begin{equation*}
		a(k)=
		\left\{
		\begin{array}{ll}
			0, & \hbox{$k>k_4$;} \\
			{D}, & \hbox{$-k_4\le k\le k_4$;} \\
			{D\delta^{\epsilon k}}, & \hbox{$k<-k_4$.}
		\end{array}
		\right.
	\end{equation*}
Temporarily assuming this result, we can immediately derive \eqref{b-d-es-2} by combining \eqref{est-bd} with Lemma~\ref{CZ}(iii). More precisely,
\begin{align*}
\sum_{n>n_2}\sum_{k\in\Z}\sum_{\a\in I_{k+n}}\|\zeta M_{\delta^n}b_{d, Q_\a^{k+n}}\zeta\|_1&\lesssim	\sum_{n>n_2}\sum_{k\in\Z}\sum_{\a\in I_{k+n}}a(k)\|b_{d,Q_{\a}^{k+n}}\|_1\\
&\leq\bigg(\sum_{k\in\Z}a(k)\bigg)\bigg(\sum_{n}\sum_{\a\in I_{n}}\|b_{d,Q_{\a}^{n}}\|_1\bigg)\\
&\lesssim\|f\|_1.
\end{align*}

We now proceed to prove \eqref{est-bd}.\\
	
	{\noindent\bf{Case}} $k>k_4$. From the proof of Lemma~\ref{supp lemma}(ii), we have $\zeta(h)b_{d,Q_\a^{k+n}}(g)\chi_{B(h,2C_1\delta^{k+n+1})}(g)\zeta(h)=0$. This immediately 
	implies
	\begin{align*}
		&\zeta(h)M_{\delta^n}(b_{\textit{d},Q_{\a}^{k+n}})(h)\zeta(h)\\
		&=\frac{1}{m(B(h,\delta^n))}\int_{B(h,\delta^n)}\zeta(h)b_{\textit{d},Q_{\a}^{k+n}}(g)
		\chi_{B(h,2C_1\delta^{k+n+1})}(g)\zeta(h)dm(g)\\
		&=0.
	\end{align*}
	The desired result is  established.\\
		
	{\noindent \bf{Case}} $-k_4\le k\le k_4$. By Lemma~\ref{aver}, we have
	\begin{align*}
		\|\zeta M_{\delta^n}b_{d, Q_\a^{k+n}}\zeta\|_1\le D\|b_{d,Q_{\a}^{k+n}}\|_1.
	\end{align*}\\
	{\noindent\bf{Case}} $k<-k_4$. Set $\mathcal{I}(B(h,r),Q)=\{Q\cap B(h,r): Q\cap\partial B(h,r) \neq\emptyset\}$. Observe that $b_{d,Q_{\a}^{k+n}}$  has support contained in $\widehat{Q}^{k+n}_\alpha$, and satisfies the vanishing property $\int_{\widehat{Q}^{k+n}_\alpha}b_{d,Q_{\a}^{k+n}}=0$. Then
	\begin{equation*}
		M_{\delta^n}(b_{d,Q_{\a}^{k+n}})(h)=\frac{1}{m(B(h,\delta^n))}\int_{\mathcal{I}(B(h,\delta^n),\widehat{Q}^{k+n}_\alpha)}
		b_{d,Q_{\a}^{k+n}}(g)dm(g).
	\end{equation*}
	Note that $\mathcal{I}(B(h,\delta^n),\widehat{Q}^{k+n}_\alpha)\subseteq\{g:\delta^n-2C_1\delta^{n+k+1}\le d(g,h)\le\delta^n+2C_1\delta^{n+k+1}\}$. It follows from~\eqref{ineq1} that
	\begin{align*}
		&\|\zeta M_{\delta^n}(b_{d,Q_{\a}^{k+n}})\zeta\|_{L_1(\N)}\le \| M_{\delta^n}(b_{d,Q_{\a}^{k+n}})\|_{L_1(\N)}\\
		&\le \int_{\G}\frac{1}{m(B(h,\delta^n))}\int_{\delta^n-2C_1\delta^{n+k+1}\le d(g,h)\le\delta^n+2C_1\delta^{n+k+1}}\|b_{d,Q_{\a}^{k+n}}(g)\|_{L_1(\M)}
		dm(g)dm(h)\\
		&\lesssim D\delta^{\epsilon k}\|b_{d,Q_{\a}^{k+n}}\|_{L_1(\N)},
	\end{align*}
	which completes the proof.
\end{proof}

\medskip

\subsubsection{\bf{Weak type estimate for off-diagonal terms $b_{\it{off}}$: \eqref{weak-eq2}}}\quad

\medskip

Our objective is to establish the inequality
\begin{align}\label{b-off}
	{\varphi} \big(\chi_{(\lambda,\infty)}(|\zeta Lb_{\it{off}}\zeta|)\big)\lesssim\frac{ \|f\|_{1}}{\lambda}.
\end{align}
We are now prepared to prove the estimate \eqref{b-off}.
\begin{proof}[Proof of estimate~\eqref{b-off}]
Similar to the estimate of $\eqref{bd}$, it suffices to show
\begin{equation}\label{b-off-es}
		\sum_{n>n_{2}}\|\zeta(M_{\delta^n}b_{\it{off}}-\mathsf{E}_nb_{\it{off}})\zeta\|_1\lesssim \|f\|_1.
	\end{equation}	
Prior to establishing \eqref{b-off-es}, we introduce an analogous claim similar to \eqref{canc-1}
\begin{equation}\label{canc-2}
\forall n>n_2,~\zeta\mathsf{E}_n b_{\text{off}}\zeta=0.
\end{equation}
This claim follows directly from the proof of \eqref{canc-1} by decomposing $b_{\text{off}}$ as $b_{\it{off}}=
\sum_{k}b_{\textit{off},k}=\sum_{k}\big(q_kfp_k+p_kfq_k\big)$, where the detailed verification is omitted for brevity.

Given the expression $b_{\it{off}}=\sum_kb_{\it{off},k}=\sum_{k}\big(q_kfp_k+p_kfq_k\big)$ and Claim \eqref{canc-2}, by using a similar approach as in \eqref{b-d-es-1}, we conclude that the estimate \eqref{b-off-es} reduces to proving the inequality
\begin{equation}\label{l-boff-1}
\sum_{n>n_2}\sum_{k\in\Z}\sum_{\a\in I_{k+n}}\|\zeta M_{\delta^n}b_{\it{off},k+n}\zeta\|_1\lesssim\|f\|_1.
\end{equation}
Observing that \eqref{decom} implies
\begin{equation}\label{decom-ineq}
	\sum_n\|p_nfp_n\|_1\le\|f\|_1,
\end{equation}
we conclude that to prove \eqref{l-boff-1}, it suffices to show
	\begin{equation}\label{b-off-es1}
		\|\zeta M_{\delta^n}b_{\it{off},k+n}\zeta\|_1\lesssim  a(k)\lambda^{\frac12} \|p_{k+n}\|^{\frac12}_1\|p_{k+n}fp_{k+n}\|^{\frac12}_1,
	\end{equation}
	where
	\begin{equation*}
		a(k)=\left\{
		\begin{array}{ll}
			0, & \hbox{$k>k_4$;}\\
			D(1+\delta^{k_4})^{\epsilon}, & \hbox{$-k_4\le k\le k_4$;} \\
			D\delta^{\epsilon k}, & \hbox{$k<-k_4$.}
		\end{array}
		\right.
	\end{equation*}
Assume this result momentarily. Combining \eqref{decom} and \eqref{weak-est-marti} with inequality~\eqref{decom-ineq}, one can immediately obtain
\begin{align*}
	\sum_{n>n_2}\sum_{k\in\Z}\|\zeta M_{\delta^n}b_{\it{off},k+n}\zeta\|_1&\lesssim \sum_{n>n_2}\sum_{k\in\Z}a(k)\lambda^{\frac12} \|p_{k+n}\|^{\frac12}_1\|p_{k+n}fp_{k+n}\|^{\frac12}_1\\
	&\leq\lambda^{\frac12}\bigg(\sum_{n>n_2}\sum_{k\in\Z} a(k)\|p_{k+n}\|_1\bigg)^{\frac12}\bigg(\sum_{n>n_2}\sum_{k\in\Z} a(k)\|p_{k+n}fp_{k+n}\|_1\bigg)^{\frac12}\\
&\lesssim \lambda^{\frac12} \frac{\|f\|^{\frac12}_1}{\lambda^\frac12}\|f\|^\frac12_1\\
&=\|f\|_1. 
\end{align*}
	
We now proceed to prove~\eqref{b-off-es1}.\\

{\noindent \bf{Case}} $k>k_4$. By Proposition \ref{CZ}(iii), we express $b_{\textit{off},k+n}=\sum_{\a\in I_{k+n}}b_{\textit{off},Q_{\a}^{k+n}}$, noting that for each $b_{\textit{off},Q_{\a}^{k+n}}$, 
\begin{align*}
		\zeta(h)M_{\delta^n}(b_{\textit{off},Q_{\a}^{k+n}})(h)\zeta(h)=\frac{1}{m(B(h,\delta^n))}\int_{B(h,\delta^n)}\zeta(h)b_{\textit{off},Q_{\a}^{k+n}}(g)
		\chi_{B(h,2C_1\delta^{k+n+1})}(g)\zeta(h)dm(g).
	\end{align*}
By applying Lemma~\ref{supp lemma}(ii), we obtain the identity $\zeta(h)b_{\it{off},Q_\a^{k+n}}(g)\chi_{B(h,2C_1\delta^{k+n+1})}(g)\zeta(h)=0$. Combining this with the previous result, we derive $\zeta(h)M_{\delta^n}(b_{\textit{off},Q_{\a}^{k+n}})(h)\zeta(h)=0$, which consequently implies $\zeta M_{\delta^n}b_{\it{off},k+n}\zeta=0$. Thus, the desired result is established.\\
	
	{\noindent \bf{Case}} $-k_4\le k\le k_4$. Recall that $b_{\textit{off},k+n}=q_{k+n}fp_{k+n}+p_{k+n}fq_{k+n}$, and by symmetry, we only need to consider the term $ \|\zeta M_{\delta^{n}}(p_{k+n}fq_{k+n})\zeta\|_1$. Note that
	\begin{equation*}
		M_{\delta^{n}}(p_{k+n}fq_{k+n})(w)=\frac{1}{m(B(h,\delta^n))}\int_{\mathcal{I}_1(B(h,\delta^n),k+n)}
		\mathsf{E}_{k+n}(p_{k+n}f\chi_{B(w,\delta^n)}q_{k+n})(g)dm(g).
	\end{equation*}
	Set $g(y)=f(y)\chi_{B(w,\delta^n)}(y)$, then by \eqref{L-ineq} for $p=1$, we have
	\begin{equation}\label{off-1}
		\|M_{\delta^{n}}(p_{k+n}fq_{k+n})\|_1\le D(1+\delta^{k_4})^{\epsilon}\|p_{k+n}g_{k+n}q_{k+n}\|_1.
	\end{equation}
	Applying the fact that the conditional expectation $\mathsf{E}_n$ is a positive operator and Lemma~\ref{Cu}(ii), we get
	\begin{equation}\label{boff-2}
		\begin{split}
			|q_{{k+n}}g_{{k+n}}p_{{k+n}}|&\le (p_{{k+n}}g_{{k+n}}p_{{k+n}})^{\frac12}\|q_{{k+n}}g_{{k+n}}q_{{k+n}}\|_{\M}^{\frac12}\\
			&\le (p_{{k+n}}f_{{k+n}}p_{{k+n}})^{\frac12}\|q_{{k+n}}f_{{k+n}}q_{{k+n}}\|^{\frac12}_{\M}\\
			&\le\lambda^{\frac12} (p_{{k+n}}f_{{k+n}}p_{{k+n}})^{\frac12}
		\end{split}
	\end{equation}
From the equality $\|p_{k+n}g_{k+n}q_{k+n}\|_1=\|q_{k+n}g_{k+n}p_{k+n}\|_1$,  together with \eqref{off-1}, \eqref{boff-2} and Proposition \ref{Holder-ineq}, we derive 	
\begin{align*}
		\|\zeta M_{\delta^{n}}(p_{k+n}fq_{k+n})\zeta\|_1&\le D(1+\delta^{k_4})^{\epsilon}\|p_{k+n}g_{k+n}q_{k+n}\|_1\\
		&\leq D(1+\delta^{k_4})^{\epsilon}\lambda^{\frac12} \|p_{{k+n}}f_{{k+n}}p_{{k+n}}\|^{\frac12}_{\frac12}\\
		&\leq D(1+\delta^{k_4})^{\epsilon}\lambda^{\frac12} \|p_{k+n}\|^{\frac12}_1\|p_{{k+n}}f_{{k+n}}p_{{k+n}}\|^{\frac12}_{1}.
			\end{align*}
By the trace-preserving property of $\mathsf{E}_{k+n}$, we have $\|p_{{k+n}}f_{{k+n}}p_{{k+n}}\|_1=\|p_{{k+n}}fp_{{k+n}}\|_1$, thereby proving the desired estimate.\\
	
	{\noindent \bf{Case}} $k< -k_4$. By extending the argument used for the case $-k_4\le k\le k_4$, we can adapt the proof by substituting \eqref{L-ineq} with \eqref{L-ineq1} for $p=1$, which ultimately yields
		\begin{align*}
		\| M_{\delta^{n}}(q_{k+n}fp_{k+n})\|_1\le D\delta^{\epsilon k}\|q_{k+n}g_{k+n}p_{k+n}\|_1.
	\end{align*}
	Then \eqref{b-off-es1} is proved.
		\end{proof}
	
	\subsubsection{\bf{Weak type estimate for the good function $g$}: \eqref{weak-eq3}}\quad
	
	\medskip
	To establish the estimate for the good part, we require the following proposition. 
	\begin{prop}\label{L-2}
		Let $h\in L_{2}(\mathcal N)$. Then
		$$\|Lh\|_{L_{2}(L_{\infty}(\Omega)\overline{\otimes}\mathcal N)}\lesssim\|h\|_{2}.$$
	\end{prop}
	The proof strategy employed here closely follows ~\cite[Lemma 3.1]{Xu23}. It should be emphasized that the boundary conditions stipulated in Propositions \ref{boundary} and \ref{measure estimate of cube} constitute indispensable elements of this proof framework. For brevity, we omit the detailed proof steps and refer to the referenced works.

	\begin{proof}[Proof of estimate \eqref{weak-eq3}]
		Using the Chebychev inequality, Proposition~\ref{L-2} and Proposition~\ref{CZ}(ii), we have
		$${\varphi}(\chi_{(\lambda/2,\infty)}(|Lg|))\leq\frac{\|Lg\|^{2}_{2}}{\lambda^{2}}
		\lesssim\frac{\|g\|_{2}^{2}}{\lambda^{2}}\lesssim\frac{\|f\|_{1}}{\lambda},$$
		which completes the proof.
	\end{proof}

\medskip

As noted earlier, by Lemma \ref{kin1-weak}, the proof of Theorem \ref{short-ineq}(i) largely follows the framework of Theorem \ref{long-ineq}(i) with slight adjustments for the specific assumptions. However, the treatment of the term $b_{\it{off}}$ differs significantly from that of \eqref{weak-eq2}. Below we present the detailed derivation and outline the key steps.

\subsection{Proof of Theorem \ref{short-ineq}(i)}\quad

\medskip

Define $\mathcal{S}_n=\{i: [r_i,r_{i+1})\subseteq [\delta^n,\delta^{n+1})\}$. It is clear that $\mathcal{S}=\cup_{n>n_{r_0}}\mathcal{S}_n$. Then using Lemma~\ref{short-in}, one can see that
\begin{equation*}
	\sum_{n_{r_0}<n\le n_{2}}\|\sum_{i\in\mathcal{S}_n}\varepsilon_i T_if\|_{L_{1}(L_{\infty}(\Omega)\overline{\otimes}\mathcal N)}\le (n_2-n_{r_0})D\|f\|_1.
\end{equation*}
Based on the preceding estimation, it suffices to consider the operator $\mathcal{T}f=\sum_{n>n_2}\sum_{i\in\mathcal S_n}\varepsilon_i T_if$. Moreover, by~\eqref{weak}, the proof will be complete once we prove
\begin{align}
	&\label{weak-seq1}\widetilde{\varphi}(\chi_{(\lambda,\8)}(|\mathcal{T}b_d|))\lesssim \frac{\|f\|_1}{\lambda}\\
	&\label{weak-seq2}\widetilde{\varphi}(\chi_{(\lambda,\8)}(|\mathcal{T}b_{\it{off}}|))\lesssim \frac{\|f\|_1}{\lambda}\\
	&\label{weak-seq3}\widetilde{\varphi}(\chi_{(\lambda,\8)}(|\mathcal{T}g|))\lesssim \frac{\|f\|_1}{\lambda}.
\end{align}

We begin with the term \eqref{weak-seq1}.

\subsubsection{\bf{Weak type estimate for diagonal terms $b_d$: \eqref{weak-seq1}}}\quad

\medskip

Observe that \eqref{weak-seq1} reduces to
\begin{equation*}
	\frac{\|\sum_{i\in\mathcal S}\varepsilon_i\zeta T_ib_{d}\zeta\|_{L_{1}(L_{\infty}(\Omega)\overline{\otimes}\mathcal N)}}{\lambda}\lesssim \|f\|_1.
\end{equation*}
Furthermore, by applying the triangle inequality for the $L_1$ norm $\|\cdot\|_{L_{1}(L_{\infty}(\Omega)\overline{\otimes}\mathcal N)}$, it suffices to show
\begin{equation}\label{bd-short}
	\frac{\sum_{n>n_2}\|\sum_{i\in\mathcal S_n}\varepsilon_i\zeta T_ib_{d}\zeta\|_{L_{1}(L_{\infty}(\Omega)\overline{\otimes}\mathcal N)}}{\lambda}\lesssim \|f\|_1.
\end{equation}
\begin{proof}[Proof of estimate~\eqref{bd-short}]
	
Following analogous reasoning to that in~\eqref{bd}, to establish~\eqref{bd-short}, it is sufficient to verify
	\begin{equation}\label{bad-short}
		\sum_{i\in\mathcal{S}_n}\|\zeta(M_{r_i}b_{d,Q_{\a}^{k+n}}-M_{r_{i+1}}b_{d,Q_{\a}^{k+n}})\zeta\|_{1}\lesssim a(k)\|b_{d,Q_{\a}^{k+n}}\|_1,
	\end{equation}
	where
	\begin{equation*}
		a(k)=
		\left\{
		\begin{array}{ll}
			0, & \hbox{$k>k_4$;} \\
			{D\delta^{\epsilon}}, & \hbox{$-k_4\le k\le k_4$;} \\
			{D \delta^{\epsilon k}}, & \hbox{$k<-k_4$.}
		\end{array}
		\right.
	\end{equation*}\\
We now proceed to establish \eqref{bad-short}.\\

	{\noindent \bf{Case}} $k>k_4$. Following the proof for the case $k>k_4$ in estimate \eqref{est-bd}, we obtain that for all $r\in[\delta^n,\delta^{n+1}]$, $\zeta M_{r}b_{d,Q_{\a}^{k+n}}\zeta=0$, which consequently gives the estimate: $\forall i\in\mathcal{S}_n$, $\zeta(M_{r_i}b_{d,Q_{\a}^{k+n}}-M_{r_{i+1}}b_{d,Q_{\a}^{k+n}})\zeta=0$.\\
	
	{\noindent \bf{Case}} $-k_4\le k\le k_4$. The desired estimate is a direct consequence of Lemma \ref{short-in}.\\
	
	{\noindent \bf{Case}} $k<-k_4$. Since $b_{d,Q_\a^{k+n}}$ is supported on $\widehat{Q}_\a^{k+n}$ and $\int_{\widehat{Q}_\a^{k+n}}b_{d,Q_\a^{k+n}}=0$, so
	\begin{equation*}
		M_{r_i}b_{d,Q_\a^{k+n}}(h)=\int_{\mathcal{I}(B(h,r_i),\widehat{Q}^{k+n}_\alpha)}b_{d,Q_{\a}^{k+n}}(g)dm(g).
	\end{equation*}
	Additionally, \eqref{aver-in} yields the inequality that
	\begin{equation}\label{short-bd}
		\begin{split}
			&\|\zeta(h)(M_{r_i}b_{d,Q_{\a}^{k+n}}-M_{r_{i+1}}b_{d,Q_{\a}^{k+n}})(h)\zeta(h)\|_{L_1(\M)}\\
			&\le \|M_{r_i}b_{d,Q_{\a}^{k+n}}(h)-M_{r_{i+1}}b_{d,Q_{\a}^{k+n}}(h)\|_{L_1(\M)}\\
			&\le\Big(\frac{1}{m(B(h,r_i))}-\frac{1}{m(B(h,r_{i+1}))}\Big)
			\int_{\mathcal{I}(B(h,r_i),\widehat{Q}^{k+n}_\alpha)}\|b_{d,Q_{\a}^{k+n}}(g)\|_{L_1(\M)}dm(g)\\
			&\ \ \ +\frac{1}{m(B(h,r_{i+1}))}\int_{\mathcal{I}(B(h,r_i)\setminus B(h,r_{i-1}),\widehat{Q}^{n+k}_\alpha)}\|b_{d,Q_{\a}^{k+n}}(g)\|_{L_1(\M)}dm(g).
		\end{split}
	\end{equation}
	Let $i_0\in\mathcal{S}_n$ such that $\mathcal{I}(B(h,r_i),\widehat{Q}^{n+k}_\alpha)\neq \emptyset$ (if such an $i_0$ does not exist then the desired estimate is valid). Clearly, $\cup_{i:\mathcal{I}(B(h,r_i),\widehat{Q}^{n+k}_\alpha)\neq \emptyset}\mathcal{I}(B(h,r_i),\widehat{Q}^{n+k}_\alpha)\subseteq B(h,r_{i_0}+2C_1\delta^{n+k+1})\setminus B(h,r_{i_0}-2C_1\delta^{n+k+1})$ and $\mathcal{I}(B(h,r_i),\widehat{Q}^{k+n}_\alpha) \subseteq B(h,r_{i_0}+2C_1\delta^{n+k+1})\setminus B(h,r_{i_0}-2C_1\delta^{n+k+1})$. Summing over all $i$ in \eqref{short-bd}, we get
	\begin{equation*}
		\begin{split}
			&\sum_{i\in\mathcal{S}_n}\|\zeta(M_{r_i}b_{d,Q_{\a}^{k+n}}-M_{r_{i+1}}b_{d,Q_{\a}^{k+n}})\zeta\|_{L_1(\N)}\\
			&\lesssim\int_G\frac{1}{m(B(h,\delta^n))}\int_{B(h,r_{i_0}+2C_1\delta^{n+k+1})\setminus B(h,r_{i_0}-2C_1\delta^{n+k+1})}\|b_{d,Q_{\a}^{k+n}}(g)\|_{L_1(\M)}dm(g)dm(h)\\
			&\lesssim D\delta^{\epsilon k}\|b_{d,Q_{\a}^{k+n}}\|_{L_1(\N)},
		\end{split}
	\end{equation*}
	where the last inequality follows from \eqref{int} and the proof of~\eqref{ineq1}. This completes the proof.
\end{proof}
\subsubsection{\bf{Weak type estimate for off-diagonal terms $b_{\it{off}}$}: \eqref{weak-seq2}}\quad

\medskip

The approach for \eqref{weak-seq2} differs significantly from \eqref{weak-eq2}. In contrast to \eqref{weak-eq2}, our analysis employs the $L_2$-norm method. The proof hinges on the following version of the almost orthogonality principle (see, e.g., \cite{Hong-Liu21,HM1}).
	\begin{lem}\label{L1}
	Let $(T_{n,i})_{n\in\mathbb N,i\in\mathbb Z}$ be a sequence of bounded linear operators on $L_2(\mathcal{N})$. Let $(u_{k})_{k\in\Z}$ and $(v_{k})_{k\in\Z}$ are two sequences of operators in $L_{2}(\mathcal{N})$ such that $f=\sum_ku_k$ and $\sum_k\|v_k\|_2^2<\8$. Assume that for every $n> N$, there exists a positive sequence $(\sigma(k))_{k\in\mathbb N}$ with $\sigma=\sum_{k\in\mathbb Z}a(k)<\8$ such that
	\begin{equation*}
		\|(T_{n,i}(u_{k+n}))_i\|^2_{L_2(\N;\ell_2^{rc})}\le a(k)\|v_{k+n}\|_1,
	\end{equation*}
	then
	\begin{equation*}
		\sum_{n>N}\|(T_{n,i}(f))_i\|^2_{L_2(\N;\ell_2^{rc})}\le \sigma^2\sum_{k\in\mathbb Z}\|v_k\|_1.
	\end{equation*}
\end{lem}
\begin{proof}
	With the conditions in hand, using the triangle inequality for $L_2(\mathcal{N};\ell_{2}^{rc})$-norm and the Young inequality for $\ell_{2}$-norm, we have
	\begin{align*}
		\sum_{n>N}\|(T_{n,i}(f))_i\|^2_{L_2(\N;\ell_2^{rc})}\leq& \sum_{n>N}\bigg(\sum_{k\in\Z}\|(T_{n,i}(u_{k+n}))_{i}\|_{L_2(\mathcal{N};\ell_{2}^{rc})}\bigg)^{2}\\
		\leq&\sum_{n>N}\bigg(\sum_{k\in\Z}a(k)^{\frac12}\|v_{k+n}\|^{\frac12}_{1}\bigg)^{2}\\
		\leq&\ \bigg(\sum_{k\in\Z}a(k)\bigg)^{2}\bigg(\sum_{k\in\Z}\|v_{k}\|_{1}\bigg),
	\end{align*}
	which finishes the proof.
\end{proof}

We now proceed to prove~\eqref{weak-seq2}.

\begin{proof}[Proof of \eqref{weak-seq2}]
Applying Chebyshev's inequality along with the orthogonality of $\varepsilon_{i}$, \eqref{weak-seq2} reduces to showing that
\begin{equation*}
	\begin{split}
		\widetilde{\varphi}\big(\chi_{(\lambda,\infty)}(|\zeta Tb_{\it{off}}\zeta|)\big)&\le\frac{\sum_{n>n_2}\sum_{i\in\mathcal S_n}\|\zeta T_ib_{\it{off}}\zeta\|^2_2}{\lambda^2}.
	\end{split}
\end{equation*}
It suffices to show
	\begin{equation}\label{s-boff-1}
		\sum_{n>n_2}\sum_{i\in\mathcal S_n}\|\zeta T_ib_{\it{off}}\zeta\|^2_2\lesssim \lambda\sum_n\|p_nfp_n\|_1
	\end{equation}
	since by \eqref{decom}
	\begin{equation*}
		\sum_n\|p_nfp_n\|_1\le\|f\|_1.
	\end{equation*}

To establish \eqref{s-boff-1}, as in Proposition \ref{CZ}, we first express $b_{\it{off}}=\sum_kb_{\it{off},k}=\sum_{k}\big(q_kfp_k+p_kfq_k\big)$. Letting $(T_{n,i})_{n,i}=(M_{r_{i+1}}-M_{r_i})_{i\in\mathcal{S}_n}$, $N=n_2$, $u_k=b_{\it{off},k}$ and $v_k=p_kfp_k$ in Lemma~\ref{L1}, it is sufficient to show that
	\begin{equation}\label{short-boff}
		\sum_{i\in\mathcal{S}_n}\|M_{r_{i+1}}b_{\it{off},k+n}-M_{r_i}b_{\it{off},k+n}\|_2^2\le \lambda a(k)\|p_{k+n}fp_{k+n}\|_1,
	\end{equation}
	where
	\begin{equation*}
		a(k)=\left\{
		\begin{array}{ll}
			0, & \hbox{$k>k_4$;} \\
			D(1+\delta^{k_4})^{\epsilon}, & \hbox{$-k_4\le k\le k_4$;} \\
			D\delta^{\epsilon k}, & \hbox{$k<-k_4$.}
		\end{array}
		\right.
	\end{equation*}
We now turn to proving \eqref{short-boff}.\\

	{\noindent \bf{Case}} $k>k_4$. The desired estimate follows directly from the proof of the case $k>k_4$ presented in estimate \eqref{b-off-es1}.\\
	
	{\noindent \bf{Case}} $-k_4\le k\le k_4$. Recall that $b_{\textit{off},k+n}=q_{k+n}fp_{k+n}+p_{k+n}fq_{k+n}$, by symmetry, we only consider to the term $p_{k+n}fq_{k+n}$. According to \eqref{aver-in}, we first set
	\begin{align*}
		&M^1_{i}f(h)=\Big(\frac{1}{m(B(h,r_i))}-\frac{1}{m(B(h,r_{i+1}))}\Big)\int_{B(h,r_i)}f(g)dm(g)\\
		& M^2_{i}f(h)=\frac{1}{m(B(h,r_{i+1}))}\int_{B(h,r_{i+1})\setminus B(h,r_i)}f(g)dm(g).
	\end{align*}
	Therefore,
	\begin{align*}
		\|M_{r_{i+1}}p_{k+n}fq_{k+n}-M_{r_i}p_{k+n}fq_{k+n}\|_2^2\le 2\|M^1_{i}p_{k+n}fq_{k+n}\|_2^2+2\|M^2_{i}p_{k+n}fq_{k+n}\|^2_2.
	\end{align*}
	We first focus on the term $\|M^1_{i}p_{k+n}fq_{k+n}\|_2^2$. Observing that
	\begin{equation*}
		\int_{B(h,r_i)}p_{k+n}(g)f(g)q_{k+n}(g)dm(g)=\int_{\mathcal{I}_1(B(h,r_i),k+n)}
		\mathsf{E}_{k+n}(p_{k+n}f\chi_{B(h,r_i)}q_{k+n})(g)dm(g),
	\end{equation*}
	and by analogy with \eqref{boff-2}, we derive
	\begin{equation*}
			\begin{split}
			|q_{{k+n}}g_{{k+n}}p_{{k+n}}|^2&\le p_{{k+n}}g_{{k+n}}p_{{k+n}}\|q_{{k+n}}g_{{k+n}}q_{{k+n}}\|_{\M}\\
			&\le p_{{k+n}}f_{{k+n}}p_{{k+n}}\|q_{{k+n}}f_{{k+n}}q_{{k+n}}\|_{\M}\\
			&\le\lambda p_{{k+n}}f_{{k+n}}p_{{k+n}}.
		\end{split}
	\end{equation*}
Combining \eqref{meas} and \eqref{short1} with the observations above, we derive
	\begin{equation}\label{sh-boff}
		\begin{split}
			&\frac{1}{m(B(h,\delta^n))}\Big\|\int_{B(h,r_i)}p_{k+n}(g)f(g)q_{k+n}(g)dm(g)\Big\|^2_{L_{2}(\M)}\\
			&\lesssim \lambda(1+\delta^{k_4})^\epsilon\int_{\mathcal{I}_1(B(h,r_i),k+n)}\|\mathsf{E}_{k+n}(p_{k+n}f\chi_{B(w,r_i)}p_{k+n})(g)\|_{L_1(\M)}dm(g).
		\end{split}
	\end{equation}
	Since $f$ is positive and $\mathcal{I}_1(B(h,r_i),k+n)\in \mathcal{F}_{k+n}$, then by the definition of $\mathsf{E}_{k+n}$, one can see that
	\begin{equation}\label{sh-boff1}
		\begin{split}
			&\int_{\mathcal{I}_1(B(h,r_i),k+n)}\|\mathsf{E}_{k+n}(p_{k+n}f\chi_{B(h,r_i)}p_{k+n})(g)\|_{L_1(\M)}dm(g)\\
			&\leq \int_{B(h,r_i)}p_{k+n}(g)f(g)p_{k+n}(g)dm(g).
		\end{split}
	\end{equation}
The combination of \eqref{int}, \eqref{sh-boff}, and \eqref{sh-boff1} leads to
\begin{align*}
		\|M^1_{i}p_{k+n}fq_{k+n}(h)\|^2_{L_2(\M)}&\lesssim \lambda(1+\delta^{k_4})^\epsilon\Big(\frac{1}{m(B(h,r_i))}-\frac{1}{m(B(h,r_{i+1}))}\Big)\\
		&\times\int_{B(h,\delta^{n+1})}\|p_{k+n}(g)f(g)p_{k+n}(g)\|_{L_1(\M)}dm(g).
\end{align*}
	Summing $i\in\mathcal{S}_n$ over the above inequality, and applying \eqref{int} and Lemma \ref{aver}, we obtain
	\begin{equation*}
		\sum_{i\in\mathcal{S}_n}\|M^1_{i}p_{k+n}fq_{k+n}\|^2_{2}\lesssim \lambda D(1+\delta^{k_4})^\epsilon\|p_{k+n}fp_{k+n}\|_1.
	\end{equation*}
	The preceding arguments remain valid when applied to the term $\|M^2_{i}p_{k+n}fq_{k+n}\|_2^2$. Let us explain it briefly. The main observations are $\mathcal{I}_1(B(h,r_i)\setminus B(h,r_{i-1}),k+n)\subseteq \mathcal{I}_1(B(h,r_i),k+n)\cup \mathcal{I}_1(B(h,r_{i-1}),k+n)$ and
	\begin{align*}
		&\frac{1}{m(B(h,\delta^n))}\Big\|\int_{B(h,r_i)\setminus B(h,r_{i-1})}p_{k+n}(g)f(g)q_{k+n}(g)dm(g)\Big\|^2_{L_{2}(\M)}\\
		&\lesssim \lambda(1+\delta^{k_4})^\epsilon\int_{\mathcal{I}_1(B(h,r_i)\setminus B(h,r_{i-1}),k+n)}\|\mathsf{E}_{k+n}(p_{k+n}f\chi_{B(h,r_i)\setminus B(h,r_{i-1})}p_{k+n})(g)\|_{L_1(\M)}dm(g).
	\end{align*}
	So, summing $i\in\mathcal{S}_n$ over the above inequality, we have
	\begin{equation*}
		\sum_{i\in\mathcal{S}_n}\|M^2_{i}p_{k+n}fq_{k+n}\|^2_{2}\lesssim \lambda D(1+\delta^{k_4})^\epsilon\|p_{k+n}fp_{k+n}\|_1.
	\end{equation*}
	The desired estimate is proved.\\
	
	{\noindent \bf{Case}} $k<-k_4$. The argument for the case $-k_4\le k\le k_4$ remains valid here. Observing that using \eqref{meas1}, \eqref{short1} and \eqref{boff-2}, we obtain
		\begin{align*}
		&\frac{1}{m(B(h,\delta^n))}\Big\|\int_{B(h,r_i)}p_{k+n}(y)f(g)q_{k+n}(g)dm(g)\Big\|^2_{L_{2}(\M)}\\
		&\lesssim \lambda\delta^{k\epsilon}\int_{\mathcal{I}_1(B(h,r_i),k+n)}\|\mathsf{E}_{k+n}(p_{k+n}f\chi_{B(h,r_i)}p_{k+n})(g)\|_{L_1(\M)}dm(g).
	\end{align*}
	and
	\begin{align*}
		&\frac{1}{m(B(h,\delta^n))}\Big\|\int_{B(h,r_i)\setminus B(h,r_{i-1})}p_{k+n}(y)f(g)q_{k+n}(g)dm(g)\Big\|^2_{L_{2}(\M)}\\
		&\lesssim \lambda\delta^{k\epsilon}\int_{\mathcal{I}_1(B(h,r_i)\setminus B(h,r_{i-1}),k+n)}\|\mathsf{E}_{k+n}(p_{k+n}f\chi_{B(h,r_i)\setminus B(h,r_{i-1})}p_{k+n})(g)\|_{L_1(\M)}dm(g).
	\end{align*}
	We have therefore established \eqref{short-boff}, which completes the proof.\end{proof}

\subsubsection{\bf{Weak type estimate for the good function $g$}: \eqref{weak-seq3}}\quad

Before Proving \eqref{weak-seq3}, we require the following proposition.
\begin{prop}\label{S-2}
		Let $h\in L_{2}(\mathcal N)$. Then
		$$\|\mathcal{T}h\|_{L_{2}(L_{\infty}(\Omega)\overline{\otimes}\mathcal N)}\lesssim\|h\|_{2}.$$
	\end{prop}
This proposition corresponds to the Hilbert-valued square function inequalities, and the proof follows an analogous approach to the classical case (see e.g. \cite{Hong-Liu21}). We omit the details for brevity.
\begin{proof}[Proof of estimate \eqref{weak-seq3}]

Following the approach for \eqref{weak-eq3}, application of Chebyshev's inequality and proposition \ref{S-2} gives
 $$\widetilde{\varphi}(\chi_{(\lambda,\8)}(|\mathcal{T}g|))\leq\frac{\|\mathcal{T}g\|^{2}_{L_{2}(L_{\infty}(\Omega)\overline{\otimes}\mathcal N)}}{\lambda^{2}}=\frac{\bigg\|\bigg(\sum_{n>n_2}\sum_{i\in\mathcal S_n}|T_if|^2\bigg)^{\frac12}\bigg\|^2_2}{\lambda^{2}}\lesssim\frac{\|g\|^2_2}{\lambda^{2}}\lesssim\frac{\|f\|_{1}}{\lambda}.$$
 
 \end{proof}

\section{Proof of Theorem~\ref{long-ineq} and \ref{short-ineq}: $(L_{\infty},\mathrm{BMO})$ and strong type $(p,p)$ estimates}\label{BMO}

In this section, we show the $(L_\8,\mathrm{BMO})$ and strong type $(p,p)$ estimates stated in Theorem~\ref{long-ineq} and \ref{short-ineq}, respectively.

\subsection{$(L_{\infty},\mathrm{BMO})$ estimate}\label{BMO-est}\quad

\medskip

Let $\mathcal A$ be the von Neumann algebra equipped with the  trace $\varphi$. Let $(\N_k)_{k\in\Z}$ be a filtration of $\N$ given by Subsection~\ref{nonmar} . It is clear that $(\N_k)_{k\in\Z}$ is a filtration of $\mathcal A$. By abuse of notation, we continue to let $\mathsf{E}_k$ stand for the conditional expectation with respect to $\mathcal{A}_k$. Let us recall the definition of noncommutative $\mathrm{BMO}$ space associated with the filtration $(\mathcal{A}_k)_{k\in\Z}$.
\begin{align*}
  &\mathrm{BMO}_c(\mathcal A)=\{f\in L_\infty(\mathcal{A}):\|f\|_{\mathrm{BMO}_c(\mathcal A)}=\sup_{k}\|\mathsf{E}_k(|f-\mathsf{E}_{k+1}f|^2)\|_{\8}^{1/2}<\8\}\\
&\mathrm{BMO}_r(\mathcal A)=\{f\in L_\infty(\mathcal{A}):\|f\|_{\mathrm{BMO}_r(\mathcal A)}=\|f^*\|_{\mathrm{BMO}_c(\mathcal A)}<\8\}.
\end{align*}
The space $\mathrm{BMO}(\mathcal A)$ is defined as
\begin{equation*}
  \mathrm{BMO}(\mathcal A)=\mathrm{BMO}_c(\mathcal A)\cap \mathrm{BMO}_r(\mathcal A),
\end{equation*}
equipped with the intersection norm
\begin{equation}\label{bmo}
	\|f\|_{\mathrm{BMO}(\mathcal A)}=\max\{\|f\|_{\mathrm{BMO}_c(\mathcal A)},\|f\|_{\mathrm{BMO}_r(\mathcal A)}\}.
\end{equation}

A direct computation establishes the following expressions for the $\mathrm{BMO}$ norms
\begin{equation}\label{BMO-Norm}
\begin{split}
  &\|f\|_{\mathrm{BMO}_c(\mathcal A)}=\sup_{Q_\a^k\in\mathcal{F}}\Big\|\frac{1}{m(Q_\a^k)}\int_{Q_\a^k}\Big|f(h) -f_{\widehat{Q}_\alpha^{k}}\Big|^2dm(h)\Big\|^{\frac12}_{\M},\\
&\|f\|_{\mathrm{BMO}_r(\mathcal A)}=\sup_{Q_\a^k\in\mathcal{F}}\Big\|\frac{1}{m(Q_\a^k)}\int_{Q_\a^k}\Big|f^*(h) -f^*_{\widehat{Q}_\alpha^{k}}\Big|^2dm(h)\Big\|^{\frac12}_{\M}.
\end{split}
\end{equation}

We now proceed to prove Theorem~\ref{long-ineq}(ii).

\begin{proof}[Proof of~Theorem~\ref{long-ineq}(ii)]
Let $f\in L_{\8}(\mathcal A)$. By the definition of $\mathrm{BMO}$ space, it is enough to establish
\begin{equation}\label{BMO-C}
	\|Lf\|_{\mathrm{BMO}_c(\mathcal A)}\lesssim\|f\|_{\8}
\end{equation}
and
\begin{equation}\label{BMO-R}
	\|Lf\|_{\mathrm{BMO}_r(\mathcal A)}\lesssim\|f\|_{\8}.
\end{equation}
It suffices to prove \eqref{BMO-C}. To see this, suppose \eqref{BMO-C} holds. Then, by the fact that $\|f\|_{\mathrm{BMO}_r(\mathcal A)}=\|f^*\|_{\mathrm{BMO}_c(\mathcal A)}$, we obtain
\begin{equation*}
	\|Lf\|_{\mathrm{BMO}_r(\mathcal A)}=\|(Lf)^*\|_{\mathrm{BMO}_c(\mathcal A)}=\|Lf^*\|_{\mathrm{BMO}_c(\mathcal A)}\lesssim \|f^*\|_{\infty}=\|f\|_{\infty},
\end{equation*}
which establishes \eqref{BMO-R}. 

We establish~\eqref{BMO-C} in the following. Exploiting the operator convexity of the square function $w\mapsto|w|^{2}$, we obtain for any $c\in\mathcal A$
	\begin{equation*}
		\begin{split}
			\frac{1}{m(Q_\a^k)}&\int_{Q_\a^k}\Big|f(h)-f_{\widehat{Q}_\alpha^{k}}\Big|^2dm(h)\le \frac{2}{m(Q_\a^k)}\int_{Q_\a^k}\Big|f(h) -c\Big|^2d\m(h)+2|f_{\widehat{Q}_\alpha^{k}}-c|^2\\
			&\le\frac{2}{m(Q_\a^k)}\int_{Q_\a^k}\Big|f(h) -c\Big|^2dm(h)+ \frac{2}{m(\widehat{Q}_\a^k)}\int_{\widehat{Q}_\a^k}\Big|f(h)-c\Big|^2dm(h).
		\end{split}
	\end{equation*}
Combining the above estimate with \eqref{BMO-Norm}, we conclude
	\begin{equation}\label{BMO-Es}
		\begin{split}
			\|Lf\|_{\mathrm{BMO}_c(\mathcal A)}&\lesssim\sup_{Q_\a^k\in\mathcal{F}}\inf_{c\in\mathcal A}\Big(\Big\|\frac{1}{m(Q_\a^k)}\int_{Q_\a^k}|Lf-c|^2dm(h)\Big\|^{\frac12}_{\M}\\
			&+\Big\|\frac{1}{m(\widehat{Q}_\a^k)}\int_{\widehat{Q}_\a^k}|Lf-c|^2dm(h)\Big\|^{\frac12}_{\M}\Big).
		\end{split}
	\end{equation}
	Fixing  a dyadic cube' $Q^k_\beta$, we recall that $\widehat{Q}_\beta^{k}$ is the parent of $Q_\beta^{k}$ and $k_2=\min\{k:a_0\delta^k>r_0\}$. Let $\widetilde{Q}_\b^k=\{h\in \G:d(h,z_\b^k)\le 4C_1\delta^{k+1}\}$. We then define
	\begin{equation*}
		{Q^k_\beta}^*=\left\{
		\begin{array}{ll}
			\widetilde{{Q}}^k_\beta, & \hbox{$k>k_2$;}\\
			{Q}^k_\beta, & \hbox{$k\le k_2$,}
		\end{array}
		\right.
	\end{equation*}%
	and decompose $f$ as
	$$f=f\mathds{1}_{{Q^k_\beta}^*}+f\mathds{1}_{\G\setminus {Q^k_\beta}^*}:=f_1+f_2.$$
	By applying the operator convexity of square function $w\mapsto|w|^{2}$ again, we obtain
	\begin{equation}\label{L-BMO2}
		\begin{split}
			|Lf-c|^2\le2|Lf_1|^2
			+2|Lf_2-c|^2.
		\end{split}
	\end{equation}
	Combining \eqref{BMO-Es} with \eqref{L-BMO2}, \eqref{BMO-C} reduces to proving the two equalities
	\begin{equation}\label{L-BMO1}
		\begin{split}
			\sup_{Q_\a^k\in\mathcal{F}}\inf_{c\in\mathcal A}&\Big(\Big\|\frac{1}{m(Q_\a^k)}\int_{Q_\a^k}|Lf_1(h)|^2dm(h)\Big\|^{\frac12}_{\M}+\Big\|\frac{1}{m(\widehat{Q}_\a^k)}\int_{\widehat{Q}_\a^k}|Lf_1(h)|^2dm(h)\Big\|^{\frac12}_{\M}\Big)\\
			&\lesssim \|f\|_\8
		\end{split}
	\end{equation}
	and
	\begin{equation}\label{LBMO}
		\begin{split}
			\sup_{Q_\a^k\in\mathcal{F}}\inf_{c\in\mathcal A}&\Big(\Big\|\frac{1}{m(Q_\a^k)}\int_{Q_\a^k}|Lf_2(h)-c|^2dm(h)\Big\|^{\frac12}_{\M}+\Big\|\frac{1}{m(\widehat{Q}_\a^k)}\int_{\widehat{Q}_\a^k}|Lf_2(h)-c|^2dm(h)\Big\|^{\frac12}_{\M}\Big)\\
			&\lesssim \|f\|_\8.
		\end{split}
	\end{equation}
We now turn our attention to \eqref{L-BMO1}. It is known that the elements of the von Neumann algebra $\M$ can be seen as bounded linear operators on $L_2(M)$ by right or left multiplication. By Proposition \ref{L-2}, we obtain
	\begin{equation*}
		\begin{split}
			\Big\|\int_{Q^k_{\alpha}}|Lf_1(h)|^2dm(h)\Big\|_{\M}&=\sup_{\|a\|_{L_2(\M)}\le 1}\tau\bigg(\int_{Q^k_{\alpha}}|Lf_1(h)a|^2dm(h)\bigg)\\
			&\leq\sup_{\|\|_{L_2(\M)}\le 1}\tau\bigg(\int_{G}|L(f_1a)(h)|^2dm(h)\bigg)\\
			&\lesssim \sup_{\|a\|_{L_2(\M)}\le 1}\|f_1a\|^2_{L_2(\G;L_2(\M))}\\
			&\lesssim m({Q^k_\beta}^*)\|f\|^2_\8.
		\end{split}
	\end{equation*}
Similarly, we have
$$\Big\|\int_{\widehat{Q}_\a^k}|Lf_1(h)|^2dm(h)\Big\|_{\M}\lesssim m({Q^k_\beta}^*)\|f\|^2_\8.$$
By applying~\eqref{int}, the above two inequalities imply	\begin{equation*}
		\begin{split}
			\Big\|\int_{Q^k_{\alpha}}|Lf_1(h)|^2dm(h)\Big\|_{\M}^{\frac12}+\Big\|\int_{\widehat{Q}_\a^k}|Lf_1(h)|^2dm(h)\Big\|_{\M}^{\frac12}\lesssim \Big(\frac{m({Q^k_\beta}^*)}{m(Q_\a^k)}+\frac{m({Q^k_\beta}^*)}{m(\widehat{Q}_\a^k)}\Big)^{\frac12}\|f\|_\8\lesssim\|f\|_\8,
		\end{split}
	\end{equation*}
	thereby proving~\eqref{L-BMO1}.
	
It remains to prove \eqref{LBMO}. Letting $c=Lf_2(z_\beta^k)$, we reduce \eqref{LBMO} to proving 
	\begin{equation}\label{BMO-ineq}
		\sup_{h\in \widehat{Q}_\b^k}\Big\|Lf_2(h)-Lf_2({z}_\b^k)\Big\|_{\M}\lesssim\|f\|_{\8}.
	\end{equation}
	Indeed, with this estimate established, we obtain
	\begin{equation*}		\begin{split}
			&\inf_{c\in\mathcal A}\Big(\Big\|\frac{1}{m(Q_\a^k)}\int_{Q_\a^k}\Big|Lf_2(h)-c\Big|^2dm(h)\Big\|^{\frac12}_{\M} +\Big\|\frac{1}{m(\widehat{Q}_\a^k)}\int_{\widehat{Q}_\a^k}\Big|Lf_2(h)-c\Big|^2dm(h)\Big\|^{\frac12}_{\M}\Big)\\
			&\le2\Big(\sup_{h\in \widehat{Q}_\b^k}\Big\||Lf_2(h)-Lf_2({z}_\b^k)|^2\Big\|_{\M}\Big)^{\frac12}\\
			&\lesssim \|f\|_\8.
		\end{split}
	\end{equation*}
Observe that 
	$$Lf_2(h)-Lf_2({z}_\b^k)=\sum_{n>n_{r_0}}v_n\big(M_{\delta^{n}}f_2(h)-\mathsf{E}_{n}f_2(h)-(M_{\delta^{n}}f_2({z}_\b^k)-\mathsf{E}_{n}f_2({z}_\b^k))\big).$$
Set $n_3=\max\{n_{r_0},k_2+[\log_{\delta}2C_1\delta]+1\}$. We split $n$ into two case: $n_{r_0}<n\le n_3$ and $n>n_3$. For $n_{r_0}<n\le n_3$, by the triangle inequality of the $L_{\infty}$-norm, we have
	\begin{align*}
		&\Big\|v_n\big(M_{\delta^{n}}f_2(h)-\mathsf{E}_{n}f_2(h)-(M_{\delta^{n}}f_2({z}_\b^k)-\mathsf{E}_{n}f_2({z}_\b^k))\big)\Big\|_{\M}\\
		&\leq \sum_{n_{r_0}<n\le n_3}\|(v_n)_n\|_{\ell^\infty}\|M_{\delta^{n}}f_2(h)-\mathsf{E}_{n}f_2(h)-(M_{\delta^{n}}f_2({z}_\b^k)-\mathsf{E}_{n}f_2({z}_\b^k))\Big\|_{\M}\\
		&\lesssim\sum_{n_{r_0}<n\le n_3}\sup_{h\in \widehat{Q}^k_\beta}\big(\|M_{\delta^{n}}f_2(h)\|_{\M}+\|\mathsf{E}_{n}f_2(h)\|_{\M}\big)\\
		&\lesssim (n_3-n_{r_0})\|f\|_\8.
	\end{align*}
	It remains to deal with the case $n>n_3$. We first claim that
	\begin{equation}\label{Canc}
		\mathsf{E}_{n}f_2(h)-\mathsf{E}_{n}f_2(z^k_\beta)=0,~\forall h\in \widehat{Q}^k_\beta.
	\end{equation}
	Let $ Q_{\a}^{n}\ni h$. If $n\le k$, then by Proposition~\ref{dyadic cube}(ii), $Q_{\a}^{n}\subseteq \widehat{Q}^k_\beta$. Note that $k>n_3>k_2$, one can check at once that $Q_{\a}^{n}\subseteq\widehat{Q}^k_\beta\subseteq \widetilde{Q}^k_\beta$. Thus by definition of $f_2$, we have $\mathsf{E}_{n}f_2(h)=\mathsf{E}_{n}f_2(z^k_\beta)=0$. If $n>k$, using Proposition~\ref{dyadic cube}(ii) again, we have $\widehat{Q}^k_\beta\subseteq Q_{\a}^{n}$, thus for all $h\in \widehat{Q}^k_\beta$, $\mathsf{E}_nf_2(h)$ is a constant operator, and $\mathsf{E}_{n}f_2(h)-\mathsf{E}_{n}f_2(z^k_\beta)=0$, which proves the claim.
	
	Under this claim, \eqref{BMO-ineq} reduces to proving the following inequality.
	\begin{equation}\label{BMO-ineq-1}
		\sum_{n>n_3}\Big\|M_{\delta^{n}}(f_2)(h)-M_{\delta^{n}}(f_2)(z_\b^k)\Big\|_{\M}\lesssim\|f\|_{\8}.
	\end{equation} 
	Set $n_4=\min\{n:\delta^n>2C_1\delta^{k+1}\}$. We analyze two cases: $n_3\ge n_4$ and $n_3<n_4$.\\
	
	\noindent{\bf{Case}} $n_3\ge n_4$. Applying \eqref{aver-in}, we obtain
	\begin{align*}
		&M_{\delta^{n}}(f_2)(h)-M_{\delta^{n}}(f_2)(z_\b^k)\\
		&=\frac{1}{m(B(h,\delta^{n}))}\Big(\int_{B(h,\delta^{n})}f_2(g)dm(g)-\int_{B(z_\b^k,\delta^{n})}f_2(g)dm(g)\Big)\\
		&\ \ \ +\Big(\frac{1}{m(B(h,\delta^{n}))}-\frac{1}{m(B(z_\beta^k,\delta^{n}))}\Big)
		\int_{B(z_\beta^k,\delta^{n})}f_2(g)dm(g).
	\end{align*}
	Thus 
	\begin{align*}
		&\Big\|M_{\delta^{n}}(f_2)(h)-M_{\delta^{n}}(f_2)(z_\b^k)\Big\|_{\M}\\
		&\le\frac{m(B(h,\delta^{n})\triangle B(z_\beta^k,\delta^{n}))}{m(B(h,\delta^{n}))^2}\|f\|_\8
		+\big|\frac{m(B(z_\beta^k,\delta^{n}))}{m(B(h,\delta^{n}))}-1\big\|f\|_\8\\
		&\le 2\frac{m(B(h,\delta^{n})\triangle B(z_\beta^k,\delta^{n}))}{m(B(h,\delta^{n}))}\|f\|_{\8}.
	\end{align*}
	Note that
	\begin{align*}
		&B(h,\delta^{n})\triangle B(z_\beta^k,\delta^{n})\\
		&\subseteq \bigg(B(h,d(h,z_\beta^k)+\delta^{n})\setminus B(h,\delta^{n})\bigg)\cup \bigg(B(z_\beta^k,d(h,z_\beta^k)+\delta^{n})\setminus B(z_\beta^k,\delta^{n})\bigg),
	\end{align*}
	and for every $h\in\widehat{Q}_\b^k$, $d(h,z_\beta^k)\le 2C_1\delta^{k+1}$. Then we have $B(z_\b^k,\delta^{n_3})\subseteq B(h,\delta^{n_3}+2C_1\delta^{k+1})$. Moreover, for $n>n_3$, then $\delta^{n}>\delta^{n_3}>2C_1\delta^{k+1}$. Using \eqref{int}, we obtain
	\begin{equation*}
		\frac{m(B(z_\beta^k,\delta^{n}))}{m(B(h,\delta^{n}))}\le \frac{m(B(h,\delta^{n}+2C_1\delta^{k+1})}{m(B(h,\delta^{n}))}\le 2^\epsilon(K+1).
	\end{equation*}
	Moreover, the above two observations and condition~\eqref{decay property}  yields
	\begin{align*}
		&\frac{m(B(h,\delta^{n})\triangle B(z_\beta^k,\delta^{n}))}{m(B(h,\delta^{n}))}
		\lesssim \bigg(\frac{2C_1\delta^{k+1}}{\delta^{n}}\bigg)^{\epsilon}.
	\end{align*}
	It follows that for every $h\in\widehat{Q}_\b^k$
	\begin{equation}\label{S-BMO1}
		\Big\|M_{\delta^{n}}(f_2)(h)-M_{\delta^{n}}(f_2)(z_\b^k)\Big\|_{\M}\lesssim \bigg(\frac{2C_1\delta^{k+1}}{\delta^{n}}\bigg)^{\epsilon}\|f\|_\8.
	\end{equation}
	We conclude
	\begin{align*}
		\sum_{n>n_3}\Big\|M_{\delta^{n}}(f_2)(h)-M_{\delta^{n}}(f_2)(z_\b^k)\Big\|_{\M}
		\lesssim\sum_{n>n_3}\bigg(\frac{2C_1\delta^{k+1}}{\delta^{n}}\bigg)^{\epsilon}\|f\|_\8
		\lesssim \|f\|_{\8},
	\end{align*}
	which proves \eqref{BMO-ineq-1}.\\
	
	\noindent{\bf{Case} $n_3< n_4$}. Note that, in this case, we have $2C_1\delta^{k+1}\ge \delta^{n_4-1}\ge\delta^{n_3}$. This fact yields $k\ge k_2$. So $f_2=f\mathds{1}_{\G\setminus {\widetilde{Q}^k_\beta}}$. We prove for every $n_3<n<n_4$
	\begin{equation}\label{BMO-ineq1}
		M_{\delta^{n}}f_2(h)=0, ~\forall h\in \widehat{Q}_\beta^k.
	\end{equation}
	Indeed, for every $h\in \widehat{Q}_\beta^k$ and $g\in B(h,\delta^{n})$, one has
	\begin{equation*}
		d(g,z_\beta^k)\le d(g,h)+d(h,z_\beta^k)\le \delta^{n}+2C_1\delta^{k+1}\le 4C_1\delta^{k+1}.
	\end{equation*}
It follows that for every $h\in \widehat{Q}_\beta^k$,  we have $B(h,\delta^{n})\subseteq \widetilde{Q}^k_\beta$, thus proving \eqref{BMO-ineq1} is proved. Additionally, by integrating \eqref{BMO-ineq1} with the previous proof, it also follows that
	\begin{align*}
		\sum_{n>n_3}\Big\|M_{\delta^{n}}(f_2)(h)-M_{\delta^{n}}(f_2)(z_\b^k)\Big\|_{\M}=\sum_{n\geq n_4}\Big\|M_{\delta^{n}}(f_2)(h)-M_{\delta^{n}}(f_2)(z_\b^k)\Big\|_{\M}\lesssim \|f\|_{\8},
			\end{align*}
which proves \eqref{BMO-ineq-1}.
	
\end{proof}

Prior to proving Theorem~\ref{long-ineq}(ii), we introduce additional notation. Let $\mathcal A=\N\widebar{\otimes}\mathcal{B}(\ell_2)$ be the von Neumann algebra equipped with the tensor product trace $\psi=\varphi\otimes tr$ where $tr$ is the canonical trace on $\mathcal{B}(\ell_2)$. Let $(\N_k)_{k\in\Z}$ be a filtration of $\N$ given by Subsection~\ref{nonmar}. Set $\mathcal{A}_k=\N_k\widebar{\otimes}\mathcal{B}(\ell_2)$. We observe that $(\mathcal{A}_k)_{k\in\Z}$ constitutes a filtration of $\mathcal A$. Let $\mathsf{E}_k$ be the conditional expectation with respect to $\mathcal{A}_k$. The corresponding $\mathrm{BMO}$ norm is defined as in \eqref{bmo}. Similarly to \eqref{BMO-Norm}, we have
\begin{equation}\label{BMO-Norm-s}
\begin{split}
  &\|f\|_{\mathrm{BMO}_c(\mathcal A)}=\sup_{Q_\a^k\in\mathcal{F}}\Big\|\frac{1}{m(Q_\a^k)}\int_{Q_\a^k}\Big|f(h) -f_{\widehat{Q}_\alpha^{k}}\Big|^2dm(h)\Big\|^{\frac12}_{\M\widebar{\otimes}\mathcal{B}(\ell_2)},\\
&\|f\|_{\mathrm{BMO}_r(\mathcal A)}=\sup_{Q_\a^k\in\mathcal{F}}\Big\|\frac{1}{m(Q_\a^k)}\int_{Q_\a^k}\Big|f^*(h) -f^*_{\widehat{Q}_\alpha^{k}}\Big|^2dm(h)\Big\|^{\frac12}_{\M\widebar{\otimes}\mathcal{B}(\ell_2)}.
\end{split}
\end{equation} 


\begin{proof}[Proof of~Theorem~\ref{short-ineq}(ii)]
Let $f\in L_{\8}(\mathcal A)$. Recall that $T_i=M_{r_{i-1}}-M_{r_{i}}$ and $\mathcal{S}=\cup_{n>n_{r_0}}\mathcal{S}_{n}$, where $\mathcal{S}_n=\{i: [r_{i-1},r_{i})\subseteq [\delta^{n},\delta^{n+1})\}$. Since for each operator $T_i$ satisfies $(T_i(f))^*=T_i(f^*)$, then $T_{i}(f)\otimes e_{1i}=\big(T_i(f^*)\otimes e_{i1}\big)^*$. Combining this with the proof of \eqref{BMO-R}, the $(L_\8,\mathrm{BMO})$ estimate for $(T_i)_{i\in\mathcal S}$ follows from the two inequalities below
\begin{equation}\label{BMO-C-s}
	\Big\|\sum_{i\in\mathcal S} T_if \otimes e_{i1}\Big\|_{\mathrm{BMO}_c(\mathcal A)}\lesssim\|f\|_{\8}
\end{equation}
and
\begin{equation}\label{BMO-R-s}
	\Big\|\sum_{i\in\mathcal S} T_if \otimes e_{i1}\Big\|_{\mathrm{BMO}_r(\mathcal A)}\lesssim\|f\|_{\8}.
\end{equation}

We now establish~\eqref{BMO-C-s}. The proof closely mirrors that of~\eqref{BMO-C}, requiring only minor adjustments; a sketch is included for completeness.

 Let $ f=f_1+f_2$, where $f_1=f\mathds{1}_{{Q^k_\beta}^*}$ and $f_2=f\mathds{1}_{\G\setminus {Q^k_\beta}^*}$. Following the structure of \eqref{BMO-C}, the proof of \eqref{BMO-C-s} is reduced to verifying:
	\begin{equation}\label{BMOS-1}
		\begin{split}
			\sup_{Q_\a^k\in\mathcal{F}}\inf_{c\in\mathcal A}\Big(&\Big\|\frac{1}{m(Q_\a^k)}\int_{Q_\a^k}\Big|\sum_{i\in\mathcal S}T_i(f_1)(h)\otimes e_{i1}\Big|^2dm(h)\Big\|^{\frac12}_{\M\widebar{\otimes}\mathcal{B}(\ell_2)}\\
			&\ \ \ +\Big\|\frac{1}{m(\widehat{Q}_\a^k)}\int_{\widehat{Q}_\a^k}\Big|\sum_{i\in\mathcal S}T_i(f_1)(h)\otimes e_{i1}\Big|^2dm(h)\Big\|^{\frac12}_{\M\widebar{\otimes}\mathcal{B}(\ell_2)}\Big)\\
			&\lesssim \|f\|_\8
		\end{split}
	\end{equation}
	and
	\begin{equation}\label{BMOS-2}
		\begin{split}
			\sup_{Q_\a^k\in\mathcal{F}}\inf_{c\in\mathcal A}\Big(&\Big\|\frac{1}{m(Q_\a^k)}\int_{Q_\a^k}\Big|\sum_{i\in\mathcal S}T_i(f_2)(h)\otimes e_{i1}-c\Big|^2dm(h)\Big\|^{\frac12}_{\M\widebar{\otimes}\mathcal{B}(\ell_2)}\\
			&\ \ \  +\Big\|\frac{1}{m(\widehat{Q}_\a^k)}\int_{\widehat{Q}_\a^k}\Big|\sum_{i\in\mathcal S}T_i(f_2)(h)\otimes e_{i1}-c\Big|^2dm(h)\Big\|^{\frac12}_{\M\widebar{\otimes}\mathcal{B}(\ell_2)}\Big)\\
			&\lesssim \|f\|_\8.
		\end{split}
	\end{equation}

 Observe that $|\sum_{i\in\mathcal S}T_if\otimes e_{i1}|^2=\sum_{i\in\mathcal S}|T_if|^2\otimes e_{11}$. It follows that 
	\begin{equation*}
		\begin{split}
			\Big\|\int_{\G}\Big|\sum_{i\in\mathcal S}T_if(h)\otimes e_{i1}\Big|^2dm(h)\Big\|_{\M\widebar{\otimes}\mathcal{B}(\ell_2)}=\Big\|\int_{\G}\sum_{i\in\mathcal S}\Big|T_if(h)\Big|^2dm(h)\Big\|_{\M}.
		\end{split}
	\end{equation*}
	As observed above, the proof of \eqref{BMOS-1} proceeds along the same lines as that of \eqref{L-BMO1}; the verification is left to the reader. The proof of \eqref{BMOS-2} is similar to that of \eqref{LBMO}. We now highlight the key step.
		
	Letting $c=\sum_{i\in\mathcal S}T_i(f_2)({z}_\b^k)\otimes e_{i1}$, we restrict our attention to the case $n_3\ge n_4$, i.e., $\delta^{n_3}>2C_1\delta^{k+1}$, as the case $n_3<n_4$ follows by analogous reasoning. It therefore suffices to show that for all $h\in\widehat{Q}_\b^k$
	\begin{equation*}
		\begin{split}
			&\bigg(\sum_{i\in\mathcal{S}}\Big\|\big|M_{r_{i-1}}(f_2)(h)-M_{r_{i}}(f_2)(h)-\big(M_{r_{i-1}}(f_2)(z_\b^k)-
			M_{r_{i}}(f_2)(z_\b^k)\big)\big|^2\Big\|_{\M}\bigg)^{\frac12}\lesssim \|f\|_{\8}.
		\end{split}
	\end{equation*}
We establish a stronger result.
	\begin{equation}\label{S-BMO-1}
		\begin{split}
			&\sum_{i\in\mathcal{S}_n}\Big\|\big|M_{r_{i-1}}(f_2)(h)-M_{r_{i}}(f_2)(h)-\big(M_{r_{i-1}}(f_2)(z_\b^k)-
			M_{r_{i}}(f_2)(z_\b^k)\big)\big|^2\Big\|_{\M}\\
			&\lesssim \bigg(\frac{2C_1\delta^{k+1}}{\delta^{n}}\bigg)^{\epsilon}\|f\|^2_{\8}.
		\end{split}
	\end{equation}
	Divide $i\in\mathcal{S}_n$ into two cases: $J_1=\{i:r_{i}-r_{i-1}\le (2C_1\delta^{k+1})^\epsilon/\delta^{(\epsilon-1)n}\}$ and $J_2=\{i:r_{i}-r_{i-1}> (2C_1\delta^{k+1})^\epsilon/\delta^{(\epsilon-1)n}\}$.
	
	We first deal with the case $i\in J_1$. Note that for every $z\in\widehat{Q}_\b^k$, \eqref{aver-in} and the operator convexity of square function $w\rightarrow |w|^2$ gives
	\begin{align*}
		\big|M_{r_{i-1}}(f_2)(z)-M_{r_{i}}(f_2)(z)\big|^2&\lesssim \big|1-\frac{m(B(z,r_{i-1}))}{m(B(z,r_{i}))}\big|^2\|f\|^2_{\8}\\
		&\le \Big|\int_{m(B(z,r_{i-1}))}^{m(B(z,r_{i}))}\frac{1}{u}du\Big|^2\|f\|^2_{\8}.
	\end{align*}
	Moreover, by the H\"{o}lder inequality and condition \eqref{decay property}, one has
	\begin{align*}
		\Big|\int_{m(B(z,r_{i-1}))}^{m(B(z,r_{i}))}\frac{1}{u}du\Big|^2&\le \Big(m(B(z,r_{i}))-m(B(z,r_{i-1}))\Big)\int_{m(B(z,r_{i-1}))}^{m(B(z,r_{i}))}\frac{1}{u^2}du\\
		&\lesssim \Big(\frac{r_{i}-r_{i-1}}{r_{i-1}}\Big)^\epsilon\int_{m(B(z,r_{i-1}))}^{m(B(z,r_{i}))}\frac{m(B(z,\delta^{n+n_{t_0}+1}))}{u^2}du\\
		&\le\Big(\frac{2C_1\delta^{k+1}}{\delta^{n}}\Big)^{2\epsilon}\int_{m(B(z,r_{i-1}))}^{m(B(z,r_{i}))}
		\frac{m(B(z,\delta^{n+n_{t_0}+1}))}{u^2}du.
	\end{align*}
Thus,
	\begin{equation*}
		\big|M_{r_{i-1}}(f_2)(z)-M_{r_{i}}(f_2)(z)\big|^2\lesssim \|f\|^2_\8 \Big(\frac{2C_1\delta^{k+1}}{\delta^{n}}\Big)^{2\epsilon}\int_{m(B(z,r_{i-1}))}^{m(B(z,r_{i}))}
		\frac{m(B(z,\delta^{n+1}))}{u^2}du.
	\end{equation*}
	Moreover, using the operator convexity of square function $w\rightarrow |w|^2$ again and \eqref{int}, one has
	\begin{align*}
		&\sum_{i\in J_1}\Big\|\big|M_{r_{i-1}}(f_2)(h)-M_{r_{i}}(f_2)(h)-M_{r_{i-1}}(f_2)(z_\b^k)+
		M_{r_{i}}(f_2)(z_\b^k)\big|^2\Big\|_{\M}\\
		&\lesssim \|f\|^2_\8 \Big(\frac{2C_1\delta^{k+1}}{\delta^{n}}\Big)^{2\epsilon}\sup_{z\in\widehat{Q}_\b^k}\sum_{i\in J_1}\int_{m(B(z,r_{i-1}))}^{m(B(z,r_{i}))}
		\frac{m(B(z,\delta^{n+1}))}{u^2}du\\
		&\le \|f\|^2_\8 \Big(\frac{2C_1\delta^{k+1}}{\delta^{n}}\Big)^{2\epsilon}\sup_{z\in\widehat{Q}_\b^k}\int_{m(B(z,\delta^{n}))}^{m(B(z,\delta^{n+1}))}
		\frac{m(B(z,\delta^{n+1}))}{m(B(z,\delta^{n}))^2}du\\
		&\lesssim \Big(\frac{2C_1\delta^{k+1}}{\delta^{n}}\Big)^{2\epsilon} \|f\|^2_\8.
	\end{align*}
	For the case $i\in J_2$. Note that the numbers $\#\{J_2\}\le{(\delta-1)\delta^{\epsilon n}}/{(2C_1\delta^{k+1})^\epsilon}$. Then combining  the operator convexity of square function $w\rightarrow |w|^2$ with \eqref{S-BMO1}, we conclude
	\begin{equation*}
		\begin{split}
			&\sum_{i\in J_2}\Big\|\big|M_{r_{i-1}}(f_2)(h)-M_{r_{i}}(f_2)(h)-M_{r_{i-1}}(f_2)(z_\b^k)+
			M_{r_{i}}(f_2)(z_\b^k)\big|^2\Big\|_{\M}\\
			&\le 2\sum_{i\in J_2}\Big(\Big\|\big|M_{r_{i}}(f_2)(h)-M_{r_{i}}(f_2)(z_\b^k)\big|^2\Big\|_{\M}
			+\Big\|\big|M_{r_{i-1}}(f_2)(h)-M_{r_{i-1}}(f_2)(z_\b^k)\big|^2\Big\|_{\M}\Big)\\
			&\lesssim\frac{\delta^{\epsilon n}}{(2C_1\delta^{k+1})^\epsilon}
			\bigg(\frac{2C_1\delta^{k+1}}{\delta^{n+n_{t_0}}}\bigg)^{2\epsilon}\|f\|^2_\8\\
			&\le  \bigg(\frac{2C_1\delta^{k+1}}{\delta^{n}}\bigg)^{\epsilon}\|f\|^2_\8.
		\end{split}
	\end{equation*}
Combining the estimates for cases $i\in J_1$ and $i\in J_2$, \eqref{S-BMO-1} is proved. 
	
It remains to show \eqref{BMO-R-s}. By analogy with~\eqref{BMO-C-s}, it suffices to establish the following two inequalities.
	\begin{equation}\label{SBMO-1}
		\begin{split}
			\sup_{Q_\a^k\in\mathcal{F}}\inf_{c\in\mathcal A}\Big(&\Big\|\frac{1}{m(Q_\a^k)}\int_{Q_\a^k}\Big|\sum_{i\in\mathcal S}T_i(f^*_1)(h)\otimes e_{1i}\Big|^2dm(h)\Big\|^{\frac12}_{\M\widebar{\otimes}\mathcal{B}(\ell_2)}\\
			&\ \ \ +\Big\|\frac{1}{m(\widehat{Q}_\a^k)}\int_{\widehat{Q}_\a^k}\Big|\sum_{i\in\mathcal S}T_i(f^*_1)(h)\otimes e_{1i}\Big|^2dm(h)\Big\|^{\frac12}_{\M\widebar{\otimes}\mathcal{B}(\ell_2)}\Big)\\
			&\lesssim \|f\|_\8.
		\end{split}
	\end{equation}
	and
	\begin{equation}\label{SBMO-2}
		\begin{split}
			\sup_{Q_\a^k\in\mathcal{F}}\inf_{c\in\mathcal A}\Big(&\Big\|\frac{1}{m(Q_\a^k)}\int_{Q_\a^k}\Big|\sum_{i\in\mathcal S}T_i(f^*_2)(h)\otimes e_{1i}-c\Big|^2dm(h)\Big\|^{\frac12}_{\M\widebar{\otimes}\mathcal{B}(\ell_2)}\\
			&\ \ \ +\Big\|\frac{1}{m(\widehat{Q}_\a^k)}\int_{\widehat{Q}_\a^k}\Big|\sum_{i\in\mathcal S}T_i(f^*_2)(h)\otimes e_{1i}-c\Big|^2dm(h)\Big\|^{\frac12}_{\M\widebar{\otimes}\mathcal{B}(\ell_2)}\Big)\\
			&\lesssim \|f\|_\8.
		\end{split}
	\end{equation}
	We just prove \eqref{SBMO-1} since \eqref{SBMO-2} can be treated as before. Note that
	\begin{equation*}
		\Big|\sum_{i\in\mathcal S}T_i(f^*_1)\otimes e_{1i}\Big|^2=\sum_{i,j\in\mathcal S}T_i(f_1)T_j(f^*_1)\otimes e_{ij}.
	\end{equation*}
	Moreover, $\sum_{i,j\in\mathcal S}T_i(f_1)T_j(f^*_1)\otimes e_{ij}$ can be seen as a positive operator acting on $\ell_2(L_2(\M))$. Set $\Lambda=\sum_{i,j\in\mathcal S}T_i(f_1)T_j(f^*_1)\otimes e_{ij}$. Then
	\begin{align*}
		\|\Lambda\|_{\M\widebar{\otimes}\mathcal{B}(\ell_2)}&=\sup_{\|a\|_{\ell_2(L_2(\M))}\le 1}\langle\Lambda a, a\rangle\\
		&=\sup_{\|a\|_{\ell_2(L_2(\M))}\le 1}\tau\Big[\Big(\sum_ia_i^*\otimes e_{1i}\Big)\Lambda\Big(\sum_ja_j\otimes e_{j1}\Big)\Big]\\
		&=\sup_{\|a\|_{\ell_2(L_2(\M))}\le 1}\int_G\tau(|\sum_{i\in\mathcal S}T_i(f^*_1)(h)a_i|^2)dm(h)\\
		&=\sup_{{\begin{subarray}{c}
					\|a\|_{\ell_2(L_2(\M))}\le 1 \\\|g\|_{L_2(\mathcal A)}\le 1
		\end{subarray}}}\Big[\int_G\tau\Big(\sum_{i\in\mathcal S}T_i(f^*_1)(h)a_ig(h)\Big)dm(h)\Big]^2\\
		&=\sup_{{\begin{subarray}{c}
					\|a\|_{\ell_2(L_2(\M))}\le 1 \\\|g\|_{L_2(\mathcal A)}\le 1
		\end{subarray}}}\Big[\int_G\tau\Big(\sum_{i\in\mathcal S}f^*_1(h)a_iT_ig(h)\Big)dm(h)\Big]^2\\
		&\le \sup_{{\begin{subarray}{c}
					\|a\|_{\ell_2(L_2(\M))}\le 1 \\\|g\|_{L_2(\mathcal A)}\le 1
		\end{subarray}}}\|(T_{i}g)_{i\in\mathcal S}\|^2_{L_2(\mathcal A;\ell_2^{rc})}\|(f_1^*a_i)_{i\in\mathcal S}\|^2_{L_2(\mathcal A;\ell_2^{rc})}\\
		&\lesssim \sup_{\|a\|_{\ell_2(L_2(\M))}\le 1}\int_G\tau\sum_{i\in\mathcal S}|f_1^*(h)a_i|^2dm(h)\\
		&\le \m({Q^k_\beta}^*)\|f\|^2_\8,
	\end{align*}
	where last equation follows from the property that $\mu$ is a translation invariant measure and the last but one inequality follows from Proposition \ref{L-2}. On account of this estimate, we have
	\begin{align*}
		\Big(&\Big\|\frac{1}{m(Q_\a^k)}\int_{Q_\a^k}\Big|\sum_{i\in\mathcal S}T_i(f^*_1)(h)\otimes e_{1i}\Big|^2dm(h)\Big\|^{\frac12}_{\M\widebar{\otimes}\mathcal{B}(\ell_2)}\\
		&\ \ \ +\Big\|\frac{1}{m(\widehat{Q}_\a^k)}\int_{\widehat{Q}_\a^k}\Big|\sum_{i\in\mathcal S}T_i(f^*_1)(h)\otimes e_{1i}\Big|^2dm(h)\Big\|^{\frac12}_{\M\widebar{\otimes}\mathcal{B}(\ell_2)}\Big)\\
		&\lesssim \Big(\frac{m({Q^k_\beta}^*)}{m(Q_\a^k)}+\frac{m({Q^k_\beta}^*)}{m(\widehat{Q}_\a^k)}\Big)^{\frac12}\|f\|_\8\\
		&\lesssim\|f\|_\8,
	\end{align*}
	which proves \eqref{SBMO-1}.
	
\end{proof}
\begin{remark}\rm{
Let $L_nf=M_{\delta^{n+n_{r_0}}}f-\mathsf{E}_{n+n_{r_0}}f$. The $(L_{\infty},\mathrm{BMO})$ estimate 
\begin{align*}
			\Big\|\sum_{n\in\mathbb N} L_nf\otimes e_{1n}\Big\|_{\mathrm{BMO}_{d}(\mathcal{R})}&+\Big\|\sum_{n=1}^\8 L_nf \otimes e_{n1}\Big\|_{\mathrm{BMO}_{d}(\mathcal{R})}\le  C_{p}\,\|f\|_\infty,\; \forall f\in L_{\infty}(\mathcal N);
		\end{align*}
can be established using methods similar to those in Theorem \ref{short-ineq}(ii), where the cancellation proposition \eqref{Canc} will be used. We omit the proof.

}

\end{remark}

\subsection{Strong type $(p,p)$ estimates}\quad

\medskip

We conclude the proofs of Theorems \ref{long-ineq} and \ref{short-ineq} by establishing their respective strong type $(p,p)$ estimates. Specifically, the following proposition will verify both Theorem \ref{long-ineq}(iii) and Theorem \ref{short-ineq}(iii).

\begin{prop}\label{strong-type}
	Let $1<p<\infty$.  Then $L$ is bounded as an operator from $L_{p}(\mathcal{N})$ to $L_{p}(\mathcal{N})$, and $(S_i)_i$ is bounded as an operator from $L_{p}(\mathcal{N})$ to $L_p(\mathcal{N};\ell_{2}^{rc})$.
\end{prop}

\begin{proof}
We focus our proof on the operator $L$.  Combining the weak type $(1,1)$ estimate of $L$ with Proposition \ref{L-2}, we obtain the strong type $(p,p)$ estimate for $1<p<2$ via real interpolation \cite{Pis-Xu03}. Similarly, by applying Proposition~\ref{L-2} and the $(L_\infty,\mathrm{BMO})$ estimates, we derive the strong type $(p,p)$ estimate for $2<p<\infty$ via complex interpolation \cite{JM}.

It remains to estimate $(S_i)_i$. Let $\mathcal T=\sum_{i\in\mathcal S}\varepsilon_i T_i$.  For $1<p<2$, by combining the weak type $(1,1)$ estimate of $\mathcal T$ with Proposition \ref{S-2}, we deduce that $\mathcal T$ is bounded from $L_{p}(\mathcal{N})$ to $L_p(L_{\infty}(\Omega)\overline{\otimes}\mathcal{N})$ via real interpolation \cite{Pis-Xu03}. Consequently, invoking the noncommutative Khintchine inequalities (Proposition~\ref{nonkin}), we obtain that $(L_n)_{n\in\mathbb N}$ is bounded from $L_{p}(\mathcal{N})$ to $L_p(\mathcal{N};\ell_{2}^{rc})$.

	We now turn to the case $2<p<\infty$. Define the operators $\mathcal T_{c}f=\sum_{i\in \mathcal S} T_if \otimes e_{i1}$ and $\mathcal T_{r}f=\sum_{i\in \mathcal S} T_if \otimes e_{1i}$. By Proposition~\ref{S-2} and $(L_\infty,\mathrm{BMO})$ estimates, both $\mathcal T_{c}$ and $\mathcal T_{r}$ are bounded from $L_p(\mathcal N)$ to $L_p(\mathcal{N}\overline{\otimes}\mathcal B({\ell_2}))$ via complex interpolation \cite{JM}. Thus, $(T_i)_{i\in\mathcal S}$ is bounded from $L_p(\mathcal N)$ to $L_p(\mathcal{N};
	\ell_{2}^{rc})$ for all $2< p<\infty$.
\end{proof}

We emphasize that the use of noncommutative Khintchine inequalities (Proposition~\ref{nonkin}) enables Theorem~\ref{long-ineq} to extend the main result of Hong and Xu~\cite{HX} to a general metric space framework. This generalization is achieved by employing noncommutative Doob inequalities~\cite[Theorem 0.2]{Jun}, adapting the techniques from~\cite[Proof of Corollary 1.4]{HX}, and leveraging the structural insights in~\cite[Proposition 2.3]{HWW}. Consequently, Theorem~\ref{long-ineq} establishes both weak type $(1,1)$ and strong type $(p,p)$ estimates for the operator-valued maximal function. This recovers, through our alternative approach, a result originally proved in~\cite{HLW} via the analysis of random dyadic systems. The maximal ergodic inequalities for polynomial group actions are then derived by applying a noncommutative generalization of Calder\'{o}n's transference principle, as detailed in~\cite{HLW}. Finally, an application of the Banach principle yields the following pointwise ergodic theorem.
\begin{cor}\label{cor3}
	Let $G$ be generated by a symmetric compact subset $V$ and be of polynomial growth. Let $m$ be a Haar measure on $G$. Assume that $\a$ is an action of $G$ on the associated $L_p$-spaces $L_p(\M)$. Define the ergodic averages
	\begin{equation*}
		A_nx=\frac{1}{m(V^n)}\int_{V^n}\a_hxdm(h).
	\end{equation*}
	\begin{enumerate}[\noindent]
		\item\emph{(i)} Fix $1<p<\8$. If $\alpha$ is a strongly continuous and uniformly bounded action of $G$ on $L_p(\M)$. Then ergodic
		sequence $(A_{\delta^n}x)_n$ converges bilaterally almost uniformly for $x\in L_p(\M)$.
		\item\emph{(ii)} If $\a$ is a strong continuous action of $G$ on $\M$ by $\tau$-preserving automorphisms, so $\a$ extends to an isometric actions on the spaces $L_p(\M)$ for all $1\le p\le\8$. Then ergodic sequence $(A_{\delta^n}x)_n$ converges almost uniformly for $x\in L_p(\M)$ for all $1\le p<\8$.
	\end{enumerate}
\end{cor}
We note that Corollary \ref{cor3} is a specific instance of the main result in \cite[Theorem 1.2]{HLW}, where the reader may consult \cite{HLW} for definitions of undefined terms such as maximal inequality and bilaterally almost uniform convergence.

\section{Proof of Theorem \ref{thm:trans}}\label{thisissec5}

In this section, we investigate the noncommutative analogue of the classical transference principle established by Coifman and Weiss \cite{CoWe}. Recall that a sequence of compact sets $\{F_n\}_{n\in\mathbb{N}}$ with positive Haar measures in a locally compact group $G$ is called a F${\o}$lner sequence if, for every $g\in G$,
\begin{equation*}
 \lim_{i}\frac{m((F_ng)\triangle F_n)}{m(F_n)}=0,
\end{equation*}
or equivalently
for all compact set $K$ in $G$,
\begin{equation}\label{amen}
  \lim_{n}\frac{m(F_nK)}{m(F_n)}=1.
\end{equation}
A group  $G$ is said to be amenable if it admits a F${\o}$lner sequence.

\subsection{Strong type inequalities.}\quad

\medskip

The transference principle for strong type $(p,p)$ inequalities requires careful consideration of bounded linear operator extensions. This foundational result follows directly from noncommutative Khintchine inequalities (Proposition \ref{nonkin}), a standard tool familiar to specialists in functional analysis. We omit the detailed proof as it constitutes an exercise for advanced readers.


\begin{lem}\label{ex-lin}
	Let $1\le p<\8$. Assume that $T$ is a bounded linear operator on $L_p(\mathcal M)$. Then $T$ extends to a bounded operator on $L_p(\mathcal M;\ell^{rc}_2)$.
\end{lem}

\begin{proof}[Proof of Theorem \ref{thm:trans}(i)]
	Let $p\in(1,\8)$ and $x\in L_p(\mathcal M)$. By the standard approximation argument stated in Remark \ref{app}, it suffices to show for any fixed integer $i_0$,
	\begin{equation}\label{st-trans}
		\|(A_{r_i}x-A_{r_{i+1}}x)_{1\le i\le i_0}\|_{L_{p}(\M;\ell_2^{rc})}\lesssim \|x\|_{L_p(\M)}.
	\end{equation}
	Let $\a$ be a strongly continuous and uniformly bounded action of $G$ on $L_p(M)$. Fix $x\in L_p(\M)$ and a compact set $A$. Let $K$ be a compact set such that $B_{r_{i_0+1}}\subseteq K$. Define $L_p(\M)$-valued function $F_{AK}$ on $G$ as
	\begin{equation*}
		F_{AK}(g)=\chi_{AK}(g)\a_g x.
	\end{equation*}
	By the metric invariance property of $d$, we have $hB_r=B(h,r)$, for all $h\in G$. Furthermore, since $m$ is a Haar measure, it follows that $\forall h\in G, m(B_r)=m(B(h,r))$. This leads to the following results for $1\le i\le i_0+1$,
	\begin{equation}\label{avera-eq}
		\a_h A_{r_i}x=\frac{1}{m(B_{r_i})}\int_{B_{r_i}}\a_h\a_gxdm(g)=\frac{1}{m(B_{r_i})}\int_{B_{r_i}}F_{AK}(hg)dm(g)=M_{r_i}F_{AK}(h).
	\end{equation}
	Consequently, we obtain
	\begin{equation*}
		\a_h(A_{r_{i}}x-A_{r_{i+1}}x)=M_{r_{i}}F_{AK}(h)-M_{r_{i+1}}F_{AK}(h).
	\end{equation*}
	By combining Lemma \ref{ex-lin} with the uniform boundedness assumption of  $\a_g$ on $L_p(\M)$, specifically $\sup_{g\in G}\|\a_g\|_{L_p(\M)\rightarrow L_p(\M)}<\8$, we derive the following estimate
	\begin{equation*}
		\begin{split}
			\|(A_{r_i}x-A_{r_{i+1}}x)_{1\le i\le i_0}\|_{L_{p}(\M;\ell_2^{rc})}&\lesssim \|(\a_gA_{r_i}x-\a_gA_{r_{i+1}}x)_{1\le i\le i_0}\|_{L_{p}(\M;\ell_2^{rc})}\\
			&=\|(M_{r_i}F_{AK}(g)-M_{r_{i+1}}F_{AK}(g))_{1\le i\le i_0}\|_{L_{p}(\M;\ell_2^{rc})}.
		\end{split}
	\end{equation*}
	Furthermore, applying the strong type $(p,p)$ square function inequality for the translation action yields
	\begin{equation*}
		\|(M_{r_i}F_{AK}-M_{r_{i+1}}F_{AK})_{1\le i\le i_0}\|_{L_{p}(L_\8(G)\overline{\otimes}\M;\ell_2^{rc})}\lesssim \|F_{AK}\|_{L_{p}(L_\8(G)\overline{\otimes}\M)}.
	\end{equation*}
	We now establish \eqref{st-trans}. By applying the Fubini theorem together with the noncommutative Khintchine inequalities (Proposition~\ref{nonkin}), we obtain $L_p(L_\8(G)\overline{\otimes}\M;\ell_2^{rc})\approx L_p(G;L_p(\M;\ell_2^{rc}))$. Therefore,
	\begin{align*}
		&\|(A_{r_i}x-A_{r_{i+1}}x)_{1\le i\le i_0}\|^p_{L_{p}(\M;\ell_2^{rc})}\\
		&\lesssim\frac{1}{m(A)}\int_A \|(M_{r_i}F_{AK}(g)-M_{r_{i+1}}F_{AK}(g))_{1\le i\le i_0}\|^p_{L_{p}(\M;\ell_2^{rc})}dm(g)\\
		&\le \frac{1}{m(A)}\int_G \|(M_{r_i}F_{AK}(g)-M_{r_{i+1}}F_{AK}(g))_{1\le i\le i_0}\|^p_{L_{p}(\M;\ell_2^{rc})}dm(g)\\
		&\approx\frac{1}{m(A)}\|(M_{r_i}F_{AK}-M_{r_{i+1}}F_{AK})_{1\le i\le i_0}\|^p_{L_{p}(L_\8(G)\overline{\otimes}\M;\ell_2^{rc})}\\
		&\lesssim \frac{1}{m(A)}\|F_{AK}\|^p_{L_{p}(L_\8(G)\overline{\otimes}\M)}\lesssim \frac{m(AK)}{m(A)}\|x\|^p_p.
	\end{align*}
	Since $G$ is an amenable group, by~\eqref{amen}, for any $\varepsilon>0$, we may select a subset $A\subset G$ such that $m(AK)/m(A)\le(1+\varepsilon)$. Letting $\varepsilon$ tend to $0$, we conclude the proof of~\eqref{st-trans}.
\end{proof}

\medskip

\begin{remark}\rm
Our transference principle extends naturally to the vector-valued setting. Specifically, let $\mathbb B$ be a $p$-uniformly convex Banach spaces with $p\ge 2$. For any $x\in\mathbb B$, we define the average operator $A_rx=\frac{1}{m(B_r)}\int_{B_r}\a_gxdm(g)$, where $\a_g$ represents a $G$-action satisfying $\sup_{g\in G}\|\a_g:\mathbb B\rightarrow \mathbb B\|<\8$. Define the vector-valued averaging operator $\mathcal{A}_r$ on $L^p(G;\mathbb B)$ by 
 $$\mathcal{A}_rf(h)=\frac{1}{m(B_r)}\int_{B_r}f(hg) dm(g),~\forall f\in L^p(G;\mathbb B).$$
 A slight modification of the transference principle's proof technique allows us to establish the following implication between norm estimates. From the variational norm estimate 
 \begin{equation}\label{norm-var-1}
 	\sup_{(r_i)}\Big(\sum_i\|\mathcal{A}_{r_{i+1}}f-\mathcal{A}_{r_i}f\|_{L^p(G;\mathbb B)}^p\Big)^{\frac1p}\lesssim\|f\|^p_{L^p(G;\mathbb B)},
 \end{equation}
 we can deduce the corresponding estimate 
 \begin{equation}\label{norm-var-2}
 	 \sup_{(r_i)}\Big(\sum_i\|A_{r_{i+1}}x-A_{r_i}x\|_{\mathbb B}^p\Big)^{\frac1p}\lesssim\|x\|_{\mathbb B}.
 \end{equation}
 To the best of our knowledge, inequality \eqref{norm-var-2} was initially established by Jones et al.~\cite{JOR96} in Hilbert spaces, where $A_nx=\sum_{i=1}^n T^nx$ with $T$ being an isometry. The generalization to $p$-uniformly convex Banach spaces by Avigad and Rute~\cite{AR} required $T$ to be power bounded from above and below. Furthermore, Avigad and Rute utilized \eqref{norm-var-2} to establish upper bounds on the number of $\varepsilon$-fluctuations in ergodic average sequences $(A_{n_i}x)_i$. For more details we refer the reader to \cite[Section 5 and 6]{AR}.
 
 On the other hand, if $G$ is a group of polynomial growth with a symmetric compact generating set, or more generally satisfies the conditions of Theorem~\ref{main-thm1}, \eqref{norm-var-1} holds in this case by virtue of the methods developed in \cite{AR}, along with the boundary properties established in Propositions \ref{boundary} and~\ref{measure estimate of cube}. Consequently, \eqref{norm-var-2} follows from the aforementioned transference principle. Furthermore, employing the techniques from \cite{Hong-Liu21} and \cite{HM1}, we derive the following vector-valued variational inequality: Let $\mathbb B$ be a $p$-uniformly convex Banach space. Then for all $p_0>p$ and $1<q<\infty$, 
   $$\norm{\sup_{(r_i)}\Big(\sum_i\|\mathcal{A}_{r_{i+1}}f-\mathcal{A}_{r_i}f\|_{\mathbb B}^{p_0}\Big)^{\frac1{p_0}}}_{q}\lesssim\|f\|_{L^q(G;\mathbb B)}$$ 
\end{remark}

\subsection{Weak type inequalities.}\quad

\medskip

We now investigate the transference principle for weak-type $(p,p)$ inequalities associated with the strongly continuous action $\a_h$ of $G$ on $\M$, where $\a_h$ is induced by $\tau$-preserving automorphisms of $\M$. It is well-known that $\a_h$ extends to an isometric action on $L_p(\M)$ for all $1\le p\le \8$.
Our analysis relies crucially on the fundamental result concerning automorphisms (see   Proposition 3.3 in \cite{Pa-Su}).
\begin{prop}\cite[Proposition 3.3]{Pa-Su}\label{auto}
	Let $\M$ be a von Neumann algebra and $\a$ a $*$-automorphism on $\M$. Then for any $a\in L_0(\M)$ which is self-adjoint and Borel function $f:\mathbb{R}\rightarrow\mathbb C$ which is bounded on bounded subsets of $\mathbb R$,
	\begin{equation*}
		\a(f(a))=f(\a(a)).
	\end{equation*}
\end{prop}
We now proceed to prove Theorem \ref{thm:trans}(ii).
\begin{proof}[Proof of Theorem \ref{thm:trans}(ii)]
	Let $1\le p<\8$. By employing the methodology outlined in the preceding subsection, it suffices to establish that for any fixed integer $i_0$, the following weak-type estimate holds
	\begin{equation}\label{weak-trans}
		\|(A_{r_i}x-A_{r_{i+1}}x)_{1\le i\le i_0}\|_{L_{p,\8}(\M;\ell_2^{rc})}\lesssim \|x\|_{L_p(\M)}.
	\end{equation}
Since every element $x \in \mathcal{M}$ can be expressed as a linear combination of four positive elements, it suffices to demonstrate the inequality for non-negative $x$.  Moreover, choosing a compact set $K$ such that $B_{r_{i_0+1}}\subseteq K$. Let $A$ be a compact set. Define $L_p(\M)$-valued function $F_{AK}(g)=\chi_{AK}(g)\a_g x$. Similar to \eqref{avera-eq}, we obtain for all $h\in G$ and $1\le i\le i_0+1$, $\a_h A_{r_i}x=M_{r_i}F_{AK}(h)$ and $\a_h(A_{r_{i}}x-A_{r_{i+1}}x)=M_{r_{i}}F_{AK}(h)-M_{r_{i+1}}F_{AK}(h)$. Moreover, using the assumption that the weak type $(p,p)$ square function inequality hold for the translation action, we get
	\begin{equation}\label{w-transla}
		\|(M_{r_i}F_{AK}-M_{r_{i+1}}F_{AK})_{1\le i\le i_0}\|_{L_{p,\8}(L_\8(G)\overline{\otimes}\M;\ell_2^{rc})}\lesssim \|F_{AK}\|_{L_{p}(L_\8(G)\overline{\otimes}\M)}.
	\end{equation}
	
	To prove \eqref{weak-trans}, we first consider the case $2\le p<\8$. For brevity, we set
	\begin{align*}
		a_{i}(h)=M_{r_i}F_{AK}(h)-M_{r_{i+1}}F_{AK}(h)\, ~\textit{and}~\, b_i=A_{r_{i}}x-A_{r_{i+1}}x.
	\end{align*}
	Then, for each $h\in G$, we have $\a_{h^{-1}}a_{i}(h)=b_i$. Let $\lambda>0$. By Proposition \ref{auto} and the trace-preserving property of the automorphisms $\a_h$, we infer that
	\begin{equation}\label{weakt}
		\begin{split}
			\tau\Big(\chi_{(\lambda,\8)}\Big(\Big(\sum_{i=1}^{i_0}|b_i|^2\Big)^{\frac12}\Big)\Big)&=
			\frac{1}{m(A)}\int_A\tau\Big(\chi_{(\lambda,\8)}\Big(\Big(\sum_{i=1}^{i_0}|b_i|^2\Big)^{\frac12}\Big)\Big)dm(h)\\
			&=\frac{1}{m(A)}\int_A\tau\Big(\chi_{(\lambda,\8)}\Big(\Big(\sum_{i=1}^{i_0}|\a_{h^{-1}}a_i(h)|^2\Big)^{\frac12}\Big)\Big)dm(h)\\
			&=\frac{1}{m(A)}\int_A\tau\Big(\chi_{(\lambda,\8)}\Big(\Big(\sum_{i=1}^{i_0}|a_i(h)|^2\Big)^{\frac12}\Big)\Big)dm(h).
		\end{split}
	\end{equation}
	Multiplying both sides of the preceding equation by $\lambda^p$ and subsequently taking the supremum over$\lambda$,  we may apply Fubini's theorem to obtain
		\begin{equation*}
		\|(b_i)_{1\le i\le i_0}\|^p_{L_{p,\8}(\M;\ell_2^{c})}\le\frac{1}{m(A)}\|(a_i)_{1\le i\le i_0}\|^p_{L_{p,\8}(L_\8(G)\overline{\otimes}\M;\ell_2^{c})}.
	\end{equation*}
	Observe that $\a_h^*=\a_{h^{-1}}$, and replacing  $a_i$ with $a_i^*$ in the preceding inequality, we obtain the following result
	\begin{equation*}
		\|(b_i)_{1\le i\le i_0}\|^p_{L_{p,\8}(\M;\ell_2^{r})}\le\frac{1}{m(A)}\|(a_i)_{1\le i\le i_0}\|^p_{L_{p,\8}(L_\8(G)\overline{\otimes}\M;\ell_2^{r})}.
	\end{equation*}
	Combining the above estimates with \eqref{w-transla}, we derive
	\begin{align*}
		\|(b_i)_{1\le i\le i_0}\|^p_{L_{p,\8}(\M;\ell_2^{rc})}&\lesssim \frac{1}{m(A)}\|(a_i)_{1\le i\le i_0}\|^p_{L_{p,\8}(L_\8(G)\overline{\otimes}\M;\ell_2^{rc})}\lesssim \frac{1}{m(A)}\|F_{AK}\|^p_{L_{p}(L_\8(G)\overline{\otimes}\M)}\\
		&=\frac{m(AK)}{m(A)}\|x\|_{L_p(\M)}^p\lesssim\|x\|_{L_p(\M)}^p,
	\end{align*}
	where the last inequality follows form \eqref{amen} and a similar argument as in the proof of strong type inequalities.
	
To complete the proof, it remains to establish \eqref{weak-trans} for the case $1\le p<2$. We retain the notations $a_i$ and $b_i$ as defined in the case $2\le p<\8$. Let $\varepsilon>0$, by the definition of norm $\|\cdot\|_{L_{p,\8}(L_\8(G)\overline{\otimes}\M;\ell_2^{rc})}$, the exists a factorization $a_i=f_i+g_i$ such that
	\begin{equation}\label{wea-ineq}
		\begin{split}
			&\|(f_{i})_{1\leq i\leq i_0}\|^p_{L_{p,\8}(L_\8(G)\overline{\otimes}\M;\ell_2^{r})}+\|(g_{i})_{1\leq i\leq i_0}\|^p_{L_{p,\8}(L_\8(G)\overline{\otimes}\M;\ell_2^{c})}\\
			& \lesssim\big\|(a_i)_{1\le i\le i_0}\big\|^p_{L_{p,\8}(L_\8(G)\overline{\otimes}\M;\ell_2^{rc})}+\varepsilon^p.
		\end{split}
	\end{equation}
Since $b_i=\a_{h^{-1}}a_i(h)$ and applying the same technique as in \eqref{weakt}, we derive the following results
\begin{align*}
		\tau\Big(\chi_{(\lambda,\8)}\Big(\Big(\sum_{i=1}^{i_0}|b_i|^2\Big)^{\frac12}\Big)\Big)
		&=\frac{1}{m(A)}\int_A\tau\Big(\chi_{(\lambda,\8)}\Big(\Big(\sum_{i=1}^{i_0}|a_i(h)|^2\Big)^{\frac12}\Big)\Big)dm(h)
	\end{align*}
	and 
	\begin{align*}
		\tau\Big(\chi_{(\lambda,\8)}\Big(\Big(\sum_{i=1}^{i_0}|b^*_i|^2\Big)^{\frac12}\Big)\Big)
		&=\frac{1}{m(A)}\int_A\tau\Big(\chi_{(\lambda,\8)}\Big(\Big(\sum_{i=1}^{i_0}|a^*_i(h)|^2\Big)^{\frac12}\Big)\Big)dm(h).
	\end{align*}
This implies that
\begin{align*}
		&\|(b_i)_{1\le i\le i_0}\|^p_{L_{p,\8}(\M;\ell_2^{rc})}\\
		&\lesssim \frac{1}{m(A)}\|(a_i)_{1\le i\le i_0}\|^p_{L_{p,\8}(L_\8(G)\overline{\otimes}\M;\ell_2^{c})}+\frac{1}{m(A)}\|(a^*_i)_{1\le i\le i_0}\|^p_{L_{p,\8}(L_\8(G)\overline{\otimes}\M;\ell_2^{c})}.
	\end{align*}
Since $a_i=f_i+g_i$, by applying the quasi-triangle inequality for the weak $L_p$ norm $\|\cdot\|_{L_{p,\8}(L_\8(G)\overline{\otimes}\M;\ell_2^{c})}$ together with \eqref{wea-ineq} and \eqref{w-transla}, we obtain
	\begin{align*}
		m(A)\|(b_i)_{1\le i\le i_0}\|^p_{L_{p,\8}(\M;\ell_2^{rc})}&\lesssim \|(a_i)_{1\le i\le i_0}\|^p_{L_{p,\8}(L_\8(G)\overline{\otimes}\M;\ell_2^{c})}+\|(a^*_i)_{1\le i\le i_0}\|^p_{L_{p,\8}(L_\8(G)\overline{\otimes}\M;\ell_2^{c})}\\
		&\lesssim \|(f_i)_{1\le i\le i_0}\|^p_{L_{p,\8}(L_\8(G)\overline{\otimes}\M;\ell_2^{c})}+\|(g_i)_{1\le i\le i_0}\|^p_{L_{p,\8}(L_\8(G)\overline{\otimes}\M;\ell_2^{c})}\\
		&\ \  \ +\|(f^*_i)_{1\le i\le i_0}\|^p_{L_{p,\8}(L_\8(G)\overline{\otimes}\M;\ell_2^{c})}+\|(g^*_i)_{1\le i\le i_0}\|^p_{L_{p,\8}(L_\8(G)\overline{\otimes}\M;\ell_2^{c})}\\
		&\lesssim\big\|(a_i)_{1\le i\le i_0}\big\|^p_{L_{p,\8}(L_\8(G)\overline{\otimes}\M;\ell_2^{rc})}+\varepsilon^p\\
		&\lesssim\|F_{AK}\|^p_{L_{p}(L_\8(G)\overline{\otimes}\M)}+\varepsilon^p.
	\end{align*}
	By the arbitrariness of $\varepsilon$, we have
	\begin{equation*}
		\|(A_{r_i}x-A_{r_{i+1}}x)_{1\le i\le i_0}\|^p_{L_{p,\8}(\M;\ell_2^{rc})}\lesssim \frac{m(AK)}{m(A)}\|x\|^p_p\lesssim\|x\|^p_p,
	\end{equation*}
	which completes the proof.
\end{proof}
\begin{remark}\rm
We present an alternative approach to proving \eqref{weak-trans} for $p=1$. Recall that the definitions of $a_i$ and $b_i$  are given above. By the noncommutative Khintchine inequalities (see Proposition~\ref{nonkin}), it suffices to show
	\begin{equation}\label{weak1}
		\|\sum_{i=1}^{i_0}\varepsilon_ib_i\|_{L_{1,\8}(L_{\infty}(\Omega)\overline{\otimes}\M)}\lesssim \|x\|_{L_1(\M)}.
	\end{equation}
By the same argument as in \eqref{weakt}, we obtain
	\begin{align*}
		\int_{\Omega}\tau\Big(\chi_{(\lambda,\8)}\Big(\sum_{i=1}^{i_0}\varepsilon_i(t)b_i\Big)\Big)dm(h)dP(t)
		=\frac{1}{m(A)}\int_{\Omega}\int_A\tau\Big(\chi_{(\lambda,\8)}\Big(\sum_{i=1}^{i_0}\varepsilon(t)a_i(h)\Big)\Big)dm(h)dP(t).
	\end{align*}
	It follows form \eqref{w-transla} and Proposition~\ref{nonkin} that
	\begin{align*}
		\|\sum_{i=1}^{i_0}\varepsilon_ib_i\|_{L_{1,\8}(L_{\infty}(\Omega)\overline{\otimes}\M)}&\le
		\frac{1}{m(A)}\|\sum_{i=1}^{i_0}\varepsilon_ia_i\|_{L_{1,\8}(L_{\infty}(\Omega)\overline{\otimes}L_{\8}(G)\overline{\otimes}\M)}\\
		&\lesssim\frac{m(AK)}{m(A)}\|x\|_{L_1(\M)}\\
		&\lesssim\|x\|_{L_1(\M)},
	\end{align*}
	which proves \eqref{weak1}.
\end{remark}

\section{Lamperti representation}\label{laprepsec}
\begin{defi}\label{def:lamperti}
Let $1 \leq p < \infty$.  
A bounded linear operator $
   T $ on $L_p(\mathcal{M})$ is said to be a \emph{Lamperti operator} if for all $\tau$-finite projections $e,f \in \mathcal{M}$ with $ef = 0$, one has
\[
   (Te)^* (Tf) = 0 
   \qquad \text{and} \qquad 
   (Te)(Tf)^* = 0.
\]
\end{defi} 
\begin{thm}\cite{HongRayWang2023}\label{thm:lamperti-decomposition}
Let $1 \leq p < \infty$ and let $
   T : L_p(\mathcal{M},\tau) \;\to\; L_p(\mathcal{M},\tau)$
be a Lamperti operator of norm $\|T\|_{p\to p}^p = C$.  
Then there exist uniquely:
 a partial isometry $w \in \mathcal{M}$,
   a positive self-adjoint operator $b$ affiliated with $\mathcal{M}$,
     and a normal Jordan $*$-homomorphism $J : \mathcal{M} \to \mathcal{M}$, such that the following hold:
\begin{enumerate}
    \item $w^* w = J(1) = s(b)$, where $s(b)$ denotes the support of $b$.  
    Moreover, if $T$ is positive, then $w = J(1) = s(b)$.
    
    \item Every spectral projection of $b$ commutes with $J(x)$ for all $x \in \mathcal{M}$.
    
    \item For all $x \in \mathcal{S}_{\mathcal{M}}$, one has
    \begin{equation}\label{lampertstrrep}
        T(x) = w \, b \, J(x).
  \end{equation}
    
    \item For all $x \in \mathcal{M}_+$,
    $
       \tau\!\bigl( b^p \, J(x) \bigr) \;\leq\; C^p \, \tau(x).
$
    
\end{enumerate}
\end{thm}
\begin{remark} In Theorem \eqref{thm:lamperti-decomposition} the decomposition in \eqref{lampertstrrep} will be referred to as Lamperti decomposition of the operator $T.$
\end{remark}
\begin{prop}\label{moduluslamp}
Let $1 \leq p < \infty$ and let $T : L_p(\mathcal M) \to L_p(\mathcal M)$ be a Lamperti operator with decomposition $T = w b J$.  
Define the operator $|T|$  by
\[
   |T|(x) := b J(x), \qquad x \in \mathcal{S}_{\mathcal M}.
\]
Then $|T|$ extends to a positive Lamperti operator on $L_p(\mathcal M)$, and one has $|Tx| = ||T|x|=|T||x|$ for all $x\in L_p(\mathcal M).$
Moreover,
\[
   \|T\|_{p \to p} = \||T|\|_{p \to p}.
\]
\end{prop}

\begin{proof}
Take $x \in \mathcal{S}_{\mathcal M}$. Then
\[
   (Tx)^* (Tx) 
      = J(x)^* b w^* w b J(x) 
      = J(x)^* b^2 J(x) 
      = |b J(x)|^2 
      = \bigl(|T|x\bigr)^* \bigl(|T|x\bigr).
\]
Here, the second equality uses parts (1) and (2) of Theorem~\ref{thm:lamperti-decomposition}.  
By uniqueness of the square root, it follows that
\begin{equation}\label{modforsm}
   |Tx| = |\,|T|x\,|=|T||x|, \qquad x \in \mathcal{S}_{\mathcal M}.
\end{equation}
Consequently,
\[
   \|Tx\|_p = \||T|x\|_p, \qquad x \in\mathcal{S}_{\mathcal M}.
\]
Hence $|T|$ extends uniquely to a bounded operator on $L_p(\mathcal M)$.  
By construction, $|T|$ is positive, and it is again a Lamperti operator (see \cite[Remark~3.4]{HongRayWang2023} for details) with $\|T\|_{p\to p}=\||T|\|_{p\to p}.$

Finally, to prove \ref{modforsm} for arbitrary $x \in L_p(\mathcal M)$, choose a sequence $(x_n) \subseteq \mathcal{S}_{\mathcal M}$ with $\|x_n - x\|_p \to 0$.  
By continuity of the map $x \mapsto |x|$ in $L_p(\mathcal M)$  \cite{Kosaki1984,CaspersPotapovSukochevZanin2015}, we deduce
\[
   |Tx| = \lim_{n \to \infty} |Tx_n| 
         = \lim_{n \to \infty} ||T|x_n|
         = ||T|x|.
\]
This proves the claim.
\end{proof}
\begin{remark} For a Lamperti operator $T$ on $L_p(\mathcal{M}),$ the operator $|T|$ defined in Proposition \eqref{moduluslamp} is called the modulus of $T.$
\end{remark}
Let $G$ be a locally compact group and $1\leq p<\infty.$ Let $\alpha$ be a strongly continuous uniformly bounded action of $G$ on $L_p(\mathcal{M}).$ We say $\alpha$ is (positive) Lamperti if $\alpha_g$ is a (positive) Lamperti operator for all $g\in G.$

\begin{prop}
Let \(1 \leq p < \infty\) and let \(\alpha: G \to \mathcal{L}(L_p(\mathcal{M}))\) be a Lamperti representation. Then the map 
\[
g \mapsto |\alpha|_g := |\alpha_g|
\]
is again a Lamperti representation.
\end{prop}

\begin{proof}
First, recall that \(\alpha\) is a representation, so for all \(g,h \in G\) and \(x \in L_p(\mathcal{M})\),
\[
\alpha_g(\alpha_h(x)) = \alpha_{gh}(x).
\]
By Proposition~\eqref{moduluslamp}, for every positive element \(x \in L_p(\mathcal{M})_+\), we have
\[
|\alpha_g(\alpha_h(x))| = |\alpha_g|\,|\alpha_h|\,x,
\]
and
\[
|\alpha_{gh}(x)| = |\alpha_{gh}|\,x.
\]
Combining these two equalities gives
\[
|\alpha_g|\,|\alpha_h|\,x = |\alpha_{gh}|\,x, \qquad \forall\, x \in L_p(\mathcal{M})_+.
\]
Since positive elements span \(L_p(\mathcal{M})\), this identity extends to all \(x \in L_p(\mathcal{M})\), and we conclude that
\[
|\alpha_g|\,|\alpha_h| = |\alpha_{gh}|,
\]
showing that \(g \mapsto |\alpha_g|\) is a representation.

It remains to check strong continuity. Since \(\alpha\) is strongly continuous, for any \(x \in L_p(\mathcal{M})\),
\[
\|\alpha_g(x) - x\|_p \to 0 \quad \text{as } g \to e.
\]
In particular, this holds for \(x \in L_p(\mathcal{M})_+\). But the continuity of the map \(x \mapsto |x|\) we get
\[
\big\||\alpha_g(x)| - x\big\|_p \to 0 \quad \text{as } g \to e.
\]
Since \(|\alpha_g(x)| = |\alpha_g|(x)\), we obtain
\[
\||\alpha_g|(x) - x\|_p \to 0, \ \text{for all} x\in L_p(\mathcal{M})_{+}.
\]
Thus \(g \mapsto |\alpha_g|\) is strongly continuous.
This completes the proof of the proposition.
\end{proof}
The following lemma allows us to extend the Lamperti representation \eqref{lampertstrrep} 
to all measurable operators in the finite case.

\begin{lemma}[{\cite{HongRayWang2023}}]
Let $\mathcal{M}$ be a finite von Neumann algebra and $\tau$ a normal faithful tracial state on $\mathcal{M}$. 
Let $1 \leq p < \infty$. 
Let $T \colon L_p(\mathcal{M}) \to L_p(\mathcal{M})$ be a positive Lamperti operator with decomposition
\[
T(x) = b\,J(x), \qquad x \in \mathcal{M}.
\]
Then $J$ and $T$ extend continuously to maps on $L_0(\mathcal{M})$ with respect to the topology of convergence in measure. 
Moreover,
\[
T(x) = b\,J(x), \qquad x \in L_0(\mathcal{M}).
\]
\end{lemma}

\begin{lemma}\label{lamperrepdeo}

Let $\mathcal{M}$ be a finite von Neumann algebra, let $1 \leq p < \infty$, and let 
$\alpha \colon G \to \mathcal{L}(L_p(\mathcal{M}))$ be a positive Lamperti representation 
with decompositions $\alpha_g = b_g J_g$ for $g \in G$. 
Then, for all $g,h \in G$,
\[
b_g J_g(b_h) = b_{gh}
\quad \text{and} \quad
J_g \circ J_h = J_{gh}.
\]
\end{lemma}

\begin{proof}
Since each $\alpha_g$ is invertible, both $b_g$ and $b_h$ have full support $1$ 
(see \cite[Proposition~6.1]{HongRayWang2023}). 
Moreover, the support of $J_g(b_h)$ coincides with that of $b_h$ 
(see the proof of \cite[Proposition~6.1]{HongRayWang2023}). 

For any $x \in L_p(\mathcal{M})$, the identity $\alpha_g(\alpha_h x) = \alpha_{gh} x$ for all $g,h \in G$ 
is equivalent to
\[
b_g J_g(b_h)\, J_g J_h(x) = b_{gh} J_{gh}(x), \qquad g,h \in G,
\]
where we have used parts (2) and (3) of Theorem~\eqref{thm:lamperti-decomposition}. 
Note that $b_g J_g(b_h)$ is a positive element of $L_0(\mathcal{M})$ whose support is $1$, 
since both $b_g$ and $J_g(b_h)$ are positive with full support and they commute. 
Furthermore, $J_g J_h$ is again a normal Jordan $*$-homomorphism. 
Hence the result follows from the uniqueness of the Lamperti decomposition 
(Theorem~\eqref{thm:lamperti-decomposition}). 
\end{proof}

\begin{prop}
Let $\mathcal{M}$ be a finite von Neumann algebra, let $1 \leq p < \infty$, 
and let $\alpha \colon G \to \mathcal{L}(L_p(\mathcal{M}))$ 
be a positive Lamperti representation with decompositions $\alpha_g = b_g J_g$ for $g \in G$. 
Let $\gamma \geq 1$ and set $\mu = \tfrac{p}{\gamma}$. 
Define $\alpha_g^{(\mu)}$ by
\[
\alpha_g^{(\mu)}(x) := b_g^{\mu} J_g(x), \qquad x \in \mathcal{S}_{\mathcal{M}}.
\]
Then the following hold:
\begin{itemize}
\item[(1)] The map $\alpha^{(\mu)} \colon G \to \mathcal{L}(L_\gamma(\mathcal{M}))$, 
given by $\alpha^{(\mu)}(g) := \alpha_g^{(\mu)}$, 
is a Lamperti representation of $G$ on $L_\gamma(\mathcal{M})$.

\item[(2)] 
\[
\sup_{g \in G} \|\alpha_g^{(\mu)}\|_{\gamma \to \gamma} 
= \sup_{g \in G} \|\alpha_g\|_{p \to p}.
\]
\end{itemize}
\end{prop}

\begin{proof}
For $x \in \mathcal{S}_{\mathcal{M}}$, we have
\[
\|\alpha_g^{(\mu)}(x)\|_\gamma^\gamma 
= \tau(b_g^{\mu \gamma} J_g(|x|^\gamma)) 
= \tau(b_g^p J_g(|x|^\gamma))
\leq \|\alpha_g\|_{p \to p}^p \, \tau(|x|^\gamma)
= \|\alpha_g\|_{p \to p}^p \|x\|_\gamma^\gamma.
\]
Here, the first equality follows from \cite[Lemma~4.1]{HongRayWang2023}, 
and the inequality follows from Theorem~\eqref{thm:lamperti-decomposition}. 
Thus each $\alpha_g^{(\mu)}$ is a bounded Lamperti operator and
\[
\sup_{g \in G} \|\alpha_g^{(\mu)}\|_{\gamma \to \gamma} 
< \infty.
\]
Moreover, it is straightforward to verify that
\[
\sup_{g \in G} \|\alpha_g^{(\mu)}\|_{\gamma \to \gamma}
= \sup_{g \in G} \|\alpha_g\|_{p \to p}.
\]

We now show that $g \mapsto \alpha_g^{(\mu)}$ is a group homomorphism.  
For any $x \in L_p(\mathcal{M})$,
\[
\alpha_g^{(\mu)}(\alpha_h^{(\mu)}(x))
= b_g^{\mu} J_g(b_h^{\mu} J_h(x))
= b_g^{\mu} J_g(b_h^{\mu}) J_g J_h(x).
\]
Hence, $\alpha_g^{(\mu)} \alpha_h^{(\mu)} = \alpha_{gh}^{(\mu)}$ 
if and only if 
\[
b_g^{\mu} J_g(b_h^{\mu}) = b_{gh}^{\mu}
\quad \text{and} \quad
J_g J_h = J_{gh},
\]
for all $g,h \in G$.  
Since each $J_g$ is a normal Jordan $*$-homomorphism 
that is continuous with respect to convergence in measure, 
we have
\[
(b_g J_g(b_h))^{\mu} = b_g^{\mu} J_g(b_h^{\mu}).
\]
The result then follows from Lemma~\ref{lamperrepdeo}. The continuity of $g\mapsto\alpha_g^{(\mu)}$ follows from \cite[Theorem 3.2]{Ricard2018} and \cite[Corollary 2.10]{Ricard2021}.
\end{proof}

 Let $\alpha:G\to \mathcal{L}(L_p(\mathcal M))$ be a uniformly bounded Lamperti representation and $G$  be a group of polynomial growth with a symmetric compact generating set $V.$ Define the ergodic averages
\[
A_n(\alpha) x = \frac{1}{m(V^n)} \int_{V^n} \alpha_g(x) \, dm(g), \quad x \in L_p(\mathcal M), \; n \in \mathbb N.
\]
The following theorem generalizes Theorem 3.5 in \cite{Templeman2015}.
\begin{cor}\label{lamprepmainres}Let $\mathcal{M}$ be a finite von Neumann algebra and $G$  be a group of polynomial growth with a symmetric compact generating set $V.$ Fix \(1 < p < \infty\). Let \(\alpha\) be a strongly continuous and uniformly bounded Lamperti representation of \(G\) on \(L_p(\mathcal M)\). Let $1\leq\gamma<\infty$ and $\mu=\frac{p}{\gamma}.$ Then there exists a constant \(C_{p,\gamma} > 0\) such that
\[
\sup_{(n_i)_i} 
\Big\| \big( A_{n_{i+1}}(|\alpha|^{(\mu)})x - A_{n_i}(|\alpha|^{(\mu)})x \big)_i \Big\|_{L_\gamma(\mathcal M; \ell_2^{rc})} 
\le C_{p,\gamma} \, \|x\|_\gamma, 
\quad \forall x \in L_\gamma(\mathcal M).
\]
where the supremum is taken over all increasing subsequences \((n_i)_i \subset \mathbb{N}\).

\end{cor}

\section{Dilations of Families of Operators}\label{dilationfamily}

We begin with the dilation of a single operator and then extend the notion to several operators or families of operators simultaneously.

\begin{defi}\label{jdilations}
Let $\mathcal{T}_1,\ldots,\mathcal{T}_n$ be collections of bounded linear operators on a Banach space $X$.  
We say that $\mathcal{T}_1,\ldots,\mathcal{T}_n$ are \emph{pairwise commuting} if
\[
   TS = ST \qquad \text{for all } T \in \mathcal{T}_i, \; S \in \mathcal{T}_j, \; 1 \leq i \neq j \leq n.
\]
\end{defi}

\begin{defi}
Let $\mathcal{X}$ be a class of Banach spaces and fix $X \in \mathcal{X}$.
\begin{enumerate}
    \item[(a)] Let $\mathbf{T} = (T_1,\ldots,T_n)$ be a commuting $n$-tuple of operators on $X$, i.e. $T_iT_j = T_jT_i$ for all $1 \leq i,j \leq n$.  
    We say that $\mathbf{T}$ \emph{admits a joint dilation in $\mathcal{X}$} if there exist $Y \in \mathcal{X}$, contractive maps
    \[
       J : X \to Y, \qquad Q : Y \to X,
    \]
    and a commuting $n$-tuple of isometries $\mathbf{U} = (U_1,\ldots,U_n) \subseteq \mathcal{L}(Y)$ such that for all multi-indices $(i_1,\ldots,i_n) \in \mathbb{N}_0^n$,
    \[
        T_1^{i_1} \cdots T_n^{i_n} \;=\; Q\, U_1^{i_1} \cdots U_n^{i_n} J.
    \]

    \item[(b)] Let $(\mathcal{T}_1,\ldots,\mathcal{T}_n)$ be pairwise commuting families of operators on $X$.  
    We say that $(\mathcal{T}_1,\ldots,\mathcal{T}_n)$ \emph{admit a joint simultaneous dilation in $\mathcal{X}$} if there exist $Y \in \mathcal{X}$, contractive maps
    \[
       J : X \to Y, \qquad Q : Y \to X,
    \]
    and, for each $1 \leq i \leq n$ and $T \in \mathcal{T}_i$, an isometry $U_T \in \mathcal{L}(Y)$ such that:
    \begin{enumerate}
        \item for every finite sequence $T_1,\ldots,T_m$ with $T_k \in \mathcal{T}_{j_k}$ for some $1 \leq j_k \leq n$, one has
        \[
            T_1 \cdots T_m = Q\, U_{T_1} \cdots U_{T_m} J,
        \]
        \item and $U_T U_S = U_S U_T$ whenever $T \in \mathcal{T}_i$, $S \in \mathcal{T}_j$ with $i \neq j$.
    \end{enumerate}
\end{enumerate}
\end{defi}
%

\begin{defi}\label{def:N-dilation-joint}
Let $\mathcal{X}$ be a class of Banach spaces and fix $X \in \mathcal{X}$.
\begin{enumerate}
    \item[(a)] Let $\mathbf{T} = (T_1,\ldots,T_n)$ be a commuting $n$-tuple of operators on $X$.  
    We say that $\mathbf{T}$ \emph{admits a joint $N$-dilation in $\mathcal{X}$} if there exist $Y \in \mathcal{X}$, contractive maps
    \[
       J : X \to Y, \qquad Q : Y \to X,
    \]
    and a commuting $n$-tuple of isometries $\mathbf{U} = (U_1,\ldots,U_n) \subseteq \mathcal{L}(Y)$ such that for all multi-indices $(i_1,\ldots,i_n) \in \mathbb{N}_0^n$ with $\sum_{k=1}^n i_k \leq N$,
    \[
        T_1^{i_1} \cdots T_n^{i_n} \;=\; Q\, U_1^{i_1} \cdots U_n^{i_n} J.
    \]

    \item[(b)] Let $(\mathcal{T}_1,\ldots,\mathcal{T}_n)$ be pairwise commuting families of operators on $X$.  
    We say that $(\mathcal{T}_1,\ldots,\mathcal{T}_n)$ \emph{admit a joint simultaneous $N$-dilation in $\mathcal{X}$} if there exist $Y \in \mathcal{X}$, contractive maps
    \[
       J : X \to Y, \qquad Q : Y \to X,
    \]
    and, for each $1 \leq i \leq n$ and $T \in \mathcal{T}_i$, an isometry $U_T \in \mathcal{L}(Y)$ such that:
    \begin{enumerate}
        \item for every finite sequence $T_1,\ldots,T_m$ with $0 \leq m \leq N$ and $T_k \in \mathcal{T}_{j_k}$ for some $1 \leq j_k \leq n$, one has
        \[
            T_1 \cdots T_m = Q\, U_{T_1} \cdots U_{T_m} J,
        \]
        \item and $U_T U_S = U_S U_T$ whenever $T \in \mathcal{T}_i$, $S \in \mathcal{T}_j$ with $i \neq j$.
    \end{enumerate}
\end{enumerate}
\end{defi}
\begin{remark}
The case $n=1$ is included in the above definitions \eqref{jdilations} and \eqref{def:N-dilation-joint}. In this situation the commutativity condition is vacuous, 
and one simply says that an operator (or a family of operators) \emph{admits a dilation or an $N$-dilation} or 
\emph{simultaneous dilation or simultaneous $N$-dilation} in $\mathcal{X}$, dropping the word ``joint.'' 

\end{remark}
\begin{remark}
Note that  definitions \eqref{jdilations} and \eqref{def:N-dilation-joint} naturally extends corresponding single variable notions introduced in \cite{FacklerGluck2019}.
\end{remark}
\begin{defi} Let $\mathcal{X}$ be a class of Banach spaces. We say that $\mathcal{X}$ under finite $\ell_p$-sums if For every $X \in \mathcal{X}$ and every $n \in \mathbb{N}$, the finite direct sum $\ell_p^n(X)$ also belongs to $\mathcal{X}$.
\end{defi}
For notational simplicity, we denote the set $\{1,\dots, n\}$ by $[n].$ Define the set of functions
\begin{equation}\label{alphadefn}
\mathcal{A} := \{\alpha:[N]\to [m]\}.
\end{equation}
For $\alpha\in\mathcal{A}$, and $(\lambda_i)_{i=1}^m\in[0,1]^m$ set
\[
\Lambda(\alpha) := \prod_{k=1}^N \lambda_{\alpha(k)}.
\]Let $\sigma:[N]\to[N]$ be the $N$-cycle $(1\dots N).$
The proof of the following lemma can be found in \cite[Proof of Theorem 4.1]{FacklerGluck2019}
\begin{lemma}\label{onevaridila} Let $\mathcal{T}\subseteq\mathcal{L}(X)$ and $T_i\in\mathcal{T}$ for $1\leq i\leq m.$  Then for all $0\leq n\leq N,$ we have the identity
\begin{equation}
T^n \;=\; 
\sum_{\alpha \in A} \frac{\Lambda(\alpha)}{N}
   \sum_{k=1}^{N} 
      \prod_{j=1}^{n} T_{\alpha(\sigma^{k-1}(j))},
\end{equation}
where $(\lambda_i)_{i=1}^m\in[0,1]^m$ and $T=\sum_{i=1}^m\lambda_iT_i.$
\end{lemma}

\begin{proof}[Proof of Theorem \eqref{thm:convex-dilation}]
Since $(\mathcal{T}_1,\ldots,\mathcal{T}_n)$ admits a joint simultaneous dilation, we may assume without loss of generality that each $\mathcal{T}_i$ consists of isometries. 

For each $i$, we can write
\[
T_i=\sum_{j=1}^m\lambda_{i,j}T_{i,j},
\]
where $\lambda_{i,j}\in[0,1]$ and $\sum_{j=1}^m\lambda_{i,j}=1$, with $T_{i,j}\in\mathcal{T}_i$. For each $i\in\{1,\dots,n\}$, define the set of functions
$\mathcal{A}_i :=\mathcal A,$ where $mathcal{A}$ is as in \eqref{alphadefn}.
For $\alpha_i\in\mathcal{A}_i$, set
\[
\Lambda(\alpha_i) := \prod_{k=1}^N \lambda_{i,\alpha_i(k)}.
\]
Then one can check that
\begin{equation}\label{eq:lambda}
\sum_{\alpha_i\in\mathcal{A}_i}\Lambda(\alpha_i)=1.
\end{equation}
Let us define
\[
Y := \ell_p^{N^n m^{nN}}(X).
\]
Since $\mathcal{X}$ is stable under finite $\ell_p$-sums, we have $Y \in \mathcal{X}$.

Define $J:X\to Y$ by
\[
(Jx)_{(\alpha_1,\dots,\alpha_n),(i_1,\dots,i_n)}
=\Bigg(\frac{\prod_{i=1}^n \Lambda(\alpha_i)}{N^n}\Bigg)^{1/p} x,
\]
where $(\alpha_1,\dots,\alpha_n)\in\mathcal{A}^n$ and $(i_1,\dots,i_n)\in[N]^n$.  

We check that $J$ is an isometry:
\[
\|Jx\|_p^p
= \sum_{(\alpha_1,\dots,\alpha_n)} \sum_{(i_1,\dots,i_n)}
\frac{\prod_{i=1}^n \Lambda(\alpha_i)}{N^n}\,\|x\|_X^p.
\]
The inner sum over $(i_1,\dots,i_n)$ has $N^n$ terms, hence cancels the denominator. Using \eqref{eq:lambda},
\[
\|Jx\|_p^p
= \Bigg(\prod_{i=1}^n \sum_{\alpha_i\in\mathcal{A}_i}\Lambda(\alpha_i)\Bigg)\|x\|_X^p
= \|x\|_X^p.
\]
Thus $J$ is an isometry.
Suppose $q \in (1,\infty)$ satisfies $1/p + 1/q = 1$. Define the operator 
$Q:Y \to X$ by 
\[
Qy := \sum_{(\alpha_1,\dots,\alpha_n)} \sum_{(i_1,\dots,i_n)} 
\left(\frac{\prod_{j=1}^n \Lambda(\alpha_j)}{N^n}\right)^{1/q}
\, y_{(\alpha_1,\dots,\alpha_n),(i_1,\dots,i_n)},
\]
for 
\[
y = \Big(y_{(\alpha_1,\dots,\alpha_n),(i_1,\dots,i_n)}\Big)_{(\alpha_1,\dots,\alpha_n)\in \mathcal{A}^n,\,(i_1,\dots,i_n)\in[N]^n}.
\]
We show now that $Q$ is a contraction.

By the triangle inequality, we have
\[
\|Qy\|
 \leq \sum_{(\alpha_1,\dots,\alpha_n)} \sum_{(i_1,\dots,i_n)} 
\left(\frac{\prod_{j=1}^n \Lambda(\alpha_j)}{N^n}\right)^{1/q}
\big\|y_{(\alpha_1,\dots,\alpha_n),(i_1,\dots,i_n)}\big\|.
\]
Applying Hölder’s inequality with exponents $p$ and $q$ yields
\[
\|Qy\|
\leq \Bigg(\sum_{(\alpha_1,\dots,\alpha_n)} \sum_{(i_1,\dots,i_n)}
\frac{\prod_{j=1}^n \Lambda(\alpha_j)}{N^n}\Bigg)^{1/q}
\Bigg(\sum_{(\alpha_1,\dots,\alpha_n)} \sum_{(i_1,\dots,i_n)}
\big\|y_{(\alpha_1,\dots,\alpha_n),(i_1,\dots,i_n)}\big\|^p\Bigg)^{1/p}.
\]

Now observe that by \eqref{eq:lambda}
\[
\sum_{(\alpha_1,\dots,\alpha_n)} \sum_{(i_1,\dots,i_n)}
\frac{\prod_{j=1}^n \Lambda(\alpha_j)}{N^n}
= \sum_{(\alpha_1,\dots,\alpha_n)} \frac{\prod_{j=1}^n \Lambda(\alpha_j)}{N^n} \cdot N^n
= \sum_{(\alpha_1,\dots,\alpha_n)} \prod_{j=1}^n \Lambda(\alpha_j) = 1.
\]
Therefore, we obtain that
\[
\|Qy\| \leq \Bigg(\sum_{(\alpha_1,\dots,\alpha_n)} \sum_{(i_1,\dots,i_n)}
\big\|y_{(\alpha_1,\dots,\alpha_n),(i_1,\dots,i_n)}\big\|^p\Bigg)^{1/p}
= \|y\|_Y.
\]
Thus $Q$ is a contraction.

For $1 \leq r \leq n$, define $U_r : Y \to Y$ by
\begin{equation}\label{U_rformulae}
\big(U_r x\big)_{(\alpha_1,\dots,\alpha_n),(i_1,\dots,i_n)}
:= T_{r,\alpha_r(i_r)} \,
x_{(\alpha_1,\dots,\alpha_n),(i_1,\dots,\sigma(i_r),\dots,i_n)},
\end{equation}
where 
\[
x = \Big(x_{(\alpha_1,\dots,\alpha_n),(i_1,\dots,i_n)}\Big)_{(\alpha_1,\dots,\alpha_n)\in\mathcal{A}_1\times\cdots\times\mathcal{A}_n,\;\,(i_1,\dots,i_n)\in\{1,\dots,N\}^n}\in Y,
\]
and $\sigma:\{1,\dots,N\}\to\{1,\dots,N\}$ denotes the $N$-cycle $
\sigma = (1\;2\;\cdots\;N).$

Fix $r\in\{1,\dots,n\}$. Linearity of $U_r$ is immediate from the linearity of each $T_{r,j}$ and the definition of $U_r$ as in \eqref{U_rformulae}. We check that $U_r$ is an isometry. Recall
\[
\|x\|_Y^p \;=\;
\sum_{(\alpha_1,\dots,\alpha_n)} \sum_{(i_1,\dots,i_n)}
\big\|x_{(\alpha_1,\dots,\alpha_n),(i_1,\dots,i_n)}\big\|_X^p,
\]
and that each $T_{r,j}$ is an isometry on $X$ by hypothesis for all $1\leq j\leq m$. Hence
\begin{align*}
\|U_r x\|_Y^p
&= \sum_{(\alpha_1,\dots,\alpha_n)} \sum_{(i_1,\dots,i_n)}
\big\| \big(U_r x\big)_{(\alpha_1,\dots,\alpha_n),(i_1,\dots,i_n)} \big\|_X^p \\
&= \sum_{(\alpha_1,\dots,\alpha_n)} \sum_{(i_1,\dots,i_n)}
\big\| T_{r,\alpha_r(i_r)} \,
x_{(\alpha_1,\dots,\alpha_n),(i_1,\dots,\sigma(i_r),\dots,i_n)} \big\|_X^p \\
&= \sum_{(\alpha_1,\dots,\alpha_n)} \sum_{(i_1,\dots,i_n)}
\big\| x_{(\alpha_1,\dots,\alpha_n),(i_1,\dots,\sigma(i_r),\dots,i_n)} \big\|_X^p,
\end{align*}
where we used that $T_{r,\alpha_r(i_r)}$ is an isometry to drop it inside the norm. Now observe that the map
\[
(\alpha_1,\dots,\alpha_n,\; i_1,\dots,i_n)\longmapsto
(\alpha_1,\dots,\alpha_n,\; i_1,\dots,\sigma(i_r),\dots,i_n)
\]
is a bijection of the finite index set $\mathcal{A}^n\times[N]^n$ as it is the identity on all coordinates except the $r$-th $i$-coordinate where it applies the permutation $\sigma$. Therefore the double sum above is equal to
\[
\sum_{(\alpha_1,\dots,\alpha_n)} \sum_{(i_1,\dots,i_n)}
\big\| x_{(\alpha_1,\dots,\alpha_n),(i_1,\dots,i_n)} \big\|_X^p
= \|x\|_Y^p.
\]
Thus $\|U_r x\|_Y=\|x\|_Y$ for all $x\in Y$, so $U_r$ is an isometry.

We now prove that $U_r$ and $U_s$ commute for $r\neq s$. Fix \(r\neq s\) and an index
\(((\alpha_1,\dots,\alpha_n),(i_1,\dots,i_n))\). We compute the corresponding coordinate of \(U_rU_s x\).
First,
\[
(U_s x)_{(\alpha_1,\dots,\alpha_n),(i_1,\dots,i_n)}
= T_{s,\alpha_s(i_s)} \,
x_{(\alpha_1,\dots,\alpha_n),(i_1,\dots,\sigma(i_s),\dots,i_n)}.
\]
Therefore
\[
\begin{aligned}
\big(U_r(U_s x)\big)_{(\alpha_1,\dots,\alpha_n),(i_1,\dots,i_n)}
&= T_{r,\alpha_r(i_r)} \big( U_s x \big)_{(\alpha_1,\dots,\alpha_n),(i_1,\dots,\sigma(i_r),\dots,i_n)}\\
&=  \big( T_{r,\alpha_r(i_r)}\, T_{s,\alpha_s(i_s)}\big)\,
x_{(\alpha_1,\dots,\alpha_n),(i_1,\dots,\sigma(i_r),\dots,\sigma(i_s),\dots,i_n)}.
\end{aligned}
\]
Interchanging the roles of $r$ and $s$ gives
\[
\big(U_s(U_r x)\big)_{(\alpha_1,\dots,\alpha_n),(i_1,\dots,i_n)}
= \big( T_{s,\alpha_s(i_s)} T_{r,\alpha_r(i_r)} \big)\,
x_{(\alpha_1,\dots,\alpha_n),(i_1,\dots,\sigma(i_r),\dots,\sigma(i_s),\dots,i_n)}.
\]
By the hypothesis that the families $\mathcal{T}_1,\dots,\mathcal{T}_n$ are pairwise commuting, every operator from $\mathcal{T}_r$ commutes with every operator from $\mathcal{T}_s$. In particular,
\[
T_{r,\alpha_r(i_r)} T_{s,\alpha_s(\sigma(i_r))}
\;=\;
T_{s,\alpha_s(\sigma(i_r))} T_{r,\alpha_r(i_r)}.
\]
 Thus the two coordinate values
\[
\big(U_r(U_s x)\big)_{(\alpha_1,\dots,\alpha_n),(i_1,\dots,i_n)}
\quad\text{and}\quad
\big(U_s(U_r x)\big)_{(\alpha_1,\dots,\alpha_n),(i_1,\dots,i_n)}
\]
are equal for every index. Since this holds for every index, we conclude $U_rU_s x=U_sU_r x$ for all $x\in Y$, i.e. $U_rU_s=U_sU_r$.

Next we prove that for any $0\leq j_1,\dots,j_n\leq N$ such that $j_1+\dots+j_n=N,$ we have the identity 
\begin{equation}
T_1^{j_1}\dots T_n^{j_n}x=QU_1^{j_1}\dots U_n^{j_n}Jx, \text{for all}\ x\in X.
\end{equation}
To describe the action of the dilating operators, let 
\[
x = \bigl(x_{(\alpha_1,\dots,\alpha_n),\, (i_1,\dots,i_n)}\bigr).
\]  
Then, for a fixed $r\in\{1,\dots,n\}$, the iterates of $U_r$ act as  
\begin{equation} \label{eq:Ur-action}
\bigl(U_r^{\,j_r}x\bigr)_{(\alpha_1,\dots,\alpha_n),\, (i_1,\dots,i_n)}
   = \Biggl(\prod_{l_r=1}^{j_r} 
        T_{r,\,\alpha_r\!\bigl(\sigma^{\,i_r-1}(l_r)\bigr)}
     \Biggr) 
     \,x_{(\alpha_1,\dots,\alpha_n),\, (i_1,\dots,\sigma^{j_r}(i_r),\dots,i_n)}.
\end{equation}
In above we have used the fact that $\sigma^{k-1}(j)=\sigma^{j-1}(k)$ for all $j,k\in\{1,\dots,N\}.$ \begin{align}
QU_1^{j_1}\cdots U_n^{j_n}Jx
  &\nonumber = \sum_{\alpha \in \mathcal{A}^n}\;\sum_{i \in [N]^n}
   \left(\frac{\prod_{j=1}^n \Lambda(\alpha_j)}{N^n}\right)^{1/q}
   \left(\frac{\prod_{j=1}^n \Lambda(\alpha_j)}{N^n}\right)^{1/p}
   \Biggl(\prod_{r=1}^n \prod_{l_r=1}^{j_r} 
      T_{r,\,\alpha_r\!\bigl(\sigma^{\,i_r-1}(l_r)\bigr)}\Biggr)x\\\nonumber
      &=\sum_{\alpha \in \mathcal{A}^n}\;\sum_{i \in [N]^n}
   \left(\frac{\prod_{j=1}^n \Lambda(\alpha_j)}{N^n}\right)
   \Biggl(\prod_{r=1}^n \prod_{l_r=1}^{j_r} 
      T_{r,\,\alpha_r\!\bigl(\sigma^{\,i_r-1}(l_r)\bigr)}\Biggr)x\\\nonumber
      &=\prod_{r=1}^n\Big(\sum_{\alpha_r\in\mathcal A}\sum_{i_r\in[N]}\frac{ \Lambda(\alpha_r)}{N} \prod_{l=1}^{j_r} 
      T_{r,\,\alpha_r\!\bigl(\sigma^{\,i_r-1}(l)\bigr)}\Big)x\\ \label{eqn7id}
      &=T_1^{j_1}\dots T_n^{j_n}x.
\end{align}
In above we have used Lemma \eqref{onevaridila} to obtain \eqref{eqn7id}.
\end{proof}

\begin{cor}\label{cmunidilasur}
Let $p \in (1,\infty)$ and  and $(\mathcal{U}_1,\ldots,\mathcal{U}_n)$ be a tuple of pairwise commuting families of isometries on $L_p(\mathcal M)$ consisting of isomteries. Then, every tuple of operators  $\mathbf{T}=(T_1,\dots,T_n)$ where $T_i$ belonging to the convex hull of $\mathcal{U}_i,$ where $1\leq i\leq n,$ also admits a joint $N$-dilation for all $N\in\mathbb{N}$ on a bigger noncommutative $L_p$-space.
\end{cor}
\begin{remark} As in \cite{FacklerGluck2019}, one can use an ultraproduct argument to obtain joint dilations for SOT-limits of such commuting tuples; this will be addressed elsewhere.
\end{remark}
\section{Quantitative ergodic theorem for semigroup of operators}\label{quanergforopsemi}
Let $1\leq p<\infty.$ Let us denote $\mathbb{R}^d_{+}:=\{(t_1,\dots,t_d):t_i\geq 0,\ 1\leq i\leq d\}.$ Let $\alpha:\mathbb{R}^d_{+}\to \mathcal{L}(L_p(\mathcal M))$ be a strongly continuous uniformly bounded semigroup of operators. Define for any $t>0,$ \[
A_{t}(\alpha) x: = \frac{1}{t^d} \int_{0}^{t}\dots\int_{0}^{t} \alpha_{(s_1,\dots,s_d)}(x) \, ds_1\dots ds_d, \quad x \in L_p(\mathcal M).\] For any commuting tuple of operators $\mathbf{T}=(T_1,\dots,T_d)$ in $\mathcal{L}(L_p(\mathcal M))$ we define \[A_{n}(\mathbf{T})x:=\frac{1}{n^d} \sum_{j_1=0}^{n-1}\dots\sum_{j_d=0}^{n-1} T_1^{j_1}\dots T_d^{j_d}(x), \quad x \in L_p(\mathcal M).\] 
\begin{prop}\label{discretetocont} Let $\alpha:\mathbb{R}^d_{+}\to \mathcal{L}(L_p(\mathcal M))$ be a strongly continuous uniformly bounded semigroup of operators such that there exists a positive constant $C>0$ so that for all nonnegative integers $\mathbf{k}:=(k_1,\dots, k_d)$
\begin{equation}\label{contave}
\sup_{(n_i)_i} 
\Big\| \big( A_{n_{i+1}}(\mathbf{T}_{\frac{1}{\mathbf{k}}})x - A_{n_i}(\mathbf{T}_{\frac{1}{\mathbf{k}}})x \big)_i \Big\|_{L_p(\mathcal M; \ell_2^{rc})} 
\le C \, \|x\|_p, 
\quad \forall x \in L_p(\mathcal M).
\end{equation}
where the supremum is being taken over all increasing sequences of positive integers \((n_i)_{i \in \mathbb{N}} \subset (0, \infty)\)  and
$\mathbf{T}_{\frac{1}{\mathbf{k}}}:=(\alpha_{(\frac{1}{k},\dots,0)},\dots,\alpha_{(0,\dots,\frac{1}{k}})).$ Then 
\begin{equation}\label{discretization}
\sup_{(t_i)_i} 
\Big\| \big( A_{t_{i+1}}(\alpha)x - A_{t_{i}}(\alpha)x \big)_i \Big\|_{L_p(\mathcal M; \ell_2^{rc})} 
\le C \, \|x\|_p, 
\quad \forall x \in L_p(\mathcal M).
\end{equation}
where the supremum being taken over all increasing sequences \((t_i)_{i \in \mathbb{N}} \subset (0, \infty)\)
\end{prop}
\begin{proof}
Let \( 0 < t_1 < \dots < t_m \) be a finite increasing sequence.  
It suffices to prove \eqref{discretization} with a constant independent of the choice of the sequence. For any \( t > 0 \) and \( k \in \mathbb{N} \), set $
n_k(t) := \lfloor k t \rfloor,$ where \( \lfloor s \rfloor \) denotes the greatest integer not exceeding \( s \).

We first prove the following approximation.
For every \( t > 0 \) and \( x \in L_p(\mathcal{M}) \), we claim that
\begin{equation}\label{eq:approx}
\big\| A_{n_k(t)}(\mathbf{T}_{1/\mathbf{k}})x - A_t(\alpha)x \big\|_p \longrightarrow 0
\quad \text{as } k \to \infty.
\end{equation}
Indeed, by the strong continuity and uniform boundedness of \( \alpha \),  
the Riemann integral
\[
\int_0^t \!\!\cdots\! \int_0^t 
\alpha_{(s_1,\dots,s_d)}(x)\, ds_1 \cdots ds_d
\]
is approximated by the Riemann sums
\[
\frac{1}{k^d} 
\sum_{j_1=0}^{n_k(t)-1}\!\cdots\!\sum_{j_d=0}^{n_k(t)-1}
\alpha_{\left(\frac{j_1}{k}, \dots, \frac{j_d}{k}\right)}(x)
= A_{n_k(t)}(\mathbf{T}_{1/\mathbf{k}})x.
\]
Since \( |n_k(t)/k - t| \to 0 \) as \( k \to \infty \), the convergence in \eqref{eq:approx} follows.

To this end, we observe by using the triangle inequality, we have
\begin{align*}
&\Big\| \big( A_{t_{i+1}}(\alpha)x - A_{t_i}(\alpha)x \big)_{i=0}^{m-1} \Big\|_{L_p(\mathcal M; \ell_2^{rc})} \\
&\quad\le
\Big\| 
\big(
A_{t_{i+1}}(\alpha)x - A_{t_i}(\alpha)x
- \big(
A_{n_k(t_{i+1})}(\mathbf{T}_{1/\mathbf{k}})x 
- A_{n_k(t_i)}(\mathbf{T}_{1/\mathbf{k}})x
\big)
\big)_{i=0}^{m-1}
\Big\|_{L_p(\mathcal M; \ell_2^{rc})} \\
&\qquad
+ \Big\| 
\big(
A_{n_k(t_{i+1})}(\mathbf{T}_{1/\mathbf{k}})x 
- A_{n_k(t_i)}(\mathbf{T}_{1/\mathbf{k}})x
\big)_{i=0}^{m-1}
\Big\|_{L_p(\mathcal M; \ell_2^{rc})}.
\end{align*}
By the discrete assumption of the proposition, i.e. \eqref{discretization} the second term is bounded by \( C\|x\|_p \).  
Hence,
\begin{align*}
&\Big\| \big( A_{t_{i+1}}(\alpha)x - A_{t_i}(\alpha)x \big)_{i=0}^{m-1} \Big\|_{L_p(\mathcal M; \ell_2^{rc})} \\
&\quad\le
\Big\|
\big(
A_{t_{i+1}}(\alpha)x - A_{t_i}(\alpha)x
- \big(
A_{n_k(t_{i+1})}(\mathbf{T}_{1/\mathbf{k}})x 
- A_{n_k(t_i)}(\mathbf{T}_{1/\mathbf{k}})x
\big)
\big)_{i=0}^{m-1}
\Big\|_{L_p(\mathcal M; \ell_2^{rc})}
+ C\|x\|_p.
\end{align*}

We now handle the error term.
By the noncommutative Khintchine inequality and the convergence \eqref{eq:approx}, we obtain
\begin{align*}
&\Big\| 
\big(
A_{t_{i+1}}(\alpha)x - A_{t_i}(\alpha)x
- \big(
A_{n_k(t_{i+1})}(\mathbf{T}_{1/\mathbf{k}})x 
- A_{n_k(t_i)}(\mathbf{T}_{1/\mathbf{k}})x
\big)
\big)_{i=0}^{m-1}
\Big\|_{L_p(\mathcal M; \ell_2^{rc})} \\
&\quad\lesssim
\Big\|
\sum_i \varepsilon_i \otimes 
\Big(
A_{t_{i+1}}(\alpha)x - A_{t_i}(\alpha)x
- \big(
A_{n_k(t_{i+1})}(\mathbf{T}_{1/\mathbf{k}})x 
- A_{n_k(t_i)}(\mathbf{T}_{1/\mathbf{k}})x
\big)
\Big)
\Big\|_{L_p(L_\infty(\Omega)\,\overline{\otimes}\,\mathcal{M})} \\
&\quad\le
\sum_i 
\Big\|
A_{t_{i+1}}(\alpha)x - A_{t_i}(\alpha)x
- \big(
A_{n_k(t_{i+1})}(\mathbf{T}_{1/\mathbf{k}})x 
- A_{n_k(t_i)}(\mathbf{T}_{1/\mathbf{k}})x
\big)
\Big\|_p 
\;\longrightarrow\; 0,
\quad \text{as } k \to \infty.
\end{align*}

\smallskip
\noindent
Combining the above estimates and letting \( k \to \infty \), we obtain
\[
\Big\| \big( A_{t_{i+1}}(\alpha)x - A_{t_i}(\alpha)x \big)_{i=0}^{m-1} 
\Big\|_{L_p(\mathcal M; \ell_2^{rc})} 
\le C\|x\|_p.
\]
Since the constant \( C \) is independent of the finite sequence \( (t_i)_{i=0}^m \), the desired supremum bound in \eqref{contave} follows.  
This completes the proof.
\end{proof}
The proof of the following lemma can be proved easily as in \cite[Lemma 7.3]{HLX}.
\begin{lemma}\label{isometrytup}
Let \(1 < p < \infty\) and let $\mathbf{U}=(U_1,\dots,U_n)$ be a commuting tuple of isometries on \(L_p(\mathcal M)\). 
Let \((n_i)_{i \in \mathbb N}\) be any increasing sequence of positive integers. 
Then 
\[
\Big\| \big( A_{n_i}(\mathbf{U})x - A_{n_{i+1}}(\mathbf{U})x \big)_{i \in \mathbb N} \Big\|_{L_p(\mathcal M; \ell_2^{rc})} 
\le C_p \, \|x\|_{L_p(\mathcal M)}, 
\quad \forall x \in L_p(\mathcal M).
\]
\end{lemma}
Combining above lemma with similar computtaions as in \cite[Proof of Theorem 1.4]{HLX} we obtain the following.
\begin{cor}
Let $1<p<\infty$ and $\mathbf{S}:=(S_1,\dots,S_n)$ be strong limits of operators $\mathbf{T}$ as in Corollary \eqref{cmunidilasur} then 
\[
\Big\| \big( A_{n_i}(\mathbf{S})x - A_{n_{i+1}}(\mathbf{S})x \big)_{i \in \mathbb N} \Big\|_{L_p(\mathcal M; \ell_2^{rc})} 
\le C_p \, \|x\|_{L_p(\mathcal M)}, 
\quad \forall x \in L_p(\mathcal M),
\]
where \((n_i)_{i \in \mathbb N}\) is any increasing sequence of positive integers. 
\end{cor}
As an application of the Lemma \eqref{isometrytup}, \eqref{cmunidilasur} and the above corollary we can prove a large class of semigroup satisfies quantitative ergodic theorem by using  Proposition \eqref{discretetocont}. However, we state the following general fact, which is off course valid for single variable semigroup of Lamperti contractions or multivaribale semigroups of contractions coming from Corollary \eqref{isometrytup}. 
\begin{cor}\label{semiofR6d} Let $1<p\neq 2<\infty$ and $\alpha:\mathbb{R}^d_{+}\to \mathcal{L}(L_p(\mathcal M))$ be a strongly continuous semigroup of contractions such that for all $k\in\mathbb N,$ the commuting tuple of contractions $\mathbf{T}_{\frac{1}{\mathbf{k}}}:=(\alpha_{(\frac{1}{k},\dots,0)},\dots,\alpha_{(0,\dots,\frac{1}{k}}))$ admits a joint dilation to some other noncommutative $L_p$-space, then 
\begin{equation}\label{lampertisemigroup}
\sup_{(t_i)_i} 
\Big\| \big( A_{t_{i+1}}(\alpha)x - A_{t_{i}}(\alpha)x \big)_i \Big\|_{L_p(\mathcal M; \ell_2^{rc})} 
\le C \, \|x\|_p, 
\quad \forall x \in L_p(\mathcal M).
\end{equation}
where the supremum being taken over all increasing sequences \((t_i)_{i \in \mathbb{N}} \subset (0, \infty)\)

\end{cor}
\textbf{Acknowledgement:} The first author is partially supported by National Natural Science
Foundation of China (No. 12071355, No. 12325105, No. 12031004, No. W2441002). The third
named author thanks the DST-INSPIRE Faculty Fellowship DST/INSPIRE/04/2020/001132, Prime Minister Early Career Research Grant Scheme ANRF/ECRG/2024/000699/PMS and ANRF/ARGM/2025/000895/MTR.

\end{document}